\documentclass[a4paper,10pt]{amsart}
\newcommand{\thelanguage}{english}
\usepackage[\thelanguage]{babel}

\usepackage{geometry}
\usepackage{mathrsfs}
\usepackage{tkz-euclide, tikz, pgfplots}
\usetikzlibrary{backgrounds,patterns,matrix,calc,arrows,positioning,fit,shapes.geometric,decorations.pathmorphing}
\usepackage{tikz-cd}
\usepackage{float}
\usepackage{amsthm,amsmath,amssymb,amsfonts}
\usepackage[colorlinks=true, citecolor=blue]{hyperref}
\usepackage{multirow}
\usepackage{enumitem}
\usepackage{tabularx}
\usepackage{booktabs}
\usepackage[sort]{cite}
\usepackage[\thelanguage]{cleveref}
\usepackage{marginnote}
\usepackage{etex}
\usepackage{subfig}

\DeclareMathAlphabet{\mathpzc}{OT1}{pzc}{m}{it}
\newcommand{\pzc}[1]{\mathpzc{#1}}

\newcommand{\fr}[1]{\mathfrak{#1}}
\newcommand{\cal}[1]{\mathcal{#1}}

\newcommand{\vect}[1]{\mathbf{#1}}

\renewcommand{\d}{\mathrm{d}}

\DeclareMathOperator*{\mean}{\mathbb{E}}

\newcommand{\tuple}[1]{\pzc{#1}}

\newcommand{\Var}[1]{\mathcal{#1}}

\newcommand{\tensor}[1]{\pzc{#1}}
\newcommand{\sten}[3]{\vect{#1}_{#2}^{#3}}
\newcommand{\Tang}[2]{\mathrm{T}_{#1} {\, #2}}

\newcommand{\R}{\mathbb{R}}
\newcommand{\deriv}[2]{\d{}_{#2} #1}

\newcommand{\Sp}{\mathbb{S}}
\newcommand{\Id}{\mathrm{I}}
\newcommand{\Jac}[2]{\mathrm{Jac}\,(#1)(#2)}

\numberwithin{equation}{section}
\numberwithin{figure}{section}
\numberwithin{table}{section}

\theoremstyle{plain}
\newcounter{numbering} \numberwithin{numbering}{section}
\newtheorem{thm}[numbering]{Theorem}
\newtheorem{lemma}[numbering]{Lemma}
\newtheorem{prop}[numbering]{Proposition}
\newtheorem{cor}[numbering]{Corollary}
\theoremstyle{definition}
\newtheorem{assumption}{Assumption}
\newtheorem{conj}[numbering]{Conjecture}
\newtheorem{dfn}[numbering]{Definition}

\theoremstyle{remark}

\newtheorem{rem}[numbering]{Remark}

\crefname{equation}{}{}
\crefname{equation}{}{}
\crefname{figure}{Figure}{Figures}
\crefname{section}{Section}{Sections}
\crefname{table}{Table}{Tables}
\crefname{lemma}{Lemma}{Lemmata}
\crefname{appendix}{Appendix}{Appendix}
\crefname{prop}{Proposition}{Propositions}
\crefname{thm}{Theorem}{Theorems}
\crefname{cor}{Corollary}{Corollaries}
\crefname{dfn}{Definition}{Definitions}
\crefname{hyp}{Hypothesis}{Hypotheses}
\crefname{notation}{Notations}{Notations}
\crefname{rem}{Remark}{Remarks}
\crefname{claim}{Claim}{claims}
\crefname{assumption}{Assumption}{Assumptions}
\newcommand{\C}{\mathbb{C}}
\newcommand{\F}{\mathbb{F}}

\newcommand{\Kang}{\kappa_{\mathrm{ang}}}

\allowdisplaybreaks

\title[The average condition number is infinite in most cases]{The average condition number of most tensor rank decomposition problems is infinite}
\author{Carlos Beltr\'an}
\thanks{CB: Universidad de Cantabria, beltranc@unican.es. Supported by Spanish
	``Ministerio de Econom\'ia y Competitividad'' under projects MTM2017-83816-P
	and MTM2017-90682-REDT (Red ALAMA), as well as by the Banco Santander and Universidad de Cantabria under
	project 21.SI01.64658.}
\author{Paul Breiding}
\thanks{PB: Universit\"at Osnabr\"uck, pbreiding@uni-osnabrueck.de.}
\author{Nick Vannieuwenhoven}
\thanks{NV: KU Leuven, Department of Computer Science, nick.vannieuwenhoven@kuleuven.be. Supported by the Postdoctoral Fellowship of the Research Foundation--Flanders (FWO) with project numbers 12E8116N and 12E8119N}

\begin{document}

\maketitle

\begin{abstract}
	The tensor rank decomposition, or canonical polyadic decomposition, is the decomposition of a tensor into a sum of rank-1 tensors.
	The condition number of the tensor rank decomposition measures the sensitivity of the rank-1 summands with respect to structured perturbations. Those are perturbations preserving the rank of the tensor that is decomposed. On the other hand, the angular condition number measures the perturbations of the rank-1 summands up to scaling.
	
	We show for random rank-2 tensors that the expected value of the condition number is infinite for a wide range of choices of the density. Under a mild additional assumption, we show that the same is true for most higher ranks $r\geq 3$ as well. In fact, as the dimensions of the tensor tend to infinity, asymptotically all ranks are covered by our analysis. On the contrary, we show that rank-2 tensors have finite expected angular condition number. Based on numerical experiments, we conjecture that this could also be true for higher ranks.
	
	Our results underline the high computational complexity of computing tensor rank decompositions. We discuss consequences of our results for algorithm design and for testing algorithms computing tensor rank decompositions.
	\end{abstract}

	\section{Introduction}
	
	\subsection{The condition number of tensor rank decomposition}
	In this article, a \emph{tensor} is a multidimensional array filled with numbers:
		$$\tensor{A} := (a_{i_1,\ldots,i_d})_{1\leq i_1\leq n_1,\ldots,1\leq i_d\leq n_d} \in \R^{n_1 \times \cdots \times n_d}.$$
	The integer $d$ is called the \emph{order} of $\tensor{A}$. The tensor product of $d$ vectors $\vect{u}^{1} \in \R^{n_1},\ldots, \vect{u}^{d} \in \R^{n_d}$ is defined to be the tensor
	$\vect{u}^{1} \otimes \cdots \otimes \vect{u}^{d} \in \R^{n_1 \times \cdots \times n_d}$ with entries
	\[
	 (\vect{u}^{1}  \otimes \cdots \otimes \vect{u}^{d})_{i_1,\ldots,i_d} := u_{i_1}^{(1)} \cdots u_{i_d}^{(d)}, \text{ where } \vect{u}^{j} = [u_i^{(j)}]_{1\leq i\leq n_j}.
	\]
	Any nonzero multidimensional array obeying this relation is called a \emph{rank-1 tensor}. Not every multidimensional array represents a rank-$1$ tensor, but every tensor~$\tensor{A}$ is a finite linear combination of rank-$1$ tensors:
	\begin{align}\label{eqn_CPD0}
	 \tensor{A} =  \sum_{i=1}^r \tensor{A}_i, \text{ where } \tensor{A}_i = \sten{u}{i}{1} \otimes \cdots \otimes \sten{u}{i}{d} \text{ has rank one for each } 1\leq i\leq d.
	\end{align}
	Hitchcock \cite{hitchcock} coined the name \emph{polyadic decomposition} for the decomposition~\cref{eqn_CPD0}. The smallest number $r$ for which $\tensor{A}$ admits an expression as in \cref{eqn_CPD0} is called the (real) \emph{rank} of $\tensor{A}$. A corresponding minimal decomposition is called a \emph{canonical polyadic decomposition} (CPD).
	
	For instance, in algebraic statistics \cite{AMR2009,M1987}, chemical sciences \cite{SBG2004}, machine learning \cite{AGHKT2014}, psychometrics \cite{Kroonenberg2008}, signal processing \cite{Comon1994,CJ2010,Review2016}, or theoretical computer science \cite{BCS1997}, the input data has the structure of a tensor and the CPD of this tensor reveals the information of interest.
	Usually, this data is subject to measurement errors, which will cause the CPD computed from the measured data to differ from the CPD of the true data. In numerical analysis, the sensitivity of the model parameters, such as the rank-$1$ summands in the CPD, to perturbations of the data is often quantified by the \emph{condition number} \cite{Rice1966}.
	
	When there are multiple CPDs of a tensor $\tensor{A}$, the condition number must be defined at a decomposition $\{\tensor{A}_1, \ldots, \tensor{A}_r\}$. However, in this article, we will restrict our analysis to tensors $\tensor{A}$ having a \emph{unique} decomposition. Such tensors are called \emph{identifiable}. In this case, the condition number of the tensor rank decomposition of a tensor $\tensor{A}$ is well-defined, and we denote it by~$\kappa(\tensor{A})$. We will explain in \cref{sec:identifiablity} below in greater detail the notion of identifiablity of tensors. At this point, the reader should mainly bear in mind that the assumption of being identifiable is comparably weak as most tensors of low rank satisfy it. However, note that matrices ($d=2$) are never identifiable, so we assume that the order of the tensor is $d\geq 3$.

	The condition number of tensor rank decomposition was {characterized} in \cite{BV2017}, and it is the condition number of the following computational problem: On input $\tensor{A}\in \R^{n_1\times \cdots \times n_d}$ of rank $r$, compute the set of rank-1 terms $\{\tensor{A}_1, \ldots, \tensor{A}_r\}$ in the decomposition \cref{eqn_CPD0}. This condition number measures the sensitivity of the rank-$1$ terms
	with respect to perturbations of the tensor $\tensor{A}$.
	{In other words, when the condition number $\kappa(\tensor{A})$ of the rank-$r$ identifiable tensor $\tensor{A} = \sum_{i=1}^r \tensor{A}_i$ in~\cref{eqn_CPD0} is finite, it is the smallest value $\kappa(\tensor{A})$ such that
	\begin{align} \label{eqn_error_bound1}
	\min_{\pi \in \mathfrak{S}_r} \sqrt{ \sum_{i=1}^r \| \tensor{A}_i - \tensor{A}_{\pi_i}'  \|^2 } \le \kappa(\tensor{A}) \| \tensor{A} - \tensor{A}' \| + o( \| \tensor{A} - \tensor{A}' \| )
	\end{align}
	holds for all rank-$r$ tensors $\tensor{A}' = \sum_{i=1}^r \tensor{A}_i'$ (with $\tensor{A}_i'$ of rank $1$) sufficiently close to $\tensor{A}$. Herein, the norm on $\R^{n_1\times\cdots\times n_d}$ is the usual Euclidean norm, and $\mathfrak{S}_r$ is the permutation group on $\{1,\ldots,r\}$. It was shown in \cite[Corollary 5.5]{BV2021} that the same expression holds if $\tensor{A}'$ is \textit{any} tensor close to $\tensor{A}$ and $\sum_{i=1}^r \tensor{A}_i'$ is the best rank-$r$ approximation of $\tensor{A}'$ in the Euclidean norm.}
	
	As a general principle in numerical analysis, {the condition number} is an intrinsic property of the {computational} problem that governs the forward error and attainable precision of any method for {solving the problem}. Its study is also useful for other purposes. For example, in \cite{BV2018,BV2018b} the local rate of convergence of Riemannian Gauss--Newton optimization methods for computing the CPD was related to the condition number $\kappa(\tensor{A})$.
	
	{{A conventional wisdom} in numerical analysis is that it is harder to compute the condition number of a given problem instance than solving the problem itself \cite{Demmel,Demmel0}. This viewpoint led Smale to initiate the study of the probability distribution of condition numbers: If {the condition number is small with high probability}, then for many practical purposes one can assume that any given input is well-conditioned; at least the probability of failure necessarily will be small. Smale started studying the probability that a polynomial is ill-conditioned \cite{Smale1981}. This strategy was extended to linear algebra condition numbers \cite{ES2005,ChenDongarra,BuCu2010}, to systems of polynomial equations in diverse settings \cite{BezII,Rojas}, to linear systems of inequalities \cite{Conic}, to linear and convex programming \cite{SpielmanTeng,Dennis}, eigenvalue and eigenvectors in the classic and other settings \cite{ArmentanoCucker}, to polynomial eigenvalue problems \cite{Khazh,AB2019}, and to other computational models \cite{Pardo}, {among others}. As there is a substantive bibliography on this setting, we refer the reader to \cite{BC2013} for further details.}
	
	{Tensor rank decomposition seems to be no exception to this wisdom: The characterization of $\kappa(\tensor{A})$ for a given $\tensor{A}$ in \cite{BV2017} requires the CPD of $\tensor{A}$ itself. This forces us to rely on probabilistic studies to establish reasonable \emph{a priori} values of the condition number.} {Settling this is the main purpose of this paper.}

	\subsection{Informal version of our main results and discussion.}
	
	The first probabilistic analyses of the condition number of CPD were given in \cite{BBV1,BV2018c}. In those references the expected value was computed for random rank-$1$ tensors; that is, for \textit{random output} of the computational problem of computing CPDs. This amounts to choosing random $\sten{u}{i}{k}$ in the notation above, constructing the corresponding tensor $\tensor{A}$ and studying $\kappa(\tensor{A})$. The probabilistic study is feasible, in principle, because one can obtain a closed expression for $\kappa(\tensor{A})$ which is polynomial in terms of the $\sten{u}{i}{k}$, so that the question boils down to an explicit but nontrivial integration problem.
	
	This article is the first to investigate the condition number for \emph{random input}. That is, we assume that $\tensor{A}$ is chosen at random within the set of rank-$r$ tensors (see the definition of {random tensors} in {\cref{def_gaussian_tensor} and the extension in \cref{cor:RIT}}) and we wonder about the expected value of $\kappa(\tensor{A})$. The difficulty now is that, even if we assume that a decomposition \eqref{eqn_CPD0} exists, we do not have it and hence we lack a closed expression for $\kappa(\tensor{A})$.
	
	One may wonder if these two different random procedures should give similar distributions in this or other numerical problems. The answer is {\em no}. For example, say that our problem is to compute the kernel of a given matrix $A\in\mathbb{R}^{n\times(n+1)}$ and we want to study the expected value of the associated condition number $\|A\|\,\|A^\dagger\|$. Choosing $A$ at random produces $\mathbb{E}(\|A\|\,\|A^\dagger\|)<\infty$ but choosing the kernel at random and then $A$ at random within the matrices with that kernel is the same as computing the expected value of the usual Turing's condition number of a square real Gaussian matrix, which is infinity; see \cite{ChenDongarra} for precise estimations of these quantities. The situation is similar in the study of systems of homogeneous polynomial equations: random inputs have better condition number than inputs produced from random outputs; see for example \cite{FLH}. In both these examples, the condition number of input constructed from random output is, on average, larger than the condition number of random input. This is a stroke of luck since in general one expects instances from practical, real life problems, to be somehow random within the input space, not to have a random output!
	
	In this paper we show that computing the CPD is a {\em rara avis}: We prove in \cref{thm_main} and \cref{thm_main_for_higher_ranks} that (under suitable hypotheses) the condition number of random input tensors turns out to be infinity. On the contrary, by \cite{BBV1,BV2018c} it is presumed that the average condition number is finite when choosing random output. This result reinforces the evidence that computing CPDs is a very challenging computational problem.
	
	The literature often cites the result of H\r{a}stad \cite{Hastad1990} to underline the high computational complexity of computing CPDs. {H\r{a}stad showed that the NP-complete 3-satisfiability problem (also called 3-SAT) can be reduced to computing the rank of a tensor}; hence, solving the tensor rank decomposition problem is NP-hard in the Turing machine computational model. Our main result is different in two aspects: first, H\r{a}stad showed the difficulty of only one particular instance of a CPD, whereas we show that computing the CPD is difficult \emph{on average}. Second, our evidence supporting the hardness of the problem is not based on Turing machine complexity, but given by analyzing the condition number, which is more appropriate for numerical computations \cite{BCSS1998}. Linking complexity analyses to condition numbers is common in the literature; for instance, in the case of solving polynomial systems \cite{lairez,FLH,BC2011,SS94}. In general, the book \cite{BC2013} provides a good overview. In this interpretation, we show that \emph{computing CPDs numerically is hard} on average.
	
	On the other hand, in the literature, the main result of de Silva and Lim \cite{dSL2008} is often cited as a key reason why approximating a tensor by a low-rank CPD is such a challenging problem: for some input tensors, a \emph{best} low-rank approximation may not exist! This is because the set of tensors of bounded rank is not closed: There are tensors of rank strictly greater than $r$ that can be approximated arbitrarily well by rank-$r$ tensors. It is shown in \cite{dSL2008} that this \emph{ill-posedness} of the approximation problem is not rare in the sense that for every tensor space $\R^{n_1 \times n_2 \times n_3}$ there exists an open set of input tensors which do not admit a best rank-$2$ approximation. This result is stronger than H\r{a}stad's in the sense that it proves that instances with no solution to the tensor rank approximation problem may occur on an open set, rather than in one particular set of measure zero. Notwithstanding this key result, it does not tell us about the complexity of solving the tensor rank \textit{decomposition} problem, in which we are given a rank-$r$ tensor whose CPD we seek. In this setting, there are no ill-posed inputs in the sense of \cite{dSL2008}. It was already shown in \cite{BV2017} that the condition number diverges as one moves towards the open part of the boundary of tensors of bounded rank, entailing that there exist regions with arbitrarily high condition number. One of the main result of this paper, \cref{thm_main_for_higher_ranks}, shows that such regions cannot be ignored: They are sufficiently large to cause the integral of the condition number over the set of rank-$r$ tensors to diverge. In other words, one cannot neglect the regions where the condition number is so high that a CPD computed from a floating-point representation of a rank-$r$ tensor in $\R^{n_1 \times \cdots \times n_d}$, subject only to roundoff errors, is meaningless---a result similar in spirit to de Silva and Lim \cite{dSL2008}.
	
	
	One may conclude from the above that, at least from the point of view of average stability of the problem, tensor rank decomposition is doomed to fail. However, if one only cares about the \emph{directions} of the rank-$1$ terms in the decomposition, {then the situation changes dramatically. The condition number associated with the computational problem ``Given a rank-$r$ identifiable tensor $\tensor{A}=\sum_{i=1}^r \tensor{A}_i$ as in \cref{eqn_CPD0}, output the set of normalized rank-$1$ tensors $\{ \frac{\tensor{A}_i}{\|\tensor{A}_1\|}, \ldots,  \frac{\tensor{A}_r}{\|\tensor{A}_r\|}\}$'' will be called the \textit{angular condition number} $\Kang(\tensor{A})$. Analogously to the bound \cref{eqn_error_bound1}, one can show that when $\Kang$ is finite, it is the smallest number such that
	\[
	\min_{\pi \in \mathfrak{S}_r} \sqrt{ \sum_{i=1}^r \left\| \frac{\tensor{A}_i}{\|\tensor{A}_i\|} - \frac{\tensor{A}_{\pi_i}'}{\|\tensor{A}_{\pi_i}'\|} \right\|^2 } \le \Kang(\tensor{A}) \| \tensor{A} - \tensor{A}' \| + o( \| \tensor{A} - \tensor{A}' \| )
	\]
	for all rank-$r$ tensors $\tensor{A}'=\sum_{i=1}^r \tensor{A}_i'$ (with $\tensor{A}_i'$ rank-$1$ tensors) in a sufficiently small open neighborhood of $\tensor{A}$. By \cite[Corollary 5.5]{BV2021} the same expression holds for \textit{all} tensors $\tensor{A}'$ in a small open neighborhood of $\tensor{A}$ if $\sum_{i=1}^r \tensor{A}_i'$ is the best rank-$r$ approximation of $\tensor{A}'$.}
	
	We will prove in \cref{thm_main2} that at least in the case of rank-$2$ tensors, {the angular condition number $\Kang$ for random inputs is finite, contrary to the classic condition number $\kappa$}; in fact, the numerical experiments in \cref{sec:experiments} suggest that this finite average condition seems to extend to much higher ranks as well. In other words, on average we may expect to be able to recover the \emph{angular part} of the CPD:
	\[
	 \tensor{U}_i = \frac{\tensor{A}_i}{\|\tensor{A}_i\|}, \quad\text{for } i = 1, \ldots, r,
	\]
	where $\tensor{A}_i$ is as in \cref{eqn_CPD0}.
	One could conclude from this that a tensor decomposition algorithm should aim to produce the normalized rank-$1$ terms $\tensor{U}_i$ from the tensor rank decomposition
	\[
	 \tensor{A} = \sum_{i=1}^r \lambda_i \tensor{U}_i
	\]
	accurately. Once these terms are obtained, one can recover the $\lambda_i$'s by solving a linear system of equations. Since, as a general principle, the condition number of a composite smooth map $g \circ f$ between manifolds satisfies \cite{BCSS1998,BC2013}
	\[
	 \kappa[g \circ f](x) := \| (\deriv{g}{f(x)}) (\deriv{f}{x}) \| \le \| \deriv{g}{f(x)} \| \| \deriv{f}{x} \| = \kappa[g](f(x)) \, \kappa[f](x),
	\]
	it follows that the condition number of tensor decomposition is bounded by the product of the condition numbers of the problem of finding the angular part of the CPD and the condition number of solving a linear least-squares problem. Our main results suggest that precisely the last problem will on average be ill-conditioned.
	
	The foregoing observation can have major implications for algorithm design. Indeed, solving the tensor rank decomposition problem by first solving for the angular part and then the linear least-squares problem decomposes the problem into a nonlinear and a linear part. Crucially, the latter least-squares problem can be solved by \emph{direct} methods, such as a QR-factorization combined with a linear system solver. Such methods have a uniform computational cost regardless of the condition number of the problem. By contrast, since no ({provably}) numerically stable direct algorithms for tensor rank decomposition are {currently} known \cite{BBV1}, iterative methods are indispensable for this problem. We may expect their computational performance to depend on the condition number of the problem instance. Indeed, our main results combined with the main result of \cite{BV2018} imply, for example, that Riemannian Gauss--Newton optimization methods for solving the angular part of the CPD should, on average, require less iterations to reach convergence than Riemannian Gauss--Newton methods for solving the tensor decomposition problem directly (such as the methods in \cite{BV2018,BV2018b}), because the angular condition number $\Kang$ appears to be finite on average, while the regular condition number $\kappa$ is proved to be $\infty$ on average in most cases, as we show in this article.
	
	Our main results also have consequences for researchers testing numerical algorithms for computing the CPD. In the literature, a common way of generating input data for testing algorithms is to sample the rank-$1$ terms $\tensor{A}_i = \lambda_i\sten{u}{i}{1} \otimes \sten{u}{i}{2} \otimes \cdots \otimes \sten{u}{i}{d}$ {randomly}, and then apply the algorithm to the associated tensor $\tensor{A} = \sum_{i=1}^r \tensor{A}_i$. However, our analysis in this paper and the analyses in \cite{BBV1,BV2018c} show that this procedure generates tensors that are heavily biased towards being numerically well-conditioned. Hence, this way of testing algorithms probably does not correspond to a realistic distribution on the inputs. We acknowledge that it is currently not easy to sample rank-$r$ tensors uniformly even though some methods exist \cite{BM2019}. In part, this is because equations for the algebraic variety containing the tensors of rank bounded by $r$ are hard to obtain \cite{Landsberg2012}. Nevertheless, in \cref{sec:experiments}, {using the observation from \cref{remark_conditional_distr}}, we present an acceptance-rejection method that can be applied to a few cases and yields uniformly distributed rank-$r$ tensors, relative to the {Gaussian density in \cref{def_gaussian_tensor}}. In any case we strongly advocate that the (range of) condition numbers are reported when testing the performance of iterative methods for solving the tensor rank decomposition problem, so that one can assess the difficulty of the problem instances. We believe it is always recommended to include models that are known to lead to instances with high condition numbers, such as those used in \cite{BV2017,BV2018b}.
	
	The formal presentation of our main results requires some extra notation that we introduce in subsequent sections.
	
	\subsection{Identifiable tensors and a formula for the condition number}\label{sec:identifiablity}
	A particular feature of higher-order tensors that distinguishes them from matrices is \emph{identifiability}. This means that in many cases the CPD of tensors of order $d\geq 3$ of small rank is unique. A tensor $\tensor{A} \in \R^{n_1 \times \cdots \times n_d}$ is called \emph{$r$-identifiable} if there is a unique set $\{ \tensor{A}_1, \ldots, \tensor{A}_r \}$ of cardinality $r$ such that $\tensor{A} = \tensor{A}_1 + \cdots + \tensor{A}_r$ and all $\tensor{A}_i$'s are rank-$1$ tensors.
	A celebrated criterion by Kruskal \cite{Kruskal1977} gives a tool to decide if a given tensor of order 3 satisfies this property.
	
	\begin{lemma}[Kruskal's criterion \cite{Kruskal1977,SGB2000}] \label{kruskal}
	Let $\F$ be $\R$ or $\C$, $\tensor{A}\in \F^{n_1\times n_2\times n_3}$ a tensor of order $3$ and assume that
	\(
	\tensor{A} = \sum_{i=1}^r \tensor{A}_i,
	\)
	where $\tensor{A}_i = \lambda_i \sten{u}{i}{1} \otimes \sten{u}{i}{2}\otimes\sten{u}{i}{3} \in \F^{n_1 \times n_2 \times n_3}.$
	Define the factor matrices $U_\ell = [\sten{u}{i}{\ell}]_{1\leq i \leq r} \in \F^{n_\ell \times r}$ for $\ell=1,2,3$, and let $k_\ell$ be the largest integer $k$ such that every subset of $k$ columns of $U_\ell$ has rank equal to $k$. If
	\(
	 r \le \frac{1}{2}( k_1 + k_2 + k_3 - 2)
	\)
	and $k_1, k_2, k_3 > 1$, then the tensor $\tensor{A}$ is $r$-identifiable over $\F$.
	\end{lemma}
	Since matrix rank does not change with a field extension from $\R$ to $\C$, a real rank-$r$ tensor $\tensor{A} \in \R^{n_1 \times n_2 \times n_3}$ that satisfies the assumptions of \cref{kruskal} is $r$-identifiable over $\R$ and also automatically $r$-identifiable over $\C$. In other words, Kruskal's criterion is certifying \emph{complex $r$-identifiability} of tensors, which is a strictly stronger notion than $r$-identifiability over $\R$ \cite{ABC2017}.

	Most order 3 tensors of low-rank satisfy Kruskal's criterion \cite{COV2017}: There is an open dense subset of the set of rank-$r$ tensors in $\R^{n_1 \times n_2 \times n_3}$, $n_1 \ge n_2 \ge n_3 \ge 2$, where complex $r$-identifiability holds, provided $r \le n_1 + \min\{ \tfrac{1}{2} \delta, \delta\}$ with $\delta := n_2 + n_3 - n_1 - 2$. In fact, this phenomenon occurs much more generally than third-order tensors of very small rank.
	Let us denote the set of complex tensors of complex rank bounded by $r$ by
	\[
	 \sigma_{r;n_1,\ldots,n_d}^\C := \{ \tensor{A} \in \C^{n_1 \times \cdots \times n_d} \mid \mathrm{rank}_\C(\tensor{A}) \le r \}.
	\]
	This constructible\footnote{The elements of $\sigma_{r;n_1,\ldots,n_d}^\C$ can be parameterized as in \cref{eqn_CPD0} changing $\R$ to $\C$.} set turns out to be an open dense subset (in the Euclidean topology) of its Zariski closure $\overline{\sigma_{r;n_1,\ldots,n_d}^\C}$; see \cite{Landsberg2012}. One says that $\sigma_{r;n_1,\ldots,n_d}^\C$ is \emph{generically complex $r$-identifiable}
	if the subset of points of $\sigma_{r;n_1,\ldots,n_d}^\C$ that are not complex $r$-identifiable is {contained in} a proper closed subset in the Zariski topology on the algebraic variety $\overline{\sigma_{r;n_1,\ldots,n_d}^\C}$; see \cite{CO2012}.
	It is known from dimensionality arguments \cite{CO2012} that there is a maximum value of $r$ for which generic $r$-identifiability of $\sigma_{r;n_1,\ldots,n_d}$ can hold, namely
	\begin{align} \label{eqn_r_bound}
	r \le r^{\text{crit}}_{n_1,\ldots,n_d}, \quad \text{ where }\quad r^{\text{crit}}_{n_1,\ldots,n_d}:=\frac{n_1 \cdots n_d}{1 + \sum_{k=1}^d (n_k - 1)}.
	\end{align}
	In fact, it is conjectured that the inequality is strict in general; see \cite{HOOS2016} for details.
	For all other values of $r$, generic $r$-identifiability does not hold.
	In \cite{BCO2014,CO2012,COV2014,DDL2015} it is proved that in the majority of choices for $n_1, \ldots, n_d$, generic complex $r$-identifiability holds for most ranks with~$r < r_\text{crit}$; see \cite[Theorem 7.2]{BCO2014} for a result that is asymptotically optimal. For a summary of the conjecturally complete picture of complex $r$-identifiability results, see \cite[Section 3]{COV2017}.

	\begin{assumption}\label{assumpt1}
	In the rest of this article, we will assume that $\sigma_{r;n_1,\ldots,n_r}^\C$ is generically complex $r$-identifiable.
	\end{assumption}
	
	The reason why we make this assumption is because it greatly simplifies some of the arguments. At the same time, \cref{assumpt1} is (conjectured to be) extremely weak and only limits the generality in the exceptional cases listed in \cite[Theorem 1.1]{COV2014}, and even then generic $r$-identifiability only fails very close to the upper bound $r_\text{crit}$ of the permitted ranks.
	
	An immediate benefit of \cref{assumpt1} is that it allows for a nice expression of the condition number of the tensor rank decomposition problem.
	 Let us denote the set of rank-1 tensors in~$\R^{n_1\times \cdots\times n_d}$ by
	\[
	\Var{S}_{n_1,\ldots,n_d} = \{\vect{a}^{1} \otimes \cdots \otimes \vect{a}^{d} \mid \vect{a}^{k}\in\R^{n_k}\backslash \{0\}\}.
	\]
	It is a smooth manifold, called the \emph{Segre manifold} \cite{Harris1992,Landsberg2012}. The set of tensors of rank bounded by $r$ is the image of the addition map: $\sigma_{r;n_1,\ldots,n_d} = \Phi(\Var S_{n_1,\ldots,n_d}^{\times r})$, where
	\begin{equation}\label{def_Phi}
	\Phi: \Var{S}_{n_1,\ldots,n_d} \times \cdots \times \Var{S}_{n_1,\ldots,n_d} \to \R^{n_1\times \cdots \times n_d},\; (\tensor{A}_1,\ldots,\tensor{A}_r) \mapsto \tensor{A}_1+\cdots+\tensor{A}_r.
	\end{equation}
	Then, under \cref{assumpt1}, there exists an open dense subset $\Var{N}_{r;n_1,\ldots,n_d}$ of $\sigma_{r;n_1,\ldots,n_d}$ such that for all $\tensor{A} \in \Var{N}_{r;n_1,\ldots,n_d}$ we have $\vert \Phi^{-1}(\tensor{A})\vert = r!$ by \cite[Proposition 4.5--4.7]{BBV1}.\footnote{The preimage of an $r$-identifiable tensor under the map $\Phi$ consists of the $r!$ permutations of the summands.}
	In particular, the points in the fiber are isolated, so there is a local inverse map $\Phi^{-1}_{\tuple{a}}$ of $\Phi$ for each $\tuple{a} \in \Phi^{-1}(\tensor{A})$. {Recall from \cite{BV2017} that} the condition number of the CPD at $\tensor{A}\in \Var{N}_{r;n_1,\ldots,n_d}$ is then the condition number {(in the classic sense of Rice \cite{Rice1966}; see also \cite{numerical_linear_algebra,BC2013})}
	of any of these local inverses:
	\begin{multline}\label{def_kappa}
	\kappa(\tensor{A}) := \lim_{\epsilon \to 0} \;\;\sup_{\Vert \Delta\tensor{A}\Vert <\epsilon \text{ s.t.\ }  \tensor{A}+\Delta\tensor{A}\in\sigma_{r;n_1,\ldots,n_d} } \; \frac{\Vert \Phi^{-1}_{\tuple{a}}(\tensor{A}) - \Phi^{-1}_{\tuple{a}}(\tensor{A}+\Delta\tensor{A}) \Vert}{\Vert \Delta\tensor{A}\Vert}\\{=\|\deriv{\Phi_a^{-1}}{\tensor A}\|_2},
	\end{multline}
	where $\tuple{a} \in \Phi^{-1}(\tensor{A})$ is arbitrary; it is a corollary of \cite[Theorem 1.1]{BV2017} that the above definition does not depend on the choice of $\tuple{a}$.
	Herein, $\Vert\cdot\Vert$ in the denominator is the Euclidean norm induced by the ambient $\R^{n_1\times \cdots\times n_d}$, and the norm in the numerator is the product norm of the Euclidean norms inherited from the ambient $\R^{n_1\times \cdots\times n_d}$'s.
	{The right-hand side $\|\deriv{\Phi_a^{-1}}{\tensor A}\|_2$ is the spectral norm of the derivative of $\Phi_a^{-1}$ at $\tensor{A}$. See \cref{sec_preliminaries} for more details.
	By \cite[Proposition 4.4]{BV2017}, the condition number $\kappa(\tensor{A})$ does not depend on the norm of $\tensor{A}$: $\kappa(t \tensor{A})=\kappa(\tensor{A})$ for $t\in\R\setminus\{0\}$.}
	
	\begin{rem} \label{rem_nodefinition}
	We did not specify the value of the condition number for $\tensor{A} \in \sigma_{r;n_1,\ldots,n_d}\setminus\Var{N}_{r;n_1,\ldots,n_d}$. The main reason is that our analysis is independent of the values that the condition number takes on this set of measure zero, so that for simplicity we decided against including the more complicated general case where there can be several distinct elements in the preimage.
	\end{rem}
	
	\subsection{Main results}\label{sec:mainresults}
	The goal of this paper is to study the average condition number {relative to ``reasonable'' density functions. By this we mean probability distributions $\hat{\rho}$ that are \textit{comparable} to the standard Gaussian density $\rho$: There exist positive constants $c_1,c_2$ such that $c_1 \le \frac{\hat\rho}{\rho} \le c_2$. The main result, \cref{cor:RIT}, applies, among others, for all distributions $\hat\rho$ comparable to the following Gaussian density defined on the set of bounded rank tensors $\sigma_{r;n_1,\ldots,n_d}$.}
	
	\begin{dfn}[Gaussian Identifiable Tensors]\label{def_gaussian_tensor}
	We define a random variable $\tensor{A}$ on $\sigma_{r;n_1,\ldots,n_d}$ by specifying its density as
	\[
	\rho(\tensor{A}):= (C_{r;n_1,\ldots,n_d})^{-1} \, e^{-\frac{\|\tensor{A}\|^2}{2}},\quad
	\text{where }\; C_{r;n_1,\ldots,n_d} = \int_{\sigma_{r;n_1,\ldots,n_d}}e^{-\frac{\|\tensor{A}\|^2}{2}}\, \d{}\tensor{A}
	\]
	is the normalization constant. Under \cref{assumpt1}, if $\tensor{A}\in\sigma_{r;n_1,\ldots,n_d}$ and  $\tensor{A}\sim \rho$, we say that $\tensor{A}$ is a Gaussian Identifiable Tensor~(GIT) of rank $r$.
	\end{dfn}
	
	\begin{rem}\label{remark_conditional_distr}
	{Suppose that $r$ is a \emph{typical rank} of tensors in $\mathbb R^{n_1\times \cdots\times n_d}$. This means that $\sigma_{r;n_1,\ldots,n_d}$ contains a Euclidean open subset of $\mathbb R^{n_1\times \cdots\times n_d}$ and is {of maximum dimension $n_1\cdots n_d$}. Then, the distribution defined in \cref{def_gaussian_tensor} is a \textit{conditional} probability distribution: A GIT $\tensor{A}$ of rank $r$ has the distribution $\tensor{A}\sim (\tensor{B} \mid \mathrm{rank}(\tensor{B}) = r)$, where $\tensor{B}$ is a tensor with independent and identically distributed (i.i.d.) standard Gaussian entries. We exploit this fact in our numerical experiments to sample GITs using an acceptance-rejection method.}
	\end{rem}
	
	{We first state our results for the foregoing Gaussian density. At the end of this subsection, in \cref{cor:RIT}, we generalize these results to other densities, including all densities comparable to the Gaussian density.}
	Our first contribution is the following result. We prove it in \cref{sec:proof_main}.
	
	\begin{thm}\label{thm_main}
	Let $\tensor{A}\in \sigma_{2;n_1,\ldots,n_d}$ be a GIT of rank $r= 2$. Then, $\mean \kappa(\tensor{A}) = \infty.$
	\end{thm}
	
	It should be mentioned that in our analysis we consider a small subset of $\sigma_{2;n_1,\ldots,n_d}$ and show that on this subset the condition number integrates to infinity. In particular, a weak average-case analysis as proposed in \cite{AL2017} would be of interest in this problem.
	
	Under one additional assumption we can extend the result from \cref{thm_main} to higher ranks. We prove the following theorem in \cref{sec:proof_main_higher_ranks}.
	
	\begin{thm}\label{thm_main_for_higher_ranks}
	Let $n_1,\ldots,n_d\geq 3$. On top of \cref{assumpt1} we assume that $\sigma_{r-2,n_1-2,\ldots,n_r-2}$ is generically complex identifiable. Then, for a GIT $\tensor{A}\in \sigma_{r;n_1,\ldots,n_d}$, $r \geq 3$, we have $\mean \kappa(\tensor{A}) = \infty.$
	\end{thm}
	
	By \cite[Theorem 7.2]{BCO2014}, the assumptions of \cref{thm_main_for_higher_ranks} are satisfied in a large number of cases. In fact, as the size of the tensor increases, the assumptions become weaker: When $n_1 \ge n_2 \ge \cdots \ge n_d \ge 2$ the conditions in \cref{thm_main_for_higher_ranks} are satisfied for $r\leq \min(s_1,s_2)$ with
	\begin{align*}
	 s_1=&\frac{n_1 n_2 - (n_1 + n_2 + n_3 - 2)}{n_1 n_2} r_{n_1,\ldots,n_d}^{\textrm{crit}}\\
	  s_2=&\frac{(n_1-2) (n_2-2) - (n_1 + n_2 + n_3 - 8)}{(n_1-2) (n_2-2)} r_{n_1-2,\ldots,n_d-2}^{\textrm{crit}} + 2.
	\end{align*}%
	Note that for large $n_i$, the second piece $s_2$ is the most restrictive. From \cref{eqn_r_bound} it is implied that $r_{n_1-2,\ldots,n_d-2}^\textrm{crit} = (1 - \delta_{n_1,\ldots,n_d}){r_{n_1,\ldots,n_d}^\textrm{crit}}$ with $\delta_{n_1,\ldots,n_d} = \mathcal{O}( \sum_{k=1}^d \frac{1}{n_k} )$. Therefore, we obtain the following asymptotically optimal result.
	\begin{cor}\label{cor_mostlytrue}
	Let $d \ge 3$ be fixed, and $n_1 \ge n_2 \ge \cdots \ge n_d \ge 2$. If $n_1,\ldots,n_d \to \infty$, then for a GIT $\tensor{A} \in \sigma_{r;n_1,\ldots,n_d}$ we have $\mathbb{E}\,\kappa(\tensor{A}) = \infty$ for all
	\[
	 2 \le r < (1 - \epsilon_{n_1,\ldots,n_d}) \, r_{n_1,\ldots,n_d}^{\textrm{crit}},
	\]
	where
	\(
	 \lim_{n_1,\ldots,n_d \to \infty} \epsilon_{n_1,\ldots,n_d} \to 0.
	\)
	\end{cor}
	
	It follows from dimensionality arguments that if $r>r_\text{crit}$, then the addition map $\Phi$ does not have a local inverse. In fact, in this case all of the connected components in the fiber of $\Phi$ at $\tensor{A}\in\sigma_{r;n_1,\ldots,n_d}$ have positive dimension \cite{Harris1992}. It follows from \cite{BV2017} that the condition number of the tensor rank decomposition problem at each expression \cref{eqn_CPD0} of length $r$ of such a tensor $\tensor{A}$ is $\infty$.
	In this case, $\kappa(\tensor{A}) = \infty$, regardless of how the tensor decomposition problem is defined\footnote{This is exactly the concern of \cref{rem_nodefinition}: What computational problem are we interested in solving when a tensor has several distinct CPDs? Are we interested in the CPD with the best sensitivity? Or the worst? Or the expected condition number of one randomly chosen CPD in the fiber? This depends on the context. The results of this paper are valid regardless of the particular variation of the problem one is interested in.} when $\tensor{A}$ has multiple distinct decompositions; see also the discussion in \cite[Remark 14.14]{BC2013}. In this case the average condition number is infinite, as well.
	
	Our results lead us to the conjecture that the expected condition number is infinite, also without making the assumption from \cref{thm_main_for_higher_ranks} and without any upper bound on the rank.
	\begin{conj} \label{conj_regular_cn}
	Let $\tensor{A}\in \sigma_{r;n_1,\ldots,n_d}$ be a GIT
	of rank $r\geq 2$. Then, $\mean \kappa(\tensor{A}) = \infty.$
	\end{conj}
	\Cref{cor_mostlytrue} above proves this conjecture asymptotically, in practice leaving only a small range of ranks for which it might fail.
	
	As mentioned above, it turns out that for GITs the expected \emph{angular condition number} is not always infinite. Formally, the angular condition number is defined as follows: Let the canonical projection onto the sphere be $p: \R^{n_1\times \cdots\times n_d} \to \mathbb{S}(\R^{n_1\times \cdots\times n_d})$. Then the angular condition number of $\tensor{A}\in \Var{N}_{r;n_1,\ldots,n_d}$ is
	\begin{equation}\label{def_kappa_ang}
	\Kang(\tensor{A}) := \lim_{\epsilon \to 0}\;\; \sup_{\substack{\Vert \Delta\tensor{A}\Vert <\epsilon,\\ \tensor{A}+\Delta\tensor{A}\in\sigma_{r;n_1,\ldots,n_d}} } \; \frac{\Vert (p^{\times r}\circ\Phi^{-1}_\tuple{a})(\tensor{A}) - (p^{\times r}\circ\Phi^{-1}_\tuple{a})(\tensor{A}+\Delta\tensor{A}) \Vert}{\Vert \Delta\tensor{A}\Vert},
	\end{equation}
	where $\Phi^{-1}_\tuple{a}$ is an arbitrary local inverse of $\Phi$ with $\tensor{A} = \Phi(\tuple{a})$. As before we do not specify what happens on the measure-zero set $\sigma_{r;n_1,\ldots,n_d}\setminus\Var{N}_{r;n_1,\ldots,n_d}$, because it is not relevant for this paper. The angular condition number only accounts for the angular part of the CPD, i.e., the directions of the tensors, not for their magnitude, hence the name.
	
	To distinguish the condition numbers \cref{def_kappa,def_kappa_ang}, we will refer to the condition number from \cref{def_kappa} as the \emph{regular condition number}. Oftentimes we even drop the clarification ``regular''.
	
	Here is the result for $\Kang(\tensor{A})$ for tensors of rank two that we prove in \cref{sec:angular_CN}.
	\begin{thm}\label{thm_main2}
	Let $\tensor{A}\in \sigma_{2;n_1,\ldots,n_d}$ be a GIT of rank 2. Then, $\mean \Kang(\tensor{A})  < \infty$.
	\end{thm}
	Unfortunately, we do not know if this theorem can be extended to higher rank tensors. However, based on our experiments in Section \ref{sec:experiments}, we pose the following:
	
	\begin{conj}
		Let $\tensor{A}\in \sigma_{r;n_1,\ldots,n_d}$ be a GIT of rank r. Then, $\mean \Kang(\tensor{A})  < \infty$.
	\end{conj}
	
	{We finally observe that {the foregoing main results are not limited to GITs. They are valid} for a wide range of {distributions of} random tensors.
	\begin{thm}\label{cor:RIT}
	Theorems \ref{thm_main}, \ref{thm_main_for_higher_ranks}, Corollary \ref{cor_mostlytrue} and Theorem \ref{thm_main2} are still true if instead of GITs we take random tensors defined by a wide range of other probability distributions, including some of interest such as:
	\begin{enumerate}
	\item All probability distributions that are comparable to the standard Gaussian density $\rho$. This means that the random tensor $\tensor{A}$ has a density $\hat{\rho}$ for which there exists positive constants $c_1,c_2$ such that $c_1 \le \frac{\hat\rho}{\rho} \le c_2$.
	\item {Uniformly} randomly chosen $\tensor A$ in the unit sphere $\mathbb S(\sigma_r)$.
	\item {Uniformly} randomly chosen $\tensor A$ in the unit ball $\{\tensor A\in\sigma_r:\|\tensor A\|\leq 1\}$.
	\end{enumerate}
	\end{thm}}

	\subsection{Acknowledgements} 
	Part of this work was made while the second and third author were visiting the Universidad de Cantabria, supported by the funds of Grant 21.SI01.64658 (Banco Santander and Universidad de Cantabria), Grant MTM2017-83816-P from the Spanish Ministry of Science, and the FWO Grant for a long stay abroad V401518N. We thank these institutions for their support. We also thank two anonymous referees for helpful comments.
	
	\subsection{Organization of the article}
	The rest of the article is organized as follows. In the next section we give some preliminary material. {Thereafter, in \cref{sec:proof_main,sec:proof_main_higher_ranks,sec:angular_CN,sec:otherrandom}, we successively prove \cref{thm_main}, \cref{thm_main_for_higher_ranks}, \cref{thm_main2} and \cref{cor:RIT}.}
	In \cref{sec:experiments} we present numerical experiments supporting our main results. Finally, in \cref{app_main,sec:proof_important_lemma,app_angular} we give proofs for several lemmata that we need in the other sections.
	
	\section{Notation and Preliminaries} \label{sec_preliminaries}
	
	\subsection{Notation}
	We will use the following typographic conventions for convenience: Vectors are typeset in a bold face ($\vect{a}, \vect{b}$), matrices in upper case ($A$, $B$), tensors in a calligraphic font ($\tensor{A}$, $\tensor{B}$), and manifolds and linear spaces in a different calligraphic font ($\Var{A}, \Var{B}$).
	
	The positive integer $d \ge 2$ is reserved for the order of a tensor, $n_1,\ldots,n_d \ge 2$ are its dimensions, and $r \ge 1$ is its rank. The following integers are used throughout the paper:
	\[
	 \Sigma := 1 + \sum_{k=1}^d (n_k - 1) \quad\text{and}\quad
	 \Pi := \prod_{k=1}^d n_k;
	\]
	they correspond to the dimension of the Segre manifold $\Var{S}_{n_1,\ldots,n_d}$ and the dimension of the ambient space $\R^{n_1 \times \cdots \times n_d}$ respectively.
	The symmetric group on $r$ elements is denoted by $\fr{S}_r$.
	
	We work exclusively with real vector spaces, for which $\langle \cdot ,\cdot \rangle$ denotes the Euclidean inner product and $\Vert \cdot \Vert$ always denotes the associated norm.
	We will switch {freely} between the finite-dimensional vector spaces $\R^{n_1 \cdots n_d}$ and $\R^{n_1 \times \cdots \times n_d}$ for representing tensors in the abstract vector space $\R^{n_1} \otimes \cdots \otimes \R^{n_d}$. By the above choice of norms all of these finite-dimensional Hilbert spaces are isometric; specifically, if $\tensor{A} \in \R^{n_1} \otimes \cdots \otimes \R^{n_d}$ and $\vect{a} \in \R^{n_1 \cdots n_d}$ is its coordinate array with respect to an orthogonal basis, then $\|\tensor{A}\| = \|\vect{a}\|$. Similarly, if the coordinates $\vect{a}$ are reshaped into a multidimensional array $A \in \R^{n_1 \times \cdots \times n_d}$, then $\|A\| = \|\tensor{A}\| = \|\vect{a}\|$. It is important to note that this notation can conflict with the usual meaning of $\|A\|$ when $d=2$; to distinguish the spectral norm from the standard norm in this paper, we write $\|A\|_2$ for the former; see \cref{def_spectral_norm}.
	
	{For matrices $U_1 \in \R^{m_1\times n_1}, \ldots, U_d \in \R^{m_d\times n_d}$, the tensor product $U_1\otimes \cdots\otimes U_d$ acts on rank-$1$ tensors as follows:
	\[
	(U_1\otimes \cdots\otimes U_d)(\vect{u}^1\otimes \cdots \otimes \vect{u}^d) = (U_1\vect{u}^1)\otimes \cdots \otimes (U_d\vect{u}^d).
	\]
	By the universal property \cite{Greub1978}, this extends to a linear map $\R^{n_1} \otimes \cdots \otimes \R^{n_d}\to \R^{m_1} \otimes \cdots \otimes \R^{m_d}$. Note that we can view $U_1\otimes \cdots\otimes U_d $ as a matrix in $\R^{(m_1\cdots m_d)\times (n_1\cdots n_d)}$.}
	
	For any subset $U \subset V$ of a {normed} vector space $V$, we define the sphere over $U$ as
	\[
	\mathbb{S}(U) := \left\{ \frac{\vect{u}}{\|\vect{u}\|}  \mid \vect{u} \in U\setminus\{0\} \right\} \subset V.
	\]
	In particular, the unit sphere in $\R^n$ is denoted by $\mathbb{S}(\R^n)$.
	
	Given an $m \times n$ matrix $R$ or a linear operator $R:\R^n\to \R^m$, we denote the pseudo-inverse by $R^\dagger$. The spectral norm and smallest singular value of $R$ are denoted respectively by
	\begin{equation}\label{def_spectral_norm}
	\Vert R\Vert_2 := \max_{\vect{v}\in \R^n}\frac{\|R \vect{v}\|}{\|\vect{v}\|} \quad \text{ and }\quad
	\varsigma_{\min}(R) := \min_{\vect{v}\in \R^n}\frac{\|R \vect{v}\|}{\|\vect{v}\|}.
	\end{equation}
	A special role will be played in this paper by the product of all but the smallest singular values of $R$, which we denote by $q(R)$. In other words, if $R$ is injective, then
	\begin{equation}\label{def_q}
	q(R) := {\varsigma_1(R) \cdots \varsigma_{n-1}(R)} = \frac{\sqrt{\det(R^TR)}}{\varsigma_{\min}(R)},
	\end{equation}
	where $R^T$ is the transposed matrix (operator) and $\varsigma_i(R)$ is the $i$th largest singular value of $R$.
	
	\subsection{Differential geometry}
	In this article we only consider submanifolds of Euclidean spaces; see, e.g., \cite{Lee2013} for the general definitions.
	A smooth ($C^\infty$) manifold is a topological manifold with a smooth structure, in the sense of \cite{Lee2013}.
	The \textit{tangent space} $\Tang{x}{\Var{M}}$ at $x$ to an embedded $n$-dimensional smooth submanifold $\Var{M} \subset \R^N$ is the set
	\[
	\left\{ \vect{v} \in \R^N \;|\; \exists \text{ a smooth curve } \gamma(t) \subset \Var{M} \text{ with }
	\gamma(0)=x: \vect{v} = \frac{\d{}}{\d{}t}\Big|_{t=0} \, \gamma(t) \right\}.                                                         \]
	At every point $x \in \Var{M}$, there exist open neighborhoods $\Var{V} \subset \Var{M}$ and $\Var{U} \subset \Tang{x}{\Var{M}}$ of $x$, and a bijective smooth map $\phi : \Var{V} \to \Var{U}$ with smooth inverse. The tuple $(\Var{V},\phi)$ is a \textit{coordinate chart}
	of $\Var{M}$. A \textit{smooth map} between manifolds $F : \Var{M} \to \Var{N}$ is a map such that for every $x \in \Var{M}$ and coordinate chart $(\Var{V},\phi)$ containing $x$, and every coordinate chart $(\Var{W}, \psi)$ containing $F(x)$, we have that $\psi \circ F \circ \phi^{-1} : \phi(\Var{U}) \to \psi(F(\Var{U}))$ is a smooth map. The \textit{derivative} of $F$ can be defined as the linear map $\deriv{F}{x} : \Tang{x}{\Var{M}} \to \Tang{F(x)}{\Var{N}}$ taking the tangent vector $\vect{v} \in \Tang{x}{\Var{M}}$ to $\frac{\d{}}{\d{}t}|_{t=0} F(\gamma(t)) \in \Tang{F(x)}{\Var{N}}$ where $\gamma(t) \subset \Var{M}$ is a curve with $\gamma(0) = x$ and $\gamma'(0) = \vect{v}$.
	If $\dim \Var M = \dim \Var N$ and if $\deriv{F}{x}$ has full rank, there is a neighborhood $\Var{W} \subset \Var{M}$ on which $F$ is invertible and its inverse is also smooth; that is, $F$ is a \emph{diffeomorphism} between $\Var{W}$ and $F(\Var{W})$. If this property holds for all $x \in \Var{M}$, then $F$ is called a \emph{local diffeomorphism}.
	
	A differentiable submanifold $\Var M\subset \R^N$ can be equipped with a \emph{Riemannian metric} $g$, turning it into a \textit{Riemannian manifold}, {allowing for the computation of integrals. The manifolds in this paper are all embedded submanifolds of Euclidean space, so the Riemannian metric for us will always be the metric inherited from the ambient space}.
	
	\subsection{The manifold of $r$-nice tensors}
	As in the introduction, the Segre manifold is
		\[
		\cal{S}_{n_1,\ldots,n_d} = \{\vect{u} ^1\otimes\cdots\otimes\vect{u}^d \mid \vect{u}^k\in \R^{n_k}\backslash \{0\}\}.
		\]
	It is a smooth manifold of dimension $\Sigma$. Its tangent space is given by
	\begin{equation}\label{tangent_space_segre}
	\Tang{\vect{u}^1\otimes\cdots\otimes\vect{u}^d}{\Var{S}_{n_1,\ldots,n_d}}
	= \R^{n_1} \otimes \vect{u}^2 \otimes\cdots\otimes\vect{u}^d
	+ \cdots +\vect{u}^1\otimes\cdots\otimes \vect{u}^{d-1}\otimes\R^{n_d};
	\end{equation}
	note that this is not a direct sum.
	
	The Euclidean inner product between rank-1 tensors is {conveniently computed} by the following {formula (see, e.g., \cite{Hackbusch2012})}:
	\begin{equation}\label{inner_prod_rank_one}
	\langle \vect{u}^1\otimes\cdots\otimes\vect{u}^d,\vect{v}^1\otimes\cdots\otimes\vect{v}^d\rangle = \prod_{i=1}^d \langle \vect{u}^i, \vect{v}^i\rangle.
	\end{equation}
	The set of tensors of rank at most $r$ is denoted by
	\[
	\sigma_{r;n_1,\ldots,n_d} = \{\tensor{A}\in \R^{n_1\times \cdots \times n_d} \mid \mathrm{rank}(\tensor{A})\leq r\};
	\]
	it is a semialgebraic set of dimension at most $\min\{r\Sigma, \Pi\}$; see, e.g., \cite{QCL2016}. Under \cref{assumpt1} the dimension of $\sigma_{r;n_1,\ldots,n_d}$ is exactly $r\Sigma$.

	{In \cite[Section 4]{BBV1} we introduced an open dense subset of $\sigma_{r;n_1,\ldots,n_d}$ with favorable differential-geometric properties.} We called it the manifold of \emph{$r$-nice tensors} in \cite[Definition 4.2]{BBV1}. Below, we present a slightly modified definition that is suitable for our present purpose; it eliminates conditions $(4)$ and $(5)$ from \cite[Definition 4.2]{BBV1}.
	
	{In what follows, we denote the real closure in the Zariski topology of a subset $A \subset \R^\Pi$ by~$\overline{A}$. This is the real algebraic variety $\overline{A}:=\overline{A}^{\mathbb C} \cap  \R^\Pi$, where $\overline{A}^{\mathbb C}$ is the closure of $A$ in the Zariski topology in $\C^\Pi$. By \cite[Lemma 8]{whitney}, the real dimension of $\overline{A}$ equals the complex dimension of~$\overline{A}^{\mathbb C}$.
	}
	
	\begin{dfn}\label{r_nice}
	Recall the addition map $\Phi$ defined in \cref{def_Phi}. Let $\Var M_{r;n_1,\ldots,n_d} \subset (\Var{S}_{n_1,\ldots,n_d})^{\times r}$ be the subset of $r$-tuples $\tuple{a}:=(\tensor{A}_1,\ldots,\tensor{A}_r)$ of rank-1 tensors satisfying all of the following properties:
	\begin{enumerate}
	\item $\Phi(\tuple{a})$ is a smooth point of {the algebraic variety $\overline{\sigma_{r;n_1,\ldots,n_d}}$};
	\item $\Phi(\tuple{a})$ is {complex} $r$-identifiable; and
	\item $\kappa(\Phi(\tuple{a})) < \infty$.
	\end{enumerate}
	The set of $r$-nice tensors is $\Var N_{r;n_1,\ldots,n_d}:=\Phi(\Var M_{r;n_1,\ldots,n_d})$.
	\end{dfn}
	Remark that the third item in the definition is well defined because of the second item.
	\begin{prop}\label{prop:delotropaper}
	If \cref{assumpt1} holds, then the following statements are true:
	\begin{enumerate}
	\item $\Var M_{r;n_1,\ldots,n_d}$ and $\Var N_{r;n_1,\ldots,n_d}$ are smooth manifolds of dimension $r\Sigma$;
	\item $\Var{M}_{r;n_1,\ldots,n_d}$ is {Zariski-open} in $(\Var{S}_{n_1,\ldots,n_d})^{\times r}$;
	\item $\Var{N}_{r;n_1,\ldots,n_d}$ is {Zariski-open} in $\sigma_{r;n_1,\ldots,n_d}$;
	\item the addition map $\Phi{\mid_{\Var M_{r;n_1,\ldots,n_d}}}$ is a {global} diffeomorphism onto its image;
	\item $\Var N_{r;n_1,\ldots,n_d}$ is closed under multiplication by nonzero scalars; and
	\item $\Var{M}_{r;n_1,\ldots,n_d} \subset \R^{n_1 \times \cdots \times n_d} \times \cdots \times \R^{n_1 \times \cdots \times n_d}$ and $\Var{N}_{r;n_1,\ldots,n_d} \subset \R^{n_1 \times \cdots \times n_d}$ are embedded submanifolds.
	\end{enumerate}
	\end{prop}
	\begin{proof}
	{Items 1, 2, 3, and 6 are proved as follows. Let $X_1$ and $X_2$ be respectively the set of tensors in $\sigma_{r;n_1,\ldots,n_d}$ which are not complex $r$-identifiable and which are not smooth points of $\overline{\sigma_{r;n_1,\ldots,n_d}}$. Both are Zariski-closed in $\overline{\sigma_{r;n_1,\ldots,n_d}}$ under \cref{assumpt1}, and hence so are the preimages $\Phi^{-1}(X_1)$ and $\Phi^{-1}(X_2)$. Moreover, the third defining condition of $\Var M_{r;n_1,\ldots,n_d}$ is also Zariski-closed in $(\Var{S}_{n_1,\ldots,n_d})^{\times r}$ from the explicit formula for the condition number \cref{characterization_CN} below. Hence, $\Var M_{r;n_1,\ldots,n_d}$ is Zariski-open. An open subset of an embedded submanifold is itself an embedded submanifold so the claim for $\Var M_{r;n_1,\ldots,n_d}$ is proved. Moreover, the dimension of the complement of $\Var M_{r;n_1,\ldots,n_d}$ is at most $r\Sigma-1$ and so its image by the rational map $\Phi$ is contained in an algebraic set of dimension at most $r\Sigma-1$, thus proving that $\Var{N}_{r;n_1,\ldots,n_d}$ is also Zariski-open and indeed an embedded submanifold of the set of smooth points of $\overline{\sigma_{r;n_1,\ldots,n_d}}$, which is itself an embedded submanifold of its affine ambient space, see \cite[Proposition 3.2.9]{BenedettiRisler}.}
	
	{The fourth item is due to the definition of the condition number, the fact that it is finite on~$\Var{N}_{r;n_1,\ldots,n_d}$ by \cref{r_nice}, and the injectivity of $\Phi|_{\Var{M}_{r;n_1,\ldots,n_d}}$ by \cref{r_nice} (2).}
	
	{The fifth item follows by noting that the three defining properties of $\Var{N}_{r;n_1,\ldots,n_d}$ are all true independent of a nonzero scaling.}
	\end{proof}
	
	\begin{rem}
	The definition of $r$-nice tensors in \cite[Definition 4.2]{BBV1} involves two more requirements, but those are not needed here.
	\end{rem}
	
	Since the tangent space of $\Var N_{r;n_1,\ldots,n_d}$ at a point is the image of the derivative of the {local diffeomorphism} $\Phi$, we have the following characterization:
	\begin{equation}\label{tangent_space_N}\mathrm{T}_{\tensor{A}} \,\Var N_{r,n_1,\ldots,n_d} = \mathrm{T}_{\tensor{A}_1} \,\Var S_{n_1,\ldots,n_d} + \cdots+  \mathrm{T}_{\tensor{A}_r} \,\Var S_{n_1,\ldots,n_d},\; \text{ for } \tensor{A}=\tensor{A}_1 + \cdots+ \tensor{A}_r.
	\end{equation}
	
	\subsection{Sensitivity of CPDs}
	The condition number of the problem of computing the rank-$1$ terms of a CPD of a tensor was studied in a general setting in \cite{BV2017}; the following characterization of the condition number is Theorem 1.1 of \cite{BV2017}. Let $\tensor{A}=\tensor{A}_1+\cdots+\tensor{A}_r\in \Var N_{r,n_1,\ldots,n_d}$, where the $\tensor{A}_i \in \Var{S}_{n_1,\ldots,n_d}$ are rank-$1$ tensors. For each $i$ let $U_i$ be a matrix whose columns form an orthonormal basis of $\mathrm{T}_{\tensor{A}_i} \Var S_{n_1,\ldots,n_d}$. Then,
	\begin{equation}\label{characterization_CN}
	\kappa(\tensor{A}) = \frac{1}{\varsigma_{\min}([U_1, \ldots, U_r])}.
	\end{equation}
	The matrix $U=[U_1, \ldots, U_r]\in\R^{\Pi\times r\Sigma}$ is also called a \emph{Terracini matrix}. An explicit expression for the $U_i$'s is given in \cite[equation (5.1)]{BV2017}.
	
	Since $\tensor{A}$ uniquely depends on $\tuple{a}:= (\tensor{A}_1, \ldots, \tensor{A}_r)\in \Var{S}_{n_1,\ldots,n_d}^{\times r}$, we can view the condition number of $\tensor{A}\in \Var{N}_{r,n_1,\ldots,n_d}$ as a function of $\tuple{a}$:
	\begin{equation}\label{kappa_for_output}\kappa(\tuple{a}):= \frac{1}{\varsigma_{\min}([U_1, \ldots, U_r])},\end{equation}
	where the matrices $U_i$ are as before. The benefit of \cref{kappa_for_output} is that it is well-defined for \emph{any} tuple $\tuple{a}\in \Var{S}_{n_1,\ldots,n_d}^{\times r}$ (and not just those mapping into $\Var N_{r,n_1,\ldots,n_d}$).
	
	\subsection{Integrals}
	For fixed {$t\in(0,1]$} and a point $\vect{y} \in \mathbb{S}(\R^n)$, the spherical cap of radius $t$ around $\vect{y}$ is defined as $\mathrm{cap}(\vect{y},t) := \{\vect{x} \in \mathbb{S}(\R^n): {\langle \vect{x}, \vect{y}\rangle\; >\sqrt{ 1-t^2}}\}$. {Its volume satisfies}
	\begin{equation}\label{cap_volume}
	c_1{(n)}t^{n-1}\leq \mathrm{vol}(\mathrm{cap}(\vect{y},t))\leq c_2{(n)}t^{n-1}
	\end{equation}
	for some positive constants $0<c_1{(n)}<c_2{(n)}$.
	
	The following general lemma will be useful later.
	\begin{lemma}\label{lemma_integrals}
	Let $u,v>0$ be fixed. Then,
	$0<\int_0^\infty t^u \,e^{-\frac{(t + v)^2}{2}}\,\d{} t<\infty.$
	\end{lemma}
	\begin{proof}
	It is clear that the integral is not zero. Furthermore,
	since $(t + v)^2 > t^2+v^2$ for $t,v>0$, we see that
	$\int_0^\infty t^u \,e^{-\frac{(t + v)^2}{2}}\,\d{} t \leq \int_0^\infty t^u \,e^{-\frac{t^2 + v^2}{2}}\,\d{} t = e^{-\frac{v^2}{2}}\sqrt{2}^{u-1}\Gamma(\frac{u+1}{2})$, which is finite.
	\end{proof}
	
	\subsection{The coarea formula}
	Let $\Var M$ and $\Var N$ be submanifolds of $\R^n$ of equal dimension, and let $F: \Var M \to \Var N$ be a smooth surjective map.
	A point $y\in \Var N$ is called a \emph{regular value} of $F$ if for all points $x\in F^{-1}(y)$ the differential $\deriv{F}{x}$ is of full rank. The preimage $F^{-1}(y)$ of a regular value $y$ is a {discrete} set of points. {Let $|F^{-1}(y)|$ be the number of elements in this preimage.} Then, the coarea formula \cite{howard} states that for every integrable function $g$ we have
	\begin{align}\label{eqn_coarea}
	 \int_{\Var{N}} |F^{-1}(y)| \, g(y) \,\d{}y = \int_{\Var{M}} \mathrm{Jac}(F)(x) \, g(F(x)) \,\d{}x,
	\end{align}
	where $\mathrm{Jac}(F)(x):=\vert\det\deriv{F}{x}\vert$ is the \emph{Jacobian determinant} of $F$ at $x$.
	Note that almost all $y\in \Var N$ are regular values of $F$ by Sard's theorem \cite[Theorem 6.10]{Lee2013}. Hence, integrating over~$\Var N$ is the same as integrating over all regular values of $F$.
	\begin{rem}
	In \cite{howard}, the coarea formula is given in the more general case when $\dim \Var M \geq \dim \Var N$. In this article we only need the case when the dimension of $\dim \Var M$ and $\dim \Var N$ coincide. Moreover, if $F$ is injective, then \cref{eqn_coarea} reduces to the well-known change-of-variables formula.
	\end{rem}
	
	\section{The average condition number of Gaussian tensors of rank two}\label{sec:proof_main}
	The goal of this section is to prove \cref{thm_main}. We will proceed in three steps. First, the $2$-nice tensors are conveniently parameterized via elementary manifolds such as one-dimensional intervals and spheres in \cref{sec_step1_parameterization}. Second, the Jacobian determinant of this map is computed in \cref{sec_jacobian_determinant}. Third, the integral can be bounded from below with the help of a few technical auxiliary lemmas in \cref{sec:kaffexp}. In the next section, we will exploit \cref{thm_main} for generalizing the argument to most higher ranks. To simplify notation, in this section we let
	\[
	\Var S:= \Var S_{n_1,\ldots,n_d},\;\sigma_2:=\sigma_{2,n_1,\ldots,n_d},\;\Var N_2:=\Var N_{2,n_1,\ldots,n_d} \;\text{and}\;\Var M_2:=\Var M_{2,n_1,\ldots,n_d}.
	\]

	\subsection{Parameterizing $2$-nice tensors} \label{sec_step1_parameterization}
	Let
	\begin{equation}\label{def_P}
	\Var{P}:=\Sp{}(\R^{n_1}) \times \cdots \times \Sp{}(\R^{n_d})
	\end{equation}
	and consider the next parametrization of the Segre manifold:
	\begin{equation}\label{def_psi}
	\psi: (0, \infty) \times \Var P \to \Var S,\; (\lambda, \vect{u}^1,\ldots,\vect{u}^d) \mapsto \lambda \cdot \vect{u}^1 \otimes \cdots\otimes \vect{u}^d.
	\end{equation}
	The preimage of $\tensor{A}\in \Var S$ has cardinality $\vert \psi^{-1}(\tensor{A})\vert = 2^{d-1}$. By composing $\Psi := \psi \times \psi$ with the addition map from \cref{def_Phi} we get the following alternative representation of {tensors of rank bounded by $2$}:
	\begin{equation*} 
	(\Phi\circ \Psi)((\lambda, \vect{u}^1,\ldots,\vect{u}^d), (\mu, \vect{v}^1,\ldots,\vect{v}^d)) =  \lambda \cdot  \vect{u}^1 \otimes \cdots\otimes \vect{u}^d + \mu \cdot  \vect{v}^1 \otimes \cdots\otimes \vect{v}^d.
	\end{equation*}
	We would like to apply the coarea formula \cref{eqn_coarea} to pull back the integral of $\kappa(\tensor{A})e^{-\frac{\Vert \tensor{A}\Vert}{2}}$ over~$\sigma_{2}$ via the parametrization $\Phi\circ\Psi$. However, $\sigma_{2}$ in general is not a manifold, so the formula does not apply. Nevertheless, we can use the manifold $\Var N_{2}$ of $2$-nice tensors instead.  By \cref{prop:delotropaper}~(3), $\Var N_2$ is {Zariski open} in $\sigma_2$, so that
		\[
		\mean \kappa(\tensor{A}) = \frac{1}{C_2}\int_{\sigma_{2}} \kappa(\tensor{A}) e^{-\frac{\Vert \tensor{A}\Vert^2}{2}}\, \d{}\tensor{A} = \frac{1}{C_2}\int_{\Var N_{2}} \kappa(\tensor{A}) e^{-\frac{\Vert \tensor{A}\Vert^2}{2}}\, \d{}\tensor{A},
		\]
	where $C_2 := C_{2;n_1,\ldots,n_d}$ {is as in \cref{def_gaussian_tensor}}.
	By applying the coarea formula \cref{eqn_coarea} to the smooth map $\Phi\mid_{\Var M_2}$ we get
	\begin{align*}
	\int_{\Var N_{2}} \kappa(\tensor{A})\, e^{-\frac{\Vert \tensor{A}\Vert^2}{2}}\, \d{}\tensor{A} &=\frac{1}{2}\int_{\Var N_{2}} \vert\Phi^{-1}(\tensor{A})\vert\,\kappa(\tensor{A})\, e^{-\frac{\Vert \tensor{A}\Vert^2}{2}}\, \d{}\tensor{A}\\
	&= \frac{1}{2}\int_{\Var M_{2}} \mathrm{Jac}(\Phi)(\tensor{A}_1,\tensor{A}_2)\,{\kappa(\tensor{A}_1+ \tensor{A}_2)}\,  e^{-\frac{\Vert \tensor{A}_1+\tensor{A}_2 \Vert^2}{2}}\, \d{}\tensor{A}_1 \d{}\tensor{A}_2,
	\end{align*}
	where $\mathrm{Jac}(\Phi)(\tensor{A}_1,\tensor{A}_1)$ is the Jacobian determinant of $\Phi$ at $(\tensor{A}_1,\tensor{A}_1)$.
	In the first equality we used $|\Phi^{-1}(\tensor{A})| = 2$ for $2$-identifiable tensors; indeed, we have that $\Phi(\tensor{A}_1,\tensor{A}_2) = \Phi(\tensor{A}_2,\tensor{A}_1) = \tensor{A}$ and~$\tensor{A}_1 \ne \tensor{A}_2$ because $\tensor{A}\in\Var{N}_2$ has rank equal to $2$.
	
	In the following, we switch to the notation from \cref{kappa_for_output}: $\kappa(\tensor{A}_1 + \tensor{A}_2) = \kappa(\tensor{A}_1, \tensor{A}_2)$.
	Since $\Var M_{2}$ is {also Zariski open} in $\Var S \times \Var S$ by \cref{prop:delotropaper} (2), we may replace the integral over $\cal{M}_{2}$ by an integral over $\Var S \times \Var S$, thus obtaining
	\[
	\mean \kappa(\tensor{A})
	= \frac{1}{2C_2} \int_{\Var S \times \Var S} \mathrm{Jac}(\Phi)(\tensor{A}_1,\tensor{A}_1)\,\kappa(\tensor{A}_1,\tensor{A}_2)\,  e^{-\frac{\Vert \tensor{A}_1+\tensor{A}_2 \Vert^2}{2}}\, \d{}\tensor{A}_1\d{}\tensor{A}_2.
	\]
	We use the coarea formula again, but this time for $\Psi = \psi\times\psi$, where $\psi$ is the parametrization from \cref{def_psi}. Note that for $(\tensor{A}_1,\tensor{A}_2)\in\Var M_{2}$ we have $\vert\Psi^{-1}(\tensor{A}_1,\tensor{A}_2)\vert = 2^{2d-2}$. We get
	{\small{
	\begin{align}
	\nonumber
	\mean \kappa(\tensor{A}) = &\frac{1}{2C_2}\int_{\Var S\times \Var S} \mathrm{Jac}(\Phi)(\tensor{A}_1,\tensor{A}_2)\,{\kappa(\tensor{A}_1, \tensor{A}_2)}\,  e^{-\frac{\Vert \tensor{A}_1+\tensor{A}_2 \Vert^2}{2}}\, \d{}\tensor{A}_1\d{}\tensor{A}_2 \\
	\nonumber
	=&\frac{1}{2^{2d-1}C_2}\int_{\Var S\times \Var S}\vert \Psi^{-1}(\tensor{A}_1,\tensor{A}_2)\vert\, \mathrm{Jac}(\Phi)(\tensor{A}_1,\tensor{A}_2)\, {\kappa(\tensor{A}_1, \tensor{A}_2)}\,  e^{-\frac{\Vert \tensor{A}_1+\tensor{A}_2\Vert^2}{2}}\, \d{}\tensor{A}_1 \d{}\tensor{A}_2\\
	= &\frac{1}{2^{2d-1}C_2}\int_{((0,\infty) \times \Var{P})^{\times 2}} \mathrm{Jac}(\Phi\circ\Psi)(\tuple{a},\tuple{b})\,{\kappa(\psi(\tuple{a}), \psi(\tuple{b}))}\,  e^{-\frac{\Vert \psi(\tuple{a}) + \psi(\tuple{b}) \Vert^2}{2}}\, \d{}\tuple{a}\d{}\tuple{b},\label{eq3.8}
	\end{align}
	}}
	where $\tuple{a} = (\lambda,\vect{u}^1,\ldots,\vect{u}^d)$ and $\tuple{b} = (\mu,\vect{v}^1,\ldots,\vect{v}^d)$ are both tuples in $(0,\infty) \times \Var{P}$. Next, we compute the Jacobian determinant $\mathrm{Jac}(\Phi\circ\Psi)(\tuple{a},\tuple{b})$.
	
	\subsection{Computing the Jacobian determinant}\label{sec_jacobian_determinant}
	Note that the dimension of the domain of $\Phi\circ \Psi$ is equal to $2\Sigma$.
	As above, let
	\(
	\tuple{a} = (\lambda,\vect{u}^1,\ldots,\vect{u}^d)\)
	and \(\tuple{b} = (\mu,\vect{v}^1,\ldots,\vect{v}^d)
	\)
	be tuples in $(0,\infty) \times \Var{P}$ with $\Var{P}$ as in \cref{def_P}. In the following, we write
	\[
	\tensor{U}:=\vect{u}^1\otimes \cdots \otimes\vect{u}^d \quad\text{ and }\quad\tensor{V}:=\vect{v}^1\otimes \cdots \otimes\vect{v}^d.
	\]
	The Jacobian determinant of $\Phi\circ \Psi$ at $(\tuple{a},\tuple{b})$ is, by definition, the absolute value of the determinant of the linear map
	\[
	\deriv{(\Phi\circ \Psi)}{(\tuple{a},\tuple{b})} : \Tang{\lambda}{(0,\infty)} \times \Tang{\mu}{(0,\infty)} \times \Tang{(\vect{u}^1,\ldots,\vect{u}^d)}{\Var{P}}  \times \Tang{(\vect{v}^1,\ldots,\vect{v}^d)}{\Var{P}} \to \Tang{\lambda \tensor{U} + \mu \tensor{V}}{\Var N_2}.
	\]
	Consider the matrix of partial derivatives of $\Phi\circ \Psi$ with respect to the standard orthonormal basis of $\R^{n_1\times \cdots \times n_d}$:
	\begin{align} \label{eqn_def_Q}
	 Q := \begin{bmatrix} L & M \end{bmatrix}\in \R^{\Pi \times 2\Sigma},
	\end{align}
	where
	\begin{equation} \label{def_L_and_M}
	L :=  \begin{bmatrix} \frac{\partial(\Phi\circ \Psi)}{\partial \vect{u}^1} & \ldots &\frac{\partial(\Phi\circ \Psi)}{\partial \vect{u}^d} & \frac{\partial(\Phi\circ \Psi)}{\partial \vect{v}^1} & \ldots &\frac{\partial(\Phi\circ \Psi)}{\partial \vect{v}^d}\end{bmatrix}
	 \quad\text{and}\quad
	 M:= \begin{bmatrix} \frac{\partial(\Phi\circ \Psi)}{\partial \lambda} & \frac{\partial(\Phi\circ \Psi)}{\partial \mu} \end{bmatrix}.
	\end{equation}
	{Then, the Jacobian determinant of $\Phi\circ \Psi$ at $(\tuple{a},\tuple{b})$ is
	 \begin{equation}\label{def_vol}
	 \Jac{\Phi\circ \Psi}{\tuple{a},\tuple{b}} = \sqrt{\det(Q^TQ)}.
	 \end{equation}
	The latter is the volume of the parallelepiped spanned by the columns of $Q$. We fix notation in the next definition.}
	\begin{dfn}\label{def_vol_U}
	{Let $N{\,\geq\,}n$ be positive integers and $U\in\mathbb R^{N\times n}$ be a matrix with columns $\vect{u}_1,\ldots,\vect{u}_n\in\mathbb R^N$. We denote by $\mathrm{vol}(U)$ the volume of the parallelepiped spanned by the $\vect{u}_i$:
	\[
	\mathrm{vol}(U) := \sqrt{\det(U^TU)}.
	\]}
	\end{dfn}
	
	{We can now rewrite \cref{def_vol} as
	 \begin{equation*}
	 \Jac{\Phi\circ \Psi}{\tuple{a},\tuple{b}} =\mathrm{vol}(Q).
	 \end{equation*}}
	The reason why we write the partial derivatives of $\Phi\circ\Psi $ with respect to the standard basis of $\R^{n_1\times \cdots\times n_d}$ is that we get the following convenient description:
	\begin{align}\label{def_M}
	M =
	\begin{bmatrix}
	 \tensor{U} & \tensor{V}
	\end{bmatrix}.
	\end{align}
	For describing $L$, let for each $1\le k \le d$,
	\[
	\dot{U}^k =
	\begin{bmatrix}
	\sten{\dot{u}}{2}{k} & \sten{\dot{u}}{3}{k} & \cdots & \sten{\dot{u}}{n_k}{k}
	\end{bmatrix} \in \R^{n_k \times {(}n_k-1{)}}
	\;\text{and}\;
	\dot{V}^k =
	\begin{bmatrix}
	\sten{\dot{v}}{2}{k} & \sten{\dot{v}}{3}{k} & \cdots & \sten{\dot{v}}{n_k}{k}
	\end{bmatrix} \in \R^{n_k \times {(}n_k-1{)}}
	\]
	be matrices containing as columns an ordered orthonormal basis of $(\vect{u}^k)^\perp = \Tang{\vect{u}^k}{\Sp(\R^{n_k})}$ and $(\vect{v}^k)^\perp = \Tang{\vect{v}^k}{\Sp(\R^{n_k})}$, respectively. Then, by linearity and the product rule of differentiation, we have that $L = \begin{bmatrix} \lambda L_1 &
	\mu L_2 \end{bmatrix}$ is the block matrix consisting of $2$ blocks of the form
	\begin{align}\label{def_L_i}
	L_1 =  \begin{bmatrix} L_1^1 & \cdots & L_1^d \end{bmatrix}
	\text{ with }
	L_1^k &= \sten{u}{}{1} \otimes \cdots \otimes \sten{u}{}{k-1} \otimes \dot{U}^k \otimes \sten{u}{}{k+1} \otimes \cdots \otimes \sten{u}{}{d} \text{ and }\\
	L_2 = \begin{bmatrix} L_2^1 & \cdots & L_2^d \end{bmatrix}
	\text{ with }
	L_2^k &= \sten{v}{}{1} \otimes \cdots \otimes \sten{v}{}{k-1} \otimes \dot{V}^k \otimes \sten{v}{}{k+1} \otimes \cdots \otimes \sten{v}{}{d}.\nonumber
	\end{align}
	Both $L_1$ and $L_2$ have $\sum_{k=1}^d (n_k-1) = \Sigma-1$ columns.
	Note that $M$ depends only on the $\vect{u}^k$'s and $\vect{v}^k$'s, whereas $L$ also depends on the parameters $\lambda$ and $\mu$; we do not emphasize these dependencies in the notation.
	
	Comparing with \cite[equation (5.1)]{BV2017}, we see that the matrix $L_1$ has as columns an orthonormal basis for the orthogonal complement of $\tensor{U}$ in $\Tang{\tensor{U}}{\Var S}$. Analogously, the columns of $L_2$ form an orthonormal basis for the orthogonal complement of
	$\tensor{V}$ in $\Tang{\tensor{V}}{\Var S}$.
	Consequently, for $\Psi(\tuple{a},\tuple{b})$, Terracini's matrix from \cref{characterization_CN} can be chosen as
	\begin{align} \label{eqn_nice_expr_terracini}
	 U = \begin{bmatrix} \tensor{U} & L_1 & \tensor{V} & L_2 \end{bmatrix}.
	\end{align}
	This entails that
	\begin{align} \nonumber
	\Jac{\Phi\circ \Psi}{\tuple{a},\tuple{b}} &= \sqrt{\det\bigl( \begin{bmatrix} \lambda L_1 & \mu L_2 & \tensor{U} & \tensor{V} \end{bmatrix}^T \begin{bmatrix} \lambda L_1 & \mu L_2 & \tensor{U} & \tensor{V} \end{bmatrix} \bigr)}
	\end{align}
	and so 
	\begin{equation}\label{eqn_nice_expr_jacdet}
	\Jac{\Phi\circ \Psi}{\tuple{a},\tuple{b}} = \lambda^{\Sigma-1} \mu^{\Sigma-1} \mathrm{vol}(U)
	\end{equation}
	having used {the notation from \cref{def_vol_U}} and {the fact} that singular values are invariant under orthogonal transformations such as permutations of columns.

	\subsection{Bounding the integral}\label{sec:kaffexp}
	We are now ready to conclude the proof of \cref{thm_main}, by showing that the expected value of the condition number of tensor rank decomposition is {bounded from below by infinity}.
	
	By \cref{characterization_CN}, the condition number at $\tensor{A} = \tensor{A}_1+\tensor{A}_2 = \Phi(\tensor{A}_1,\tensor{A}_2) \in \Var{N}_2$ is the inverse of the smallest singular value of the Terracini's matrix $U$ from \cref{eqn_nice_expr_terracini}. Therefore, if we plug~\cref{eqn_nice_expr_terracini} and~\cref{eqn_nice_expr_jacdet} into~\cref{eq3.8}, then we get
	\begin{align} \nonumber 
	\mathbb{E}\kappa(\tensor{A})
	&= \frac{1}{2^{2d-1}C_2}\int_{((0,\infty) \times \Var{P})^{\times 2}}\frac{\lambda^{\Sigma-1} \mu^{\Sigma-1} \mathrm{vol}(U)}{\varsigma_{\min}(U)}\,  e^{-\frac{\Vert \lambda  \tensor{U} + \mu  \tensor{V}\Vert^2}{2}}\, \d{}\lambda\,\d{}\tuple{u}\,\d{}\mu\,\d{}\tuple{v} \\
	&= \frac{1}{2^{2d-1}C_2} \int_{((0,\infty) \times \Var{P})^{\times 2}} \lambda^{\Sigma-1}\,\mu^{\Sigma-1}\,q(U)\, e^{-\frac{\Vert \lambda \tensor{U} + \mu \tensor{V}\Vert^2}{2}}\, \d{}\lambda\,\d{}\tuple{u}\,\d{}\mu\,\d{}\tuple{v}, \label{eq3.11}
	\end{align}
	where
	\(
	{q(U) = \frac{\mathrm{vol}(U)}{\varsigma_{\min}(U)}}
	\) is as in \cref{def_q}, and
	\begin{equation}\label{def_U_V}
	\tuple{u} = (\sten{u}{}{1},\ldots, \sten{u}{}{d}), \;
	\tuple{v} = (\sten{v}{}{1},\ldots, \sten{v}{}{d}), \;
	\tensor{U}= \sten{u}{}{1} \otimes \cdots \otimes \sten{u}{}{d},\; \text{ and }
	\tensor{V}= \sten{v}{}{1} \otimes \cdots \otimes \sten{v}{}{d}.
	\end{equation}
	From \cref{eqn_nice_expr_terracini} it is clear that $U$ is a function of $\tuple{u}$ and $\tuple{v}$ but is independent of $\lambda$ and $\mu$. Therefore, if we integrate first over $\lambda$ and $\mu$, then we can ignore the factor $q(U)$. In \cref{proof_lem1} we compute this integral; the result is stated here as the next lemma.
	
	\begin{lemma}\label{lem:innerBV}
	Let $(\vect{u}^1,\ldots,\vect{u}^d), (\vect{v}^1,\ldots,\vect{v}^d) \in \Var{P}$ be fixed. Then,
		\begin{multline*}
		\int_{(0,\infty)}\int_{(0,\infty)} \lambda^{\Sigma-1}\,\mu^{\Sigma-1}\,e^{-\frac{\Vert \lambda \tensor{U} + \mu \tensor{V}\Vert^2}{2}}\,\d{}\lambda\d{}\mu
		=\\ 2^{\Sigma -1}\Gamma(\Sigma) \int_{0}^{\frac{\pi}{2}}\frac{(\cos(\theta)\sin(\theta))^{\Sigma-1}}{{\|\cos(\theta) \tensor{U}+ \sin(\theta)\tensor{V}\|^{2\Sigma}}} \,\d{}\theta,
		\end{multline*}
	where $\tensor{U} = \vect{u}^1 \otimes \cdots \otimes \vect{u}^d$ and $\tensor{V} = \vect{v}^1 \otimes \cdots \otimes \vect{v}^d$.
	\end{lemma}
	The foregoing integral can be bounded from below by exploiting the next lemma, which is proved in \cref{proof_lem2}.
	\begin{lemma}\label{lem:integralbound}
	{Let $\vect{x},\vect{y}\in \Sp(\R^p)$ be two unit-norm} vectors and $s\geq 1$. Then, there exists a constant $k=k(p,s)$ independent of $\vect{x},\vect{y}$ such that
	\[
	\int_0^{\frac{\pi}{2}}\frac{(\cos(\theta) \sin(\theta) )^{s-1}}{\|\cos(\theta)\vect{x}+\sin(\theta)\vect{y}\|^{2s}}\,\d{}\theta\geq\frac{k}{{\|\vect{x} + \vect{y}\|^{2s-1}}}.
	\]
	\end{lemma}
	Combining the foregoing lemmata and plugging the result into \cref{eq3.11}, we obtain
	\begin{equation*}
	\mathbb{E}\kappa(\tensor{A})\geq
	\frac{2^{\Sigma -1}\Gamma(\Sigma) \, k}{2^{2d-1}C_2} \int_{\Var{P} \times \Var{P}}\,\frac{q(U)}{\|\tensor{U} + \tensor{V}\|^{2\Sigma-1}} \, \d{}\tuple{u}\,\d{}\tuple{v}.
	\end{equation*}
	Next, we exploit the symmetry of the domain $\mathbb{S}(\R^{n_1})$ by flipping the sign of $\vect{v}^1$ and, hence, of $\tensor{V}=\vect{v}^1\otimes\cdots\otimes\vect{v}^d$. This substitution transforms $U$ into $UD$, where $D$ is a diagonal matrix with some pattern of $\pm 1$ on the diagonal. Since $D$ is orthogonal, $q(U) = q(UD)$, so that
	\begin{equation*}
	\mathbb{E}\kappa(\tensor{A})\geq
	\frac{2^{\Sigma -1}\Gamma(\Sigma) \, k}{2^{2d-1}C_2} \int_{\Var{P} \times \Var{P}}\,\frac{q(U)}{\|\tensor{U} - \tensor{V}\|^{2\Sigma-1}} \, \d{}\tuple{u}\,\d{}\tuple{v}.
	\end{equation*}
	Denote this last integral by $J$, and then it remains to show that $J=\infty$.
	Consider the open set
	{\small{
	\begin{align*}
		D(\epsilon)=\Big\{&(\tuple{u},\tuple{v}) \in \Var{P} \times \Var{P} \mid \tfrac{9}{10}\|\vect{u}^1-\vect{v}^1\| < \|\vect{u}^k-\vect{v}^k\| < \|\vect{u}^1-\vect{v}^1\| < \epsilon \text { for } 2\leq k\leq d\Big\}.
	\end{align*}
	}}
	Since $D(\epsilon)$ is open, we have
	\begin{equation}\label{eq3.30}
	J\geq  \int_{D(\epsilon)}\,\frac{q(U)}{\|\tensor{U}-\tensor{V}\|^{2\Sigma-1}} \, \, \d{}\tuple{u}\,\d{}\tuple{v}.
	\end{equation}
	
	We now need two lemmata. The first one is straightforward.
	\begin{lemma}\label{bound_tensor_norm}
	Let $\epsilon>0$ be sufficiently small. For all $(\tuple{u},\tuple{v})\in D(\epsilon)$ with $\tuple{u}=(\vect{u}^1,\ldots,\vect{u}^d)$ and $\tuple{v}=(\vect{v}^1,\ldots,\vect{v}^d)$, we have
	\[
	\| \vect{u}^1 - \vect{v}^1 \| \le \|\tensor{U}-\tensor{V} \| \leq d \, \Vert \vect{u}^1-\vect{v}^1\Vert,
	\]
	where $\tensor{U}=\vect{u}^1\otimes\cdots\otimes\vect{u}^d$ and $\tensor{V} = \vect{v}^1\otimes\cdots\otimes\vect{v}^d$.
	\end{lemma}
	\begin{proof}
	For proving the upper bound, apply the triangle inequality to the telescoping sum
	\[
	\sum_{i=1}^d \vect{v}^1 \otimes \cdots \otimes \vect{v}^{i-1} \otimes (\vect{u}^i - \vect{v}^i) \otimes \vect{u}^{i+1} \otimes \cdots \otimes \vect{u}^d = \tensor{U} - \tensor{V}
	\]
	and exploit $\|\vect{u}^k - \vect{v}^k\| \le \|\vect{u}^1 -\vect{v}^1\|$ for all $k=1,\ldots,d$. The lower bound follows from
	\[
	 \| \tensor{U} - \tensor{V} \|^2 = 2 - 2 \langle \tensor{U}, \tensor{V} \rangle = 2 - 2 \prod_{k=1}^d \langle \vect{u}^k, \vect{v}^k \rangle \ge 2 - 2 \langle \vect{u}^1, \vect{v}^1 \rangle = \|\vect{u}^1 - \vect{v}^1\|^2,
	\]
	having used $0 < \langle \vect{u}^k, \vect{v}^k \rangle \le 1$ for sufficiently small $\epsilon$.
	\end{proof}
	
	The second one is the final piece of the puzzle. We prove it in \cref{proof_lem3}.
	\begin{lemma}\label{sec3:auxiliary_lemma}
	For sufficiently small $\epsilon>0$, we have for all $(\tuple{u},\tuple{v})\in D(\epsilon)$ with $\tuple{u}=(\vect{u}^1,\ldots,\vect{u}^d)$ and $\tuple{v}=(\vect{v}^1,\ldots,\vect{v}^d)$ that
		\[
		q(U) \geq {2^{-d/2}} \left(\frac{\Vert \vect{u}^1-\vect{v}^1\Vert}{2}\right)^{\Sigma-1},
		\]
	where $U$ is the matrix that depends on $\tuple{u}$ and $\tuple{v}$ as in \cref{eqn_nice_expr_terracini} {and $q$ is as in \cref{def_q}}.
	\end{lemma}
	
	Combining \cref{bound_tensor_norm,sec3:auxiliary_lemma} with \cref{eq3.30} we find
	\[
	J \geq {c} \int_{D(\epsilon)}\,\frac{1}{ \Vert \vect{u}^1-\vect{v}^1\Vert^\Sigma} \, \, \d{}\tuple{u}\,\d{}\tuple{v},
	\]
	{ where $c>0$ is some constant.}
	Note that the integrand in this equation only depends on $\vect{u}^1$ and $\vect{v}^1$. By definition of $D(\epsilon)$, for each $2\leq k \leq d$, and if we fix $\vect{u}^k$, the domain of integration of $\vect{v}^k$ contains the difference of two spherical caps of respective affine radii $\frac{9}{10} \| \vect{u}^1 - \vect{v}^1 \|$ and $\| \vect{u}^1 - \vect{v}^1 \|$. From \cref{cap_volume}, the volume of this difference of caps is greater than a constant times $\| \vect{u}^1 - \vect{v}^1 \|^{n_j-1}$.
	Therefore, if we keep $\vect{u}^1,\vect{v}^1 \in \Sp(\R^{n_1})$ constant and integrate over $\vect{u}^k, \vect{v}^k \in \Sp(\R^{n_k})$, $k=2,\ldots,d$, then we get
		\[
		J
		\geq c' \underset{\substack{\vect{u}^1,\vect{v}^1\in \Sp(\R^{n_1}),\\ \|\vect{u}^1-\vect{v}^1\|\leq\epsilon}}{\int} \frac{1}{\|\vect{u}^1-\vect{v}^1\|^{\Sigma - ((n_2-1)+\cdots+(n_d-1))}}\,\d{}\vect{u}^1\,\d{}\vect{v}^1,
		\]
	where $c'>0$ is a constant. Recall that $\Sigma = 1 + \sum_{k=1}^d (n_k - 1)$, so that
	\[
	J\geq
	c'\underset{\vect{u}^1\in \Sp(\R^{n_1})}{\int}\;\;
	\underset{\substack{\vect{v}^1\in \Sp(\R^{n_1}),\\ \|\vect{u}^1-\vect{v}^1\|\leq\epsilon}}{\int} \frac{1}{\|\vect{u}^1-\vect{v}^1\|^{n_1}}\,\d{}\vect{v}^1
	\,\d{}\vect{u}^1.
	\]
	{By rotational invariance, the inner integral does not depend on $\vect{u}^1$ and moreover for small $\epsilon$ projecting through the stereographic projection (which has a Jacobian bounded above and below by a positive constant close to its center) we conclude that, for some other constant $c''$,
	\[
	J\geq c''\underset{\substack{\vect{x}\in\R^{n_1-1},\\ \|\vect{x}\|\leq\epsilon/2}}{\int} \frac{1}{\|\vect{x}\|^{n_1}}\,d\vect{x}
	=c''\mathrm{vol}(\mathbb S(\R^{n_1-1}))\int_0^{\epsilon/2}\frac{r^{n_1-2}}{r^{n_1}}\,\d r
	=\infty.
	\]
	}
	This proves $J = \infty$, so that $\mean \kappa(\tensor{A})=\infty$ for tensors of rank bounded by $2$, constituting a proof of \cref{thm_main}.

	\section{The average condition number: from rank 2 to higher ranks}\label{sec:proof_main_higher_ranks}
	Having established that the average condition number of tensor rank decomposition of rank $2$ tensors is infinite, we extend this result to higher ranks. That is, we will prove \cref{thm_main_for_higher_ranks}. As before, we abbreviate
	\(\Var S:= \Var S_{n_1,\ldots,n_d}\), \(\sigma_r:=\sigma_{r,n_1,\ldots,n_d}\), \(\Var N_r:=\Var N_{r,n_1,\ldots,n_d}\), and \(\Var M_r:=\Var M_{r,n_1,\ldots,n_d}.\)
	
	
	We proceed with an observation that is of independent interest.
	
	\begin{lemma} \label{lem_kappa_inc}
	Let $\tensor{A} = \sum_{i=1}^{r} \tensor{A}_i$ and $\tensor{B} = \sum_{i=1}^s \tensor{B}_i$ be
	$n_1 \times \cdots \times n_d$ tensors, where the $\tensor{A}_i$ and $\tensor{B}_i$ are rank-$1$ tensors. If $\tensor{A}+\tensor{B}\in\sigma_{r+s;n_1,\ldots,n_d}$ is $(r+s)$-identifiable, then we have
	\[
	 \kappa( \tensor{A} + \tensor{B} ) \ge \max\{ \kappa( \tensor{A} ), \kappa( \tensor{B} ) \}.
	\]
	\end{lemma}
	\begin{proof}
	{First we observe that $\tensor{A}$ is $r$-identifiable, and $\tensor{B}$ is $s$-identifiable. Indeed, if  the tensor $\tensor{C} = \tensor{A}+\tensor{B}$ is $(r+s)$-identifiable, then the unique set $C$ of cardinality $|C| \le r$ consisting of rank-$1$ tensors summing to $\tensor{C}$ is $C = \{\tensor{A}_1, \ldots, \tensor{A}_r, \tensor{B}_1, \ldots, \tensor{B}_s\}$. If $\tensor{A}$ had an alternative decomposition $\{\tensor{A}_1', \ldots, \tensor{A}_{r'}'\}$, potentially of a shorter length $r' \le r$, then $\{\tensor{A}_1', \ldots, \tensor{A}_{r'}', \tensor{B}_1, \ldots, \tensor{B}_s\}$ would be an alternative decomposition of $\tensor{C}$. Hence, $\{\tensor{A}_1', \ldots, \tensor{A}_{r'}'\}$ needs to equal $\{\tensor{A}_1, \ldots, \tensor{A}_{r}\}$, so that $\tensor{A}$ is $r$-identifiable. By symmetry, the result for $\tensor{B}$ follows.}
	For all $i$, let $U_i$ be a matrix with orthonormal columns that span $\mathrm{T}_{\tensor{A}_i} \Var S_{n_1,\ldots,n_d}$, and $V_i$ be a matrix with orthonormal columns that span $\mathrm{T}_{\tensor{B}_i} \Var S_{n_1,\ldots,n_d}$. Consider the matrices $U=[U_1,\ldots,U_r]$ and $V=[V_1,\ldots,V_s]$. By \cref{characterization_CN} we have
	\[
	\kappa(\tensor{A}) = \frac{1}{\varsigma_{\min}(U)}, \;\;
	\kappa(\tensor{B}) = \frac{1}{\varsigma_{\min}(V)}, \text{ and }\;
	\kappa(\tensor{A}+\tensor{B}) = \kappa(\tensor{A}, \tensor{B}) = \frac{1}{\varsigma_{\min}(\begin{bmatrix}U & V\end{bmatrix})}.
	\]
	The claim follows from standard interlacing properties of singular values; see \cite[Chapter 3]{HJ1990}.
	\end{proof}
	
	The next simple lemma is immediate.
	\begin{lemma}\label{lemma_phi}
	Consider the map
	 \(
	  \phi : \sigma_2 \times \Var S^{\times (r-2)} \to \sigma_r,\, (\tensor{B}, \tensor{A}_1,\ldots,\tensor{A}_{r-2}) \mapsto \tensor{B} + \sum_{i=1}^{r-2} \tensor{A}_i.
	 \)
	The following holds.
	\begin{enumerate}
	\item For $r > 2$, we have $\phi(\sigma_2\times \Var S^{\times (r-2)}) = \sigma_r$.
	\item Let $\tensor{A}\in\sigma_{r}$ be $r$-identifiable. Then, $\vert \phi^{-1}(\tensor{A})\vert = (r-2)!\cdot\binom{r}{2}$.
	\end{enumerate}
	\end{lemma}
	
	Finally, the next lemma is the key to \cref{thm_main_for_higher_ranks}, providing a lower bound for the Jacobian determinant of $\phi$ in a special open subset of $\sigma_2 \times \Var S^{\times (r-2)}$. We postpone its proof to \cref{sec:proof_important_lemma}.
	
	\begin{lemma}\label{important_lemma}
	On top of \cref{assumpt1} we assume that $\sigma_{r-2; n_1-2,\ldots,n_d-2}$ is generically complex identifiable.
	Then, there are constants {$\mu,\epsilon, \nu_1, \ldots, \nu_{r-2} >0$} depending only on {$r,n_1,\ldots,n_d$} with the following property: For all~$\tensor{B}\in\cal{N}_{2}$ there exists a tuple $(\tensor{A}_1,\ldots,\tensor{A}_{r-2})\in \Var S^{\times (r-2)}$ {with $\| \tensor{A}_i \| = \nu_i$} and
	\[
	\underset{\substack{(\tensor{A}_1',\ldots,\tensor{A}_{r-2}')\in \Var S^{\times (r-2)},\\ \Vert \tensor{A}_i-\tensor{A}_i'\Vert <\epsilon}}{\inf} \, \Jac{\phi}{\tensor{B},\tensor{A}_1',\ldots,\tensor{A}_{r-2}'} >\mu,
	\]
	where $\phi$ is as in \cref{lemma_phi}.
	\end{lemma}
	\begin{rem}\label{rmk:l44}
	Given any $\tensor{B}\in\sigma_2$, by taking a sequence $\tensor{B}^{(i)}\subseteq \cal{N}_{2}$ converging to $\tensor{B}$ one can generate the corresponding sequences $\tensor{A}_1^{(i)},\dots,\tensor{A}_{r-2}^{(i)} \in \Var{S}$ from Lemma \ref{important_lemma}. Now, by compactness we can find an accumulation point $\tensor{A}_1,\dots,\tensor{A}_{r-2} \in \Var{S}$. Since $\mathrm{Jac}(\phi)$ is continuous and hence uniformly continuous when restricted to a compact set, by choosing small enough $\epsilon$ we can assure that for all $\tensor{B}'$, $\|\tensor{B}-\tensor{B'}\|\leq\epsilon$ and for all $\tensor{A}_i'$, $\|\tensor{A}_i-\tensor{A}_i'\|\leq\epsilon$, we have $\Jac{\phi}{\tensor{B}',\tensor{A}_1',\ldots,\tensor{A}_{r-2}'} > \frac{\mu}{2}$, where $\epsilon$ and $\mu$ do not depend on $\tensor{B}$.
	\end{rem}
	
	Now we prove \cref{thm_main_for_higher_ranks}.
	\begin{proof}[Proof of \cref{thm_main_for_higher_ranks}]
	Recall the surjective map $\phi : \sigma_2\times {\Var{S}^{\times (r-2)}} \to \sigma_r$ from \cref{lemma_phi}. From \cref{thm_main} and the fact that $\kappa(\tensor{B})=\kappa(t\tensor{B})$ for $t>0$, there exists a tensor $\tensor{B}\in \sigma_2$ such that for every $\delta>0$ we have
		\[
		\int_{\|\tensor{B}'-\tensor{B}\|<\delta,\, \tensor{B}'\in\mathcal{N}_2}\kappa(\tensor{B}')\,\d\tensor{B}'=\infty.
		\]
	From \cref{important_lemma,rmk:l44}, there exist tensors $\tensor{A}_1,\ldots,\tensor{A}_{r-2} \in \Var{S}$ such that
		\[
		\Jac{\phi}{\tensor{B}',\tensor{A}_1',\ldots,\tensor{A}_{r-2}'}> \frac{\mu}{2}
		\]
	for all  $\tensor{B}',\tensor{A}_1',\ldots,\tensor{A}_{r-2}'$ such that $\|\tensor{B}'-\tensor{B}\|<\epsilon,\|\tensor{A}'_i-\tensor{A}_i\|<\epsilon$, and $\tensor{B}'\in\mathcal{N}_2$. Let $\mathcal{U}\subseteq \mathcal{N}_2\times \Var{S}^{\times r-2}$ be the set of all $\tensor{B}',\tensor{A}_1',\ldots,\tensor{A}_{r-2}'$ satisfying the foregoing conditions. From \cref{lem_kappa_inc}, we have
	\begin{align*}
			&\int_{\mathcal{U}}\Jac{\phi}{\tensor{B}',\tensor{A}_1',\ldots,\tensor{A}_{r-2}'}\kappa(\tensor{B}'+\tensor{A}_1'+\cdots+\tensor{A}_{r-2}') \, \d\tensor{B}'  \d\tensor{A}_1'\cdots  \d\tensor{A}_{r-2}'\\
			&\geq\int_{\mathcal{U}}\Jac{\phi}{\tensor{B}',\tensor{A}_1',\ldots,\tensor{A}_{r-2}'}\kappa(\tensor{B}')\,\d\tensor{B}'\d\tensor{A}_1'\cdots \d\tensor{A}_{r-2}'=\infty.
		\end{align*}
	Moreover, by \cref{important_lemma} and the inverse function theorem, by taking small enough $\epsilon$ and $\delta$ we can assume that $\phi|_\mathcal{U}$ is a diffeomorphism onto its image\footnote{This is different from $\phi|_{\phi^{-1}(\phi(\mathcal{U}))}$ being a diffeomorphism. Indeed, that mapping is in general finite-to-one.} and hence $\phi(\mathcal{U})$ is open. The coarea formula \cref{eqn_coarea} thus applies yielding
	\begin{multline*}
		\int_{\tensor{A}\in \phi(\Var{U})}\kappa(\tensor{A})\, \d\tensor{A}
		= \\ \int_{\mathcal{U}}\Jac{\phi}{\tensor{B}',\tensor{A}_1',\ldots,\tensor{A}_{r-2}'}\kappa(\tensor{B}'+\tensor{A}_1',\ldots,\tensor{A}_{r-2}')\, \d\tensor{B}' \d\tensor{A}_1'\cdots  \d\tensor{A}_{r-2}'
		= \infty.
		\end{multline*}
		The theorem follows since $\phi(\mathcal{U})\subseteq\sigma_r$.
	\end{proof}
	
	\section{The angular condition number of tensor rank decomposition}\label{sec:angular_CN}
	
	In this section we prove \cref{thm_main2}. As in the previous section, to ease notation, we abbreviate $\Var M_2:= \Var M_{2;n_1,\ldots,n_d}$, $\Var N_2:= \Var N_{2;n_1,\ldots,n_d}$, $\Var S_2:= \Var S_{2;n_1,\ldots,n_d}$, and $\sigma_2:= \sigma_{2;n_1,\ldots,n_d}$.
	
	\subsection{A characterization of the angular condition number as a singular value}
	
	We first derive a formula for the angular condition number in terms of singular values, similar to the one from \cref{characterization_CN}.
	Recall from \cref{def_kappa_ang} that the angular condition number for rank $r=2$ is
	\begin{equation*}
	\Kang(\tensor{A}) := \lim_{\epsilon \to 0}\;\; \sup_{\substack{\Vert \Delta\tensor{A}\Vert <\epsilon,\\ \tensor{A}+\Delta\tensor{A}\in\sigma_{2}}} \; \frac{\Vert ((p\times p)\circ\Phi^{-1}_\tuple{a})(\tensor{A}) - ((p\times p)\circ\Phi^{-1}_\tuple{a})(\tensor{A}+\Delta\tensor{A}) \Vert}{\Vert \Delta\tensor{A}\Vert},
	\end{equation*}
	where $p: \R^{n_1\times \cdots\times n_d} \to \mathbb{S}(\R^{n_1\times \cdots\times n_d})$ is the canonical projection onto the sphere and where $\Phi^{-1}_\tuple{a}$ is a local inverse of $\Phi:\Var S \times \Var S \to \sigma_2$
	at $\tuple{a}\in\Var{S}^{\times 2}$ with $\tensor{A}=\Phi(\tuple{a})$. As before, the value of $\Kang$ on $\sigma_2\setminus\Var{N}_2$ is not relevant for our analysis, so we do not specify it.
	
	\begin{prop}\label{prop:cnangular}
	Under \cref{assumpt1}, let {$\tensor{A} = \lambda \,\vect{u}^1\otimes \cdots \otimes \vect{u}^d + \mu \,\vect{v}^1\otimes \cdots \otimes \vect{v}^d \in \Var{N}_2$}, where for $1\leq k\leq d$ we have $\vect{u}^k,\vect{v}^k\in\mathbb{S}(\R^{n_k})$. Recall from \cref{def_L_and_M} the definitions of the matrices $M$ and~$L$, associated to $\tensor{A}$. The following equality holds:
		\[
		\Kang(\tensor{A})=\frac{1}{\varsigma_{\min}((\Id-MM^\dagger)L)},
		\]
	as far as the right--hand term is finite.
	\end{prop}
	\begin{proof}
	By \cref{prop:delotropaper}, any local inverse $\Phi^{-1}_\tuple{a}$ is differentiable at $\tensor{A} = \Phi(\tuple{a}) \in \Var N_2$. The projection $p$ is also differentiable, so that
			\[
			\Kang(\tensor{A}) = \Vert \deriv{(p^{\times 2} \circ\Phi^{-1}_\tuple{a})}{\tensor{A}} \Vert_2,
			\]
	where $\Vert \cdot \Vert_2$ is the spectral norm from \cref{def_spectral_norm}. We compute this norm.
	
	Let $\dot{\tensor{A}} \in \Tang{\tensor{A}}{\Var N_2}$ and $(\dot{\tensor{A}}_1, \dot{\tensor{A}}_2) = \deriv{\Phi^{-1}_\tuple{a}}{\tensor{A}}(\dot{\tensor{A}})$. Then, by linearity of the derivative, we have $\dot{\tensor{A}} = \dot{\tensor{A}}_1 + \dot{\tensor{A}}_2$. Furthermore, for $i=1,2$, the derivative $\deriv{p}{\tensor{A}_i}$ is the orthogonal projection onto the orthogonal complement of $\tensor{A}_i$ in $\R^\Pi$. According to this we decompose $\dot{\tensor{A}}_1$ and $\dot{\tensor{A}}_2$ as
		\begin{align*}
		\dot{\tensor{A}}_1 &= \dot{\tensor{A}}_1^\perp + \dot{\lambda} \tensor{U}, \text{ where }  \tensor{U} = \frac{\tensor{A}_1}{\Vert \tensor{A}_1\Vert} \text{ and } \dot{\tensor{A}}_1^\perp\in (\tensor{A}_1)^\perp,\\
		\dot{\tensor{A}}_2 &= \dot{\tensor{A}}_2^\perp + \dot\mu \tensor{V}, \text{ where }  \tensor{V} = \frac{\tensor{A}_2}{\Vert \tensor{A}_2\Vert} \text{ and } \dot{\tensor{A}}_2^\perp \in (\tensor{A}_2)^\perp.
	\end{align*}
	Then, we have $\deriv{(p^{\times2} \circ\Phi^{-1}_\tuple{a})}{\tensor{A}}(\dot{\tensor{A}}) = (\dot{\tensor{A}}_1^\perp/\|\tensor{A}_1\|, \dot{\tensor{A}}_2^\perp/\|\tensor{A}_2\|)$ and, consequently,
	\begin{equation}\label{eq_derivative_kang}
	\Vert \deriv{(p^{\times 2} \circ\Phi^{-1}_\tuple{a})}{\tensor{A}}(\dot{\tensor{A}})\Vert = \sqrt{\frac{\Vert \dot{\tensor{A}}_1^\perp \Vert^2}{\Vert \tensor{A}_1\Vert^2}  +  \frac{\Vert\dot{\tensor{A}}_2^\perp\Vert^2}{\Vert \tensor{A}_2\Vert^2}}.
	\end{equation}
	Recall from \cref{def_L_and_M} the matrices $L=\begin{bmatrix} \lambda L_1 & \mu L_2\end{bmatrix}$ and $M = \begin{bmatrix} \tensor{U} & \tensor{V} \end{bmatrix}$. We can find vectors $\vect{x}_1,\vect{x}_2 \in \R^{\Sigma-1}$  with
	$\dot{\tensor{A}}_1^\perp = \lambda L_1 \vect{x}_1$ and $\dot{\tensor{A}}_2^\perp = \mu L_2 \vect{x}_2$, and such that $\Vert \dot{\tensor{A}}_1^\perp\Vert = \lambda \Vert \vect{x}_1\Vert$ and $\Vert \dot{\tensor{A}}_2^\perp\Vert = \mu \Vert \vect{x}_2\Vert$. Observe that $\lambda = \Vert \tensor{A}_1\Vert$ and $\mu = \Vert \tensor{A}_2\Vert$. This yields
	\begin{equation}\label{eq5.1}
	\Vert \deriv{(p^{\times 2} \circ\Phi^{-1}_\tuple{a})}{\tensor{A}}(\dot{\tensor{A}})\Vert = \sqrt{\Vert \vect{x}_1\Vert^2 + \Vert\vect{x}_2\Vert^2}.
	\end{equation}
	Writing
	\[
	\dot{\tensor{A}} = \dot{\tensor{A}}_1 +  \dot{\tensor{A}}_2 =  L\begin{bmatrix} \vect{x}_1 \\ \vect{x}_2 \end{bmatrix} + M \begin{bmatrix} \dot{\lambda} \\ \dot{\mu} \end{bmatrix}, \text { we get }(\Id - MM^\dagger)\dot{\tensor{A}} = (\Id - MM^\dagger)L\begin{bmatrix} \vect{x}_1 \\ \vect{x}_2 \end{bmatrix}.
	\]
	Since we are assuming that $(\Id - MM^\dagger)L$ is injective (for $\varsigma_{\min}((\Id-MM^\dagger)L)\neq0$), it has a left inverse and we can write
	\begin{equation}\label{eq5.2}\left((\Id - MM^\dagger)L\right)^\dagger(\Id - MM^\dagger)\dot{\tensor{A}} = \begin{bmatrix} \vect{x}_1 \\ \vect{x}_2 \end{bmatrix}.
	\end{equation}
	Combining \cref{eq5.1} and \cref{eq5.2} we see that
	\begin{align*}
	\Kang(\tensor{A}) =& \Vert \left((\Id - MM^\dagger)L\right)^\dagger(\Id - MM^\dagger)\Vert_2\\=&\Vert \left((\Id - MM^\dagger)L\right)^\dagger\Vert_2\\=&\varsigma_{\min}((\Id - MM^\dagger)L)^{-1},
	\end{align*}
	the second equality from $ \left(PL\right)^\dagger P= \left(PL\right)^\dagger$, which is a basic property of the Moore--Penrose pseudoinverse holding for any orthogonal projector $P$. This finishes the proof.
	\end{proof}
	
	\subsection{Proof of Theorem \ref{thm_main2}}\label{sec:kangexp}
	Now comes the actual proof of \cref{thm_main2}. {Proceeding in exactly the same way as in \cref{sec_step1_parameterization}} and using \cref{prop:cnangular}, we get
	\begin{equation}\label{eq5.3}
	\mathbb{E}\, \Kang(\tensor{A}) = \frac{1}{2^{2d-1}C_2}\int_{((0,\infty) \times \Var{P})^{\times 2}}\frac{\mathrm{vol}(Q) \, e^{-\frac{\Vert \lambda  \tensor{U} + \mu  \tensor{V}\Vert^2}{2}} }{\varsigma_{\min}((\Id-MM^\dagger)L) }\, \d{}\lambda\,\d{}\tuple{u}\,\d{}\mu\,\d{}\tuple{v},
	\end{equation}
	where {$C_2 = C_{2;n_1,\ldots,n_d}$ is as in \cref{def_gaussian_tensor}}, $\Var{P}$ is as in \cref{def_P}, $Q = \begin{bmatrix} L & M \end{bmatrix}$ is as in \cref{eqn_def_Q}, {the volume $\mathrm{vol}$ is as in \cref{def_vol_U},} and
	\[
	\tuple{u} = (\sten{u}{}{1},\ldots, \sten{u}{}{d}), \;
	\tuple{v} = (\sten{v}{}{1},\ldots, \sten{v}{}{d}), \;
	\tensor{U} = \sten{u}{}{1} \otimes \cdots \otimes \sten{u}{}{d},\; \text{ and }
	\tensor{V} = \sten{v}{}{1} \otimes \cdots \otimes \sten{v}{}{d},
	\]
	is as in \cref{def_U_V}, so that $\tensor{A} = \lambda \tensor{U} + \mu \tensor{V}$.
	{Next, we relate $\mathrm{vol}(Q)$ to the volume of $(\Id - M M^\dagger)L$.}
	\begin{lemma}\label{lem:la}
	We have
		\(
		\mathrm{vol}(Q)
		= \mathrm{vol}(M)\,\mathrm{vol}((\Id-MM^\dagger)L).
		\)
	\end{lemma}
	\begin{proof}
	Let $Q^\perp$ be a matrix whose columns contain an orthonormal basis for the orthogonal complement of the column span of $Q$. Then, from the definition,
	{\small{
		\[
		\mathrm{vol}(Q)
		= \mathrm{vol}\bigl( \begin{bmatrix} Q & Q^\perp \end{bmatrix} \bigr)
		=\bigl|\det \bigl( \begin{bmatrix} M & L & Q^\perp \end{bmatrix} \bigr)\bigr|
		=\left|\det\left(
		\begin{bmatrix} M & L & Q^\perp \end{bmatrix}
		\begin{bmatrix}
		\Id&-M^\dagger L&0\\0&\Id&0\\0&0&\Id
		\end{bmatrix}\right)\right|,
		\]
	}}
	\noindent where in the last step we just multiplied by a matrix whose determinant is $1$. Performing the inner multiplication we then get
		\[
		\mathrm{vol}(Q)= \bigl|\det\bigl(\begin{bmatrix} M & (\Id-MM^\dagger)L & Q^\perp \end{bmatrix} \bigr)\bigr|
		= \mathrm{vol}\bigl( \begin{bmatrix} M & (\Id-MM^\dagger)L \end{bmatrix} \bigr).
		\]
	{These two blocks are mutually orthogonal, since $(\Id-MM^\dagger)$ is the projection on the orthogonal complement of the span of $M$, and hence the volume is the product of the volumes corresponding to each block. The assertion follows.}
	\end{proof}

	We use \cref{lem:la} to rewrite \cref{eq5.3} as
	{\small{
	\begin{align}
	\nonumber \mathbb{E}\, \Kang(\tensor{A})
	&= \frac{1}{2^{2d-1}C_2}\int_{((0,\infty) \times \Var{P})^{\times 2}}\frac{\mathrm{vol}(M)\,\mathrm{vol}((\Id-MM^\dagger)L)}{ \varsigma_{\min}((\Id-MM^\dagger)L) }\,  e^{-\frac{\Vert \lambda  \tensor{U} + \mu  \tensor{V}\Vert^2}{2}}\, \d{}\lambda\,\d{}\tuple{u}\,\d{}\mu\,\d{}\tuple{v}\\
	&= \frac{1}{2^{2d-1}C_2}\int_{((0,\infty) \times \Var{P})^{\times 2}}\,\mathrm{vol}(M) \; q((\Id-MM^\dagger)L)\,  e^{-\frac{\Vert \lambda  \tensor{U} + \mu  \tensor{V}\Vert^2}{2}}\, \d{}\lambda\,\d{}\tuple{u}\,\d{}\mu\,\d{}\tuple{v},\label{eq5.4}
	\end{align}
	}}
	where $q$ is as in \cref{def_q}.
	Recall from \cref{def_M} that $M$ is independent of $\lambda$ and $\mu$. We first compute the integral over $\lambda, \mu$ using the next lemma. We prove the lemma in \cref{proof_lem:inner}.
	\begin{lemma}\label{lem:inner}
	Let $L_1,L_2$ be the matrices defined as in \cref{def_L_i}, such that $L=\begin{bmatrix} \lambda L_1 & \mu L_2\end{bmatrix}$. Let
		\[
		J_\mathrm{inner}=\int_{(0,\infty)^2} q((\Id-MM^\dagger)L)\;e^{-\frac{\Vert \lambda  \tensor{U} + \mu  \tensor{V}\Vert^2}{2}}\, \d{}\lambda\,\d{}\mu\,.
		\]
		Then,
		\[
		J_\mathrm{inner} = 2^{\frac{2\Sigma-3}{2}} \Gamma\left(\frac{2\Sigma-1}{2}\right)\int_{0}^{\frac{\pi}{2}}\frac{q\left((\Id-MM^\dagger)\begin{bmatrix} \cos(\theta)L_1 &  \sin(\theta)L_2\end{bmatrix}\right)}{{\|\cos(\theta)\tensor{U}+\sin(\theta)\tensor{V}\|^{2\Sigma-1}}} \d{}\theta.
		\]
		\end{lemma}
	Inserting the results from this lemma into \cref{eq5.4}, we get
	\begin{equation*}
	\mean \Kang(\tensor{A}) = \frac{2^{\frac{2\Sigma-1}{2} - 2d}\Gamma\left(\frac{2\Sigma-1}{2}\right)}{C_2}\,J_\mathrm{outer}, 
	\end{equation*}
	where
	\[
	J_\mathrm{outer} := \int_{ \Var{P}^{\times 2}}\,\int_{0}^{\frac{\pi}{2}} \frac{\mathrm{vol}(M)\,q\left((\Id-MM^\dagger)\begin{bmatrix}\cos(\theta) L_1 & \sin(\theta) L_2\end{bmatrix}\right)}{{\|\cos(\theta) \tensor{U} + \sin(\theta)\tensor{V}\|^{2\Sigma-1}}}\, \d{}\theta\, \d{}\tuple{u}\,\d{}\tuple{v}.
	\]
	In the remaining part of this section we show that $J_\mathrm{outer}$ is bounded by a constant, which would conclude the proof. We do this by giving a sequence of upper bounds. We have no hope of providing sharp bounds, so rather than keeping track of all the constants, we will exploit the following definition for streamlining the proof.
	\begin{dfn}
			For $A,B\in[0,\infty]$ we will write $A \preceq B$ if $B\in\R$ implies $A\in\R$. That is, $A\preceq B$ is an equivalent statement to ``$B<\infty \Rightarrow A<\infty$''.
	\end{dfn}
	
	First, note that $\mathrm{vol}(M) = \sqrt{1-\langle \tensor{U},\tensor{V}\rangle^2}$, so that
	\[
	J_\mathrm{outer} = {\int_{ \Var{P}^{\times 2} }\int_{0}^{\frac{\pi}{2}} } \frac{\sqrt{1-\langle \tensor{U},\tensor{V}\rangle^2}\,q\left((\Id-MM^\dagger)\begin{bmatrix}\cos(\theta)  L_1 & \sin(\theta) L_2\end{bmatrix}\right)}{{\|\cos(\theta)  \tensor{U} + \sin(\theta) \tensor{V}\|^{2\Sigma-1}}}\, \d{}\theta\d{}\tuple{u}\, \d{}\tuple{v}.
	\]
	Next, we exploit the symmetry of $\Sp(\R^{n_1})$ and transform $\vect{v}^1 \mapsto -\vect{v}^1$. This transformation flips the sign of $\tensor{V}$, but the value of $q$ is not affected. Indeed, the matrix $\Id - M M^\dagger$ still projects onto $\mathrm{span}(\tensor{U}, \tensor{V})^\perp = \mathrm{span}(\tensor{U}, -\tensor{V})^\perp$, and
	$L_2$ is transformed into $L_2 D$, where $D$ is a diagonal matrix with some pattern of $\pm1$ on the diagonal. Since $\left[\begin{smallmatrix}I & \\ & D\end{smallmatrix}\right]$ is an orthogonal transformation, the singular values do not change. Thus, we obtain
	\[
	J_\mathrm{outer} = {\int_{ \Var{P}^{\times 2} }\int_{0}^{\frac{\pi}{2}} } \frac{\sqrt{1 - \langle \tensor{U},\tensor{V}\rangle^2}\,q\left((\Id-MM^\dagger)\begin{bmatrix}\cos(\theta) L_1 & \sin(\theta) L_2\end{bmatrix}\right)}{{\|\cos(\theta) \tensor{U} - \sin(\theta) \tensor{V}\|^{2\Sigma-1}}}\, \d{}\theta\, \d{}\tuple{u}\, \d{}\tuple{v}.
	\]
	The next lemma is proved in \cref{proof_lem:cotaq}.
	\begin{lemma}\label{lem:cotaq}
	Let $\theta\in[0,\tfrac{\pi}{2}]$ and fix $\theta,\tuple{u}$ and $\tuple{v}$. There is a constant $K>0$, depending only on $n_1,\ldots,n_d$ and $d$, such that
	\begin{equation}\label{eqcotaq}
	q\left((\Id-MM^\dagger)\begin{bmatrix}\cos(\theta)  L_1 & \sin(\theta) L_2\end{bmatrix}\right)\leq K \|\tensor{U} - \tensor{V}\|^{\Sigma-1}.
	\end{equation}
	\end{lemma}
	
	The lemma implies
	\begin{equation}\label{eq5.5}
	J_\mathrm{outer} \preceq {\int_{ \Var{P}^{\times 2} }\int_{0}^{\frac{\pi}{2}} } \frac{\sqrt{1 - \langle \tensor{U},\tensor{V}\rangle^2}\,\|\tensor{U}- \tensor{V}\|^{\Sigma-1}}{{\|\cos(\theta)  \tensor{U}-\sin(\theta) \tensor{V}\|^{2\Sigma-1}}}\, \d{}\theta \, \d{}\tuple{u} \, \d{}\tuple{v}.
	\end{equation}
	For bounding the integral over $\theta$ we need the next lemma, which we prove in \cref{proof_lem:integralauxiliar}.
	\begin{lemma}\label{lem:integralauxiliar}
	Let $a>1,p\geq1$. There exists a constant $K>0$, depending only on $a$, such that for any unit vectors $\vect{x},\vect{y}\in \Sp{}(\R^p)$, $\vect{x}\neq \vect{y}$, we have
		\[
		\int_{0}^{\frac{\pi}{2}}\frac{1}{{\|\cos(\theta)  \vect{x}-\sin(\theta) \vect{y}\|^{a}}} \d{}\theta\leq \frac{K}{\|\vect{x}-\vect{y}\|^{a-1}}.
		\]
	\end{lemma}
	
	Applying this lemma to \cref{eq5.5}, we obtain
	\[
	J_\mathrm{outer} \preceq
	\int_{ \Var{P}^{\times 2} } \frac{\sqrt{1 - \langle \tensor{U},\tensor{V}\rangle^2}}{\|\tensor{U}- \tensor{V}\|^{\Sigma-1}}\, \d{}\tuple{u}\,\d{}\tuple{v}
	= \int_{ \Var{P}^{\times 2} } \frac{\sqrt{1-\langle \tensor{U},\tensor{V}\rangle} \sqrt{1 + \langle \tensor{U},\tensor{V}\rangle}}{\|\tensor{U} - \tensor{V}\|^{\Sigma-1}}\, \d{}\tuple{u}\,\d{}\tuple{v}.
	\]
	Writing $\|\tensor{U} - \tensor{V}\| = \sqrt{2}\sqrt{1 - \langle \tensor{U},\tensor{V}\rangle}$, we arrive at
	\[
	J_\mathrm{outer} \preceq
	\int_{ \Var{P}^{\times 2} } \frac{\sqrt{1 + \langle \tensor{U},\tensor{V}\rangle} }{ \sqrt{1 - \langle \tensor{U}, \tensor{V} \rangle}^{\Sigma-2}}\, \d{}\tuple{u}\,\d{}\tuple{v}
	\preceq \int_{\Var{P}^{\times 2}} \frac{1}{\sqrt{1 - \langle \tensor{U},\tensor{V}\rangle}^{\,\Sigma-2}}\, \d{}\tuple{u}\, \d{}\tuple{v}.
	\]
	By orthogonal invariance, we may fix $\vect{u}^k\in \Sp(\R^{n_k})$ to be
	 $\vect{u}^k = (1,0,\ldots,0)$, and integrate the constant function $1$ over one copy of $\Sp{}(\R^{n_1}) \times \cdots \times \Sp{}(\R^{n_d})$. Ignoring the product of volumes $\prod_{k=1}^d\mathrm{vol}(\Sp{}(\R^{n_k}))$ we have
	\[
	J_\mathrm{outer} \preceq \int_{\Sp{}(\R^{n_1}) \times \cdots \times \Sp{}(\R^{n_d})} \frac{1}{\sqrt{1 - (\vect{v}^1)_1\cdots (\vect{v}^d)_1}^{\,\Sigma-2}}\, \d{}\tuple{v}.
	\]
	Now, this spherical integral is particularly simple because the integrand depends uniquely on one of the components of each vector. One can thus transform each integral in a sphere into an integral in an interval (see for example \cite[Lemma 1]{BMOC}) getting:
			\[
			J_\mathrm{outer}\preceq \int_{t_1,\ldots,t_d\in[-1,1]}\frac{(1-t_1^2)^{\frac{n_1-1}{2}-1}\cdots(1-t_d^2)^{\frac{n_d-1}{2}-1}}{\sqrt{1-t_1\cdots t_d}^{\Sigma-2}}\,\d{}t_1\cdots\d{}t_d.
			\]
	For this last integral we consider the partition of the cube $[-1,1]^d$ into $2^d$ pieces corresponding to the different signs of the coordinates. In the pieces where the number of negative coordinates is odd, the denominator of the integrand is bounded below by $1$ and thus the whole integrand is also bounded above by $1$. Hence it suffices to check that the integral in the rest of the pieces is bounded. Assume now that $t_{i_1},\ldots,t_{i_k}$ with $k\geq 2$ even are the negative coordinates in some particular piece of the partition. The mapping that leaves all coordinates fixed but maps $t_{i_{k-1}} \mapsto -t_{i_{k-1}}$ and $t_{i_{k}} \mapsto -t_{i_{k}}$ preserves the integrand and moves the domain to another piece of the partition with $k-2$ negative coordinates. This process can then be repeated until none of the coordinates is negative. All in one, we have
	\[
				J_\mathrm{outer}\preceq\int_{t_1,\ldots,t_d\in[0,1]}\frac{(1-t_1^2)^{\frac{n_1-1}{2}-1}\cdots(1-t_d^2)^{\frac{n_d-1}{2}-1}}{\sqrt{1-t_1\cdots t_d}^{\Sigma-2}}\,\d{}t_1\cdots\d{}t_d.
	\]
				The change of variables $t_k=\cos(\theta_k)$ for $1\leq k\leq d$ converts this last integral into
	\begin{equation}\label{eq5.7}
				\int_{\theta_1,\ldots,\theta_d\in[0,\frac{\pi}{2}]}\frac{\sin(\theta_1)^{n_1-2}\cdots\sin(\theta_d)^{n_d-2}}{\sqrt{1-\cos(\theta_1)\cdots \cos(\theta_d)}^{\Sigma-2}}\,\d{}\theta_1 \cdots\d{}\theta_d.
	\end{equation}
	The next lemma is proved in \cref{proof_lem:cosenos}.
	\begin{lemma}\label{lem:cosenos}
	Let $d\geq1$ and $\theta_1,\ldots,\theta_d\in[0,\tfrac{\pi}{2}]$. Then,
	\(
				\cos(\theta_1)\cdots\cos(\theta_d)\leq 1-\frac{\theta_1^2+\cdots+\theta_d^2}{7d}.
	\)
	\end{lemma}
	Using the lemma and the inequality $\sin(\theta)<\theta$ on $0\leq \theta\leq \tfrac{\pi}{2}$, we find that the integral in~\cref{eq5.7} is bounded by a constant times the following integral:
				\[
				\int_{\theta_1,\ldots,\theta_d\in[0,\frac{\pi}{2}]}\frac{\theta_1^{n_1-2}\cdots\theta_d^{n_d-2}}{\sqrt{\theta_1^2+\cdots+ \theta_d^2}^{\,\Sigma-2}}\,\d{}\theta_1 \cdots \d{}\theta_d.
				\]
		Changing the name of the variables to $x_1,\ldots,x_d$ and integrating over the $d$-dimensional ball of radius $\tfrac{\pi}{2}\sqrt{d}$, which contains the domain $[0,\tfrac{\pi}{2}]^d$, we get a new upper bound for the last integral, which implies
				\[
				J_\mathrm{outer}\preceq
				\int_{x_1^2+\cdots+x_d^2\leq \frac{\pi^2d}{4}}\frac{x_1^{n_1-2}\cdots x_d^{n_d-2}}{\sqrt{x_1^2+\cdots+ x_d^2}^{\, \Sigma-2}}\,\d{}x_1 \cdots \d{}x_d.
				\]
	Recall that $\Sigma = 1 + \sum_{j=1}^d (n_j-1)$. By passing to polar coordinates we get
		\begin{align*}
			J_\mathrm{outer}&\preceq\int_{x\in\Sp(\R^d)}x_1^{n_1-2}\cdots x_d^{n_d-2}\int_0^{\frac{\pi\sqrt{d}}{2}}\frac{\rho^{d-1+\sum_{j=1}^d(n_j-2)}}{\rho^{-1 + \sum_{j=1}^d (n_j-1)}}\,\d{}\rho\,\d{}x_1\cdots \d{}x_d\\
			&\leq \mathrm{vol}(\Sp(\R^d)) \int_{0}^{\frac{\pi \sqrt{d}}{2}} 1
		   = \mathrm{vol}(\Sp(\R^d)) \frac{\pi\sqrt{d}}{2}  < \infty.
		\end{align*}
	This shows $J_\mathrm{outer} < \infty$ implying $\mean \Kang(\tensor{A}) <\infty$, finishing the proof of \cref{thm_main2}. \qed

	\section{Other random tensors: proof of \cref{cor:RIT}}\label{sec:otherrandom}
	{We demonstrate how our main results can be extended to many other distributions as well.
	
	Consider the first item of \cref{cor:RIT}. We assume that $\tensor{A}\in \sigma_{r;n_1,\ldots,n_d}$ has the density $\hat{\rho}$ and that there exists positive constants $c_1,c_2$ such that $c_1 \le \frac{\hat\rho}{\rho} \le c_2$, where $\rho$ is the density of a GIT. Then, for any measurable function $f(\tensor{A})$ we have
	{\small{
	$$\mean_{\tensor{A}\sim \hat{\rho}}\, \ f(\tensor{A}) = \int_{\sigma_{r;n_1,\ldots,n_d}}\, f(\tensor{A}) \, \hat{\rho}(\tensor{A})\, \mathrm{d}\tensor{A} \geq c_1 \int_{\sigma_{r;n_1,\ldots,n_d}}\, f(\tensor{A}) \, {\rho}(\tensor{A})\, \mathrm{d}\tensor{A} = c_1 \mean_{\tensor{A}\sim {\rho}}\, \ f(\tensor{A})$$
	}}
	and
	{\small{
	$$\mean_{\tensor{A}\sim \hat{\rho}}\, \ f(\tensor{A}) = \int_{\sigma_{r;n_1,\ldots,n_d}}\, f(\tensor{A}) \, \hat{\rho}(\tensor{A})\, \mathrm{d}\tensor{A} \leq c_2\int_{\sigma_{r;n_1,\ldots,n_d}}\, f(\tensor{A}) \, {\rho}(\tensor{A})\, \mathrm{d}\tensor{A} = c_2 \mean_{\tensor{A}\sim {\rho}}\, \ f(\tensor{A})$$
	}}
	Thus, $\mean_{\tensor{A}\sim \hat{\rho}}\, \ f(\tensor{A}) = \infty$ if and only if $\mean_{\tensor{A}\sim \rho}\, \ f(\tensor{A}) = \infty$. Replacing $f$ by $\kappa$ and $\Kang$ proves the first part of \cref{cor:RIT}.
	%
	
	By \cite[Proposition 4.4]{BV2017} $\kappa$ is invariant under multiplication of $\tensor{A}$ by a scalar. Therefore, the expected value of $\kappa$ for the Gaussian is equal to the expected value when $\tensor{A}$ is chosen uniformly in the unit ball, and also when $\tensor{A}$ is chosen uniformly in the unit sphere of the space of tensors. Namely, we have (see, e.g., \cite[Section 2.2.4]{BC2013})
	\[
	\mean_{\tensor A\in\sigma_{r;n_1,\ldots,n_d}:\|\tensor{A}\|\leq 1} \kappa(\tensor{A})
	=\mean_{\tensor A\in\sigma_{r;n_1,\ldots,n_d}:\|\tensor{A}\|= 1} \kappa(\tensor{A})
	=\mean_{\tensor A\text{ a GIT in }\sigma_{r;n_1,\ldots,n_d}}\kappa(\tensor{A}).
	\]
	This proves the second and third item of \cref{cor:RIT} for $\kappa$.
	
	For $\Kang$ we need the following lemma.
	\begin{lemma} \label{lem_kappa_ang_scal}
	If $\tensor{A}\in\sigma_{r;n_1,\ldots,n_d}$ is an $r$-nice tensor, then $\Kang(t\tensor{A})=\Kang(\tensor{A})/t$ for all $t>0$.
	\end{lemma}
	\begin{proof}
	Since $\tensor{A}$ is $r$-nice, we have $\Kang(\tensor{A}) = \Vert \deriv{(p^{\times r} \circ\Phi^{-1}_\tuple{a})}{\tensor{A}}\Vert$.
	Similar as for \cref{eq_derivative_kang} we can show
	$
	\Vert \deriv{(p^{\times r} \circ\Phi^{-1}_\tuple{a})}{\tensor{A}}(\dot{\tensor{A}})\Vert = \sqrt{\sum_{i=1}^r\, \Vert \tensor{A}_i\Vert^{-2}\,\Vert \mathrm{d}_{\tensor{A}_i}p\,\dot{\tensor{A}}_i \Vert^2},
	$
	where $\tensor{A}=\tensor{A}_1+\cdots+\tensor{A}_r$ is the CPD of $\tensor{A}$ and $\dot{\tensor{A}}=\dot{\tensor{A}}_1+\cdots+\dot{\tensor{A}}_r$ is the corresponding decomposition the tangent vector. The derivative $\mathrm{d}_{\tensor{A}_i}p$ is the orthogonal projection onto $\tensor{A}_i^\perp$ and independent of scaling. Moreover, $\sigma_{r;n_1,\ldots,n_d}$ is a cone and so $\Tang{\tensor{A}}\sigma_{r;n_1,\ldots,n_d}$ can be identified with $\Tang{t\tensor{A}}\sigma_{r;n_1,\ldots,n_d}$. This shows that after scaling the tensor $\tensor{A}$ we get $\Vert \deriv{(p^{\times r} \circ\Phi^{-1}_\tuple{a})}{t\tensor{A}}(\dot{\tensor{A}})\Vert = t^{-1} \Vert \deriv{(p^{\times r} \circ\Phi^{-1}_\tuple{a})}{\tensor{A}}(\dot{\tensor{A}})\Vert $ and hence $\Kang(t\tensor{A})=\Kang(\tensor{A})/t$.
	\end{proof}
	Now, we can prove the rest of \cref{cor:RIT}.
	Recall from \cref{def_gaussian_tensor} that the density of a GIT on  $\sigma_{r;n_1,\ldots,n_d}$ is $
	\rho(\tensor{A}):= (C_{r;n_1,\ldots,n_d})^{-1} \, e^{-\frac{\|\tensor{A}\|^2}{2}}$, where $C_{r;n_1,\ldots,n_d} = \int_{\sigma_{r;n_1,\ldots,n_d}}e^{-\frac{\|\tensor{A}\|^2}{2}}\, \d{}\tensor{A}$. Since our results for $\Kang$ are for rank-$2$ tensors, we put $r=2$ in the following. We also abbreviate $\sigma_2:=\sigma_{2;n_1,\ldots,n_d}$ and $C_2:=C_{2;n_1,\ldots,n_d}$. Then, using \cref{lem_kappa_ang_scal} we can integrate in polar coordinates to obtain
	\begin{align*}
	\mean_{\tensor A\text{ a GIT in }\sigma_{2}}\Kang(\tensor{A})
	=& \frac{1}{C_2}\int_{\tensor A\in \sigma_2}\Kang(\tensor A)e^{-\frac{\|\tensor A\|^2}{2}}  \d\tensor A
	\\=& \frac{1}{C_2}\int_0^\infty e^{-\frac{t^2}{2}}\int_{\tensor A\in \sigma_2:\|\tensor A\|=t}\Kang(\tensor A)  \d\tensor A\,\d t\\
	=& \frac{1}{C_2}\int_0^\infty e^{-\frac{t^2}{2}}\int_{\tensor A\in\mathbb S( \sigma_2)}t^{2\Sigma-1}\Kang(t\tensor A)  \d\tensor A\,\d t
	\\=&\frac{1}{C_r}\int_0^\infty t^{2\Sigma-2}e^{-\frac{t^2}{2}}\d t\int_{\tensor A\in\mathbb S( \sigma_2)}\Kang(\tensor A)  \d\tensor A.
	\end{align*}
	It follows immediately that the last integral is finite, proving that a randomly chosen $\tensor A\in\mathbb S(\sigma_2)$ has finite expected $\Kang$. Finally, if $\tensor A$ is chosen randomly in the unit ball in $\sigma_2$, the same argument shows that the expected value is again finite:
	\[
	\int_{\substack{\tensor A\in \sigma_{r;n_1,\ldots,n_d}:}{\|\tensor A\|\leq 1}}\Kang(\tensor A)  \d\tensor A
	=\int_0^1t^{2\Sigma-2}\int_{\tensor A\in\mathbb S( \sigma_2)}\Kang(\tensor A) \d\tensor A\,\d t<\infty.
	\]
	This finishes the proof of \cref{cor:RIT}.}

	\section{Numerical experiments}\label{sec:experiments}
	
	Having proved that the expected value of the condition number is infinite in most cases, we provide further computational evidence in support of \cref{conj_regular_cn}. To this end, a natural idea is to perform Monte Carlo experiments in a few of the unknown cases as in \cite{BV2018c}.
	
	Sampling GITs is hard in practice, as the defining polynomial equalities and inequalities of the semialgebraic set $\sigma_r = \sigma_{r;n_1,\ldots,n_d}$ of tensors of rank bounded by $r$ are not known in the literature.\footnote{See \cite[Chapter 7]{Landsberg2012} and the references therein for some results on equations of the algebraic closure of $\sigma_r$.} Nevertheless, there are a few cases that we can treat numerically. If $r = \frac{\Pi}{\Sigma}$ and the algebraic closure $\overline{\sigma_r}(\C)$ has $\dim \overline{\sigma_r}(\C) = \Pi$, a so-called \textit{perfect tensor space}, then $\sigma_r$ is an open subset of the ambient $\R^{n_1 \times \cdots \times n_d}$; see, e.g., \cite{Landsberg2012,BT2015}.
	
	{From \cref{remark_conditional_distr}, we can sample from the density $\rho$ on $\sigma_r$ via an acceptance--rejection method:} Randomly sample tensors $\tensor{A}$ from the density $e^{-\frac{\|\tensor{A}\|^2}{2}}$ on $\R^{n_1 \times \cdots \times n_d}$ until we find one that belongs to $\sigma_r$. While this scheme will yield tensors distributed according to the density $\rho$ on $\sigma_r$, it does not yield Gaussian \emph{identifiable} tensors in general. The reason is that most perfect tensor spaces are not (expected to be) generically $r$-identifiable \cite{HOOS2016}. Fortunately, there are a few known exceptions: matrix pencils ($\R^{n \times n \times 2}$ for all $n\ge2$), $\R^{5 \times 4 \times 3}$ and $\R^{3 \times 2 \times 2 \times 2}$ are proved to be generically complex $r$-identifiable for $r = \frac{\Pi}{\Sigma}$. By applying the acceptance--rejection method to these spaces, every sampled tensor is a GIT with probability $1$.

	For numerically checking if a random tensor $\tensor{A} \in \R^{n_1 \times \cdots \times n_d}$ in a perfect tensor space lies in~$\sigma_r$ with $r = \frac{\Pi}{\Sigma}$, we apply a homotopy continuation method to the square system of $\Pi$ equations
	\[
	 \tensor{A} - \sum_{i=1}^r \sten{a}{i}{1} \otimes \begin{bmatrix}1 \\ \sten{a}{i}{2} \end{bmatrix} \otimes \cdots \otimes \begin{bmatrix} 1 \\ \sten{a}{i}{d} \end{bmatrix} = 0,
	\]
	where the $\Pi = r \Sigma$ entries of the $\sten{a}{i}{k}$'s are treated as variables, and the $n_1 \times \cdots \times n_d$ tensor $\tensor{A}$ is the tensor to decompose. We generate a start system with one solution to track by randomly sampling the entries of the vectors $\sten{a}{i}{k}$ i.i.d.~from a real standard Gaussian distribution and then constructing the corresponding tensor $\tensor{A}_0$.
	Since $r = \frac{\Pi}{\Sigma}$ is the so-called \emph{generic rank} of tensors in perfect tensor spaces $\C^{n_1 \times \cdots \times n_d}$, the above system has at least one complex solution with probability $1$ as well. If we consider complex $r$-identifiable perfect tensor spaces at the generic rank, we can thus determine if $\tensor{A} \in \sigma_r$ by solving the square system and checking whether the unique solution is real. Assuming that we use a certified homotopy method such as alphaCertified \cite{HS2012}, this approach will correctly classify $\tensor{A}$ with probability $1$, thus not impacting the overall distribution produced by the acceptance--rejection scheme.
	
	We implemented the above scheme in Julia 1.0.3 using version 0.4.3 of the package HomotopyContinuation.jl \cite{BT2017}, employing the \texttt{solve} function with default parameter settings. We deem a solution real if the norm of the imaginary part is less than $10^{-8}$. Note that this package does not offer certified tracking; however, the failure rate observed in our experiments was very low, namely $0.0512498\%$---see \cref{tab_homo_sols}. For this reason, we are convinced that the distribution produced by the acceptance--rejection scheme is very close to the true distribution.
	
	We performed the following experiment for estimating the distribution of the condition numbers of GITs of generically complex $r$-identifiable tensors in perfect tensor spaces with $r = \frac{\Pi}{\Sigma}$, the complex generic rank. As explained above, we randomly sampled an element $\tensor{A}$ of $\R^{n_1 \times \cdots \times n_d}$ from the density $e^{-\frac{\|\tensor{A}\|^2}{2}}$ by choosing its entries i.i.d.~standard normally distributed. Then, we generated one random starting starting system and applied the \texttt{solve} function from HomotopyContinuation.jl for tracking the starting solution $\tensor{A}_0$ to the target $\tensor{A}$. If the final solution of the square system was real, we recorded both the regular and angular condition numbers at the CPD of $\tensor{A}$ computed via homotopy continuation. These computations were performed in parallel using $20$ computational threads until $100,000$ finite, nonsingular, real solutions and corresponding condition numbers were obtained. This experiment was performed on a computer system consisting of $2$ Intel Xeon E5-2697 v3 CPUs with $12$ cores clocked at 2.6GHz and 128GB main memory. Information about the sampling process via the acceptance--rejection method are summarized in \cref{tab_homo_sols}, and \cref{fig_distribution} visualizes the complementary cumulative distribution functions of the regular and angular condition numbers.
	
	\begin{table}
	\begin{tabular}{ccccccc}
	 \toprule
	\multirow{2}{*}{$n_1 \times n_2 \times n_3$} & \multirow{2}{*}{$r$} & \multicolumn{3}{c}{samples} & \multirow{2}{*}{fraction in $\R$} & \multirow{2}{2.5em}{time\\ (min)} \\
	\cmidrule{3-5}
	&& $\R$ & $\C$ & failed &  \\
	 \midrule
	 $2 \times 2 \times 2$ & $2$ & $100,000$ & $27,335$ & $41$   & $\underline{0.785}3\ldots \approx \frac{\pi}{4}$ & $1.3$ \\[2pt]
	 $3 \times 3 \times 2$ & $3$ & $100,000$ & $101,345$ & $185$ & $\underline{0.49}66\ldots \approx \frac{1}{2}$ & $2.8$\\[2pt]
	 $4 \times 4 \times 2$ & $4$ & $100,000$ & $288,770$ & $325$ & $\underline{0.25}72\ldots \approx \frac{27 \pi^2}{1024}$ & $14.9$\\[2pt]
	 $5 \times 4 \times 3$ & $6$ & $100,000$ & $1,237,912$ & $643$ & $0.0747\ldots$ & $420.6$ \\[2pt]
	 $5 \times 5 \times 2$ & $5$ & $100,000$ & $810,254$ & $509$ & $\underline{0.10}98\ldots \approx \frac{1}{9}$ & $99.3$ \\[2pt]
	 \bottomrule
	\end{tabular}
	\caption{Results of sampling GITs in $\sigma_{r;n_1,n_2,n_3} \subset \R^{n_1 \times n_2 \times n_3}$ via an acceptance--rejection method. Columns three to five list the number of samples where the final tracked solution of the homotopy was real, complex, or failed, respectively. The next column shows the fraction of successful samples that were real; in the case of $n \times n \times 2$ the analytical solution from \cite{BF2011} is also stated and the correct digits from the empirical estimate are underlined. The final column indicates the total wall-clock time required to perform the Monte Carlo experiments.}
	 \label{tab_homo_sols}
	\end{table}

	{In \Cref{tab_homo_sols}, the total fractions of solutions} that are real when sampling random Gaussian tensors (with i.i.d.~standard normally distributed entries) seem to agree very well with the known theoretical results by Bergqvist and Forrester \cite{BF2011}; they showed that for random Gaussian $n \times n \times 2$ tensors the rank is $n = \frac{\Pi}{\Sigma}$ with probability $p_n := \Gamma\bigl(\frac{n+1}{2}\bigr)^n \bigl( G(n+1) \bigr)^{-1}$, where $\Gamma$ is the gamma function and $G$ the Barnes $G$-function (or double gamma function). The correct digits in the numerical approximation are underlined in the penultimate column of \cref{tab_homo_sols}.
	
	\begin{figure}
	 \centerline{\subfloat{\includegraphics[height=6.4cm]{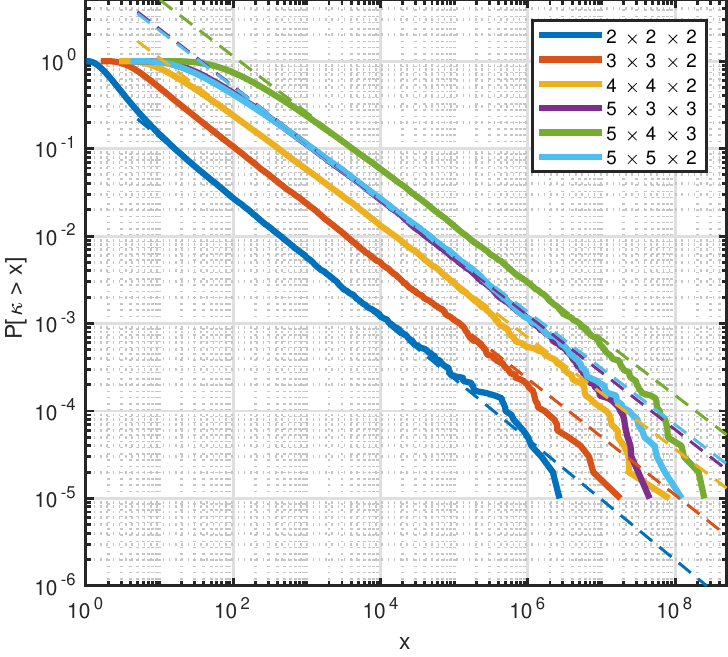}} \hspace{.4cm}
	 \subfloat{\includegraphics[height=6.4cm]{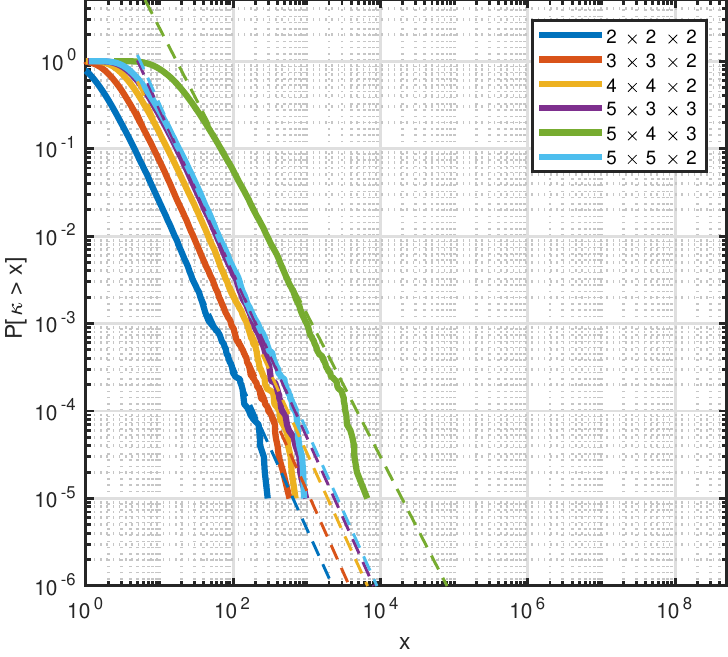}}}
	 \caption{Empirical complementary cumulative distribution function of the regular and angular condition numbers for the tensor spaces from \cref{tab_homo_sols}. Both plots are on the same scale.}
	 \label{fig_distribution}
	\end{figure}
	
	The empirical complementary cumulative distribution functions of the regular and angular condition numbers are shown in \cref{fig_distribution}. The full lines correspond to the empirical data and the thinner dashed lines correspond to an exponential model fitted to the data. From the figure it is namely reasonable to postulate that the complementary cumulative distribution function $c(x)$ for large $x$ has the form
	\begin{align} \label{eqn_exp_model}
	 c(x) = 1 - \int_0^x p(t) \,\d{}t = a x^{-b},
	\end{align}
	which corresponds to a straight line in the log-log plot in \cref{fig_distribution}. We fitted the parameters $a$ and $b$ of the postulated model to the data restricted to the range $10^{-1} \le c(x) \le 10^{-3}$. The reason for restricting the data set is that for small condition numbers it is visually evident in \cref{fig_distribution} that the model is incorrect and for large condition numbers the data contains few samples, which negatively impacts the robustness of the fit. The parameters were fitted using \texttt{fminsearch} from Matlab R2017b with starting point $(1, 1)$ and default settings. In all cases the algorithm terminated because the relative change of the parameters fell below $10^{-4}$. The obtained parameters are shown in \cref{tab_parameters}, along with the coefficient of determination $R^2$ between the log-transformed data and log-transformed model predictions; $1$ indicates perfect correlation.
	
	\begin{table}
	\begin{tabular}{ccccccccc}
	\toprule
	\multirow{2}{*}{$n_1 \times n_2 \times n_3$} && \multicolumn{3}{c}{regular} && \multicolumn{3}{c}{angular} \\
	\cmidrule{3-5} \cmidrule{7-9}
	&& $a$ & $b$ & $R^2$ && $a$ & $b$ & $R^2$ \\
	\midrule
	$2 \times 2 \times 2$ && $0.6624$  & $0.6904$ & $0.9999$ && $1.7288$   & $1.8624$ & $0.9995$\\
	$3 \times 3 \times 2$ && $2.2348$  & $0.6636$ & $0.9999$ && $5.5496$   & $1.8856$ & $0.9998$\\
	$4 \times 4 \times 2$ && $4.7318$  & $0.6388$ & $0.9997$ && $11.4165$  & $1.8455$ & $0.9994$\\
	$5 \times 4 \times 3$ && $22.3141$ & $0.6461$ & $0.9998$ && $102.4887$ & $1.6337$ & $0.9992$\\
	$5 \times 5 \times 2$ && $9.8634$  & $0.6436$ & $0.9997$ && $23.6951$  & $1.8662$ & $0.9996$\\
	\bottomrule
	\end{tabular}
	 \caption{Estimated parameters of the exponential model \cref{eqn_exp_model} fitted to the complementary cumulative distribution functions from \cref{fig_distribution}. The coefficient of determination $R^2$ between the log-transformed data and log-transformed model predictions is also indicated.}
	 \label{tab_parameters}
	\end{table}

	We may estimate the expected values of the regular and angular condition numbers based on their empirical distributions. If $p(x)$ denotes the probability {density} function of the regular condition number, then
	\[
	 \mathbb{E}\, \kappa(\tensor{A}) = \int_0^\infty x p(x) \,\d{}x.
	\]
	From the postulated model of $c(x)$ we find that $p(x) = ab x^{-b-1}$ for large $x$, so that we postulate that the expected value will be well approximated by
	\[
	\mathbb{E}\, \kappa(\tensor{A}) \approx \int_0^{\kappa_0} x p(x) \,\d{}x + ab \int_{\kappa_0}^\infty x^{-b} \,\d{}x
	\]
	for some finite $\kappa_0$. The expression on the right is finite only if $b > 1$. The same discussion applies to the angular condition number as well. Regarding the estimated parameters in \cref{tab_parameters}, our empirical \emph{data strongly suggests that the expected value of the condition number is infinite} for $r = \frac{\Pi}{\Sigma}$ in the tested cases, as $b \approx 0.6 < 1$. On the other hand, \emph{the expected angular condition number seems finite} in all cases, as $b \approx 1.8 > 1$. This suggests that both \cref{thm_main_for_higher_ranks} and \cref{thm_main2} might hold for higher ranks as well.

	\appendix
	
	\section{Proof of the Lemmata from \cref{sec:proof_main}}
	\label{app_main}
	
	\subsection{Proof of \cref{lem:innerBV}}\label{proof_lem1}
	Integrating in polar coordinates, we have
	\begin{align*}
		&\int_{(0,\infty)}\int_{(0,\infty)} \lambda^{\Sigma-1}\,\mu^{\Sigma-1}\,e^{-\frac{\Vert \lambda \tensor{U} + \mu \tensor{V}\Vert^2}{2}}\,\d{}\lambda\, \d{}\mu\\
		&= \int_{0}^{\frac{\pi}{2}}\int_0^\infty (\cos(\theta)\sin(\theta))^{\Sigma-1} \rho^{2\Sigma-1} \;e^{-\rho^2\frac{\Vert \cos(\theta) \tensor{U} + \sin(\theta) \tensor{V}\Vert^2}{2}}\,\d{}\rho\, \d{}\theta.
	\end{align*}
		The change of variables $t:= \rho\Vert \cos(\theta) \tensor{U} + \sin(\theta) \tensor{V}\Vert$ transforms the integral for~$\rho$ into
	\[
		\frac{1}{\Vert \cos(\theta) \tensor{U}+ \sin(\theta) \tensor{V}\Vert^{2\Sigma}}\int_0^\infty t^{2\Sigma-1} \;e^{-\frac{t^2}{2}}\,\d{}t.
		\]
	The last integral is $2^{\Sigma -1}\Gamma(\Sigma)$. Plugging this into the equation above shows the assertion.\qed

	\subsection{Proof of \cref{lem:integralbound}}\label{proof_lem2}
	{Expanding the denominator and taking a change of variables by $\varphi=2\theta$, the lemma reduces to proving that
	\[
	\int_0^{\pi/2}\frac{\cos^{s-1}\varphi}{(1\pm a\cos\varphi)^s} \d\varphi
	=\int_0^{\pi/2}\frac{\sin^{s-1}\varphi}{(1\pm a\sin\varphi)^s} \d\varphi
	\geq \frac{k(s)}{(1\pm a)^{s-1/2}},
	\]
	{(for the first equality, make the change of variables $\varphi\to\pi/2-\varphi$)}
	where we are denoting $a=|\langle \vect{x},\vect{y}\rangle|\in[0,1]$. With the $+$ sign the inequality is clear (choosing an appropriate $k(s)$). For the other case we have, up to a constant $k(s)$, the lower bound
	\[
	\int_0^{\pi/4}\frac{\d\varphi}{(1- a\cos\varphi)^s}
	\geq \int_0^{\pi/4}\frac{\d\varphi}{\left(1- a\left(1-\varphi^2/2\right)\right)^s}
	=\frac{1}{(1-a)^s}\int_0^{\pi/4} \frac{\d\varphi}{\left(1+\frac{a}{1-a}\frac{\varphi^2}{2}\right)^s} .
	\]
	Thus, it suffices to check that
	\[
	\int_0^{\pi/4} \frac{1}{\left(1+\frac{a}{1-a}\frac{\varphi^2}{2}\right)^s} d \varphi
	\geq k(s)\sqrt{1-a},
	\]
	but this is easily checked taking the change of variables $\varphi^2=(1-a)t$.}
	\qed

	\subsection{Proof of \cref{sec3:auxiliary_lemma}}\label{proof_lem3}
	Recall from the definition of $D(\epsilon)$ that $\Vert \vect{u}^1-\vect{v}^1\Vert < \epsilon$, and that $\frac{9}{10} \Vert \vect{u}^1 - \vect{v}^1\Vert <  \Vert\vect{u}^k - \vect{v}^k\Vert<\Vert\vect{u}^1 - \vect{v}^1\Vert$ for $2\leq k\leq d$. If $\epsilon > 0$ is sufficiently small we can assume
	\begin{equation}\label{delta_ineq}
	 \delta_k := \langle \vect{u}^k, \vect{v}^k \rangle \ge \frac{9}{10}, \quad 1 \le k \le d.
	\end{equation}
	To prove the lemma, it suffices to show that
	\[
	 q(U)^2 \ge {2^{-d}}\left( \frac{\| \vect{u}^1 - \vect{v}^1 \|^2}{4} \right)^{\Sigma - 1} = {2^{-d}} \left( \frac{1 - \delta_1}{2} \right)^{\Sigma - 1}
	\]
	for sufficiently small $\epsilon$.
	
	For convenience, we will first introduce a few auxiliary variables. Consider the next picture:
	\[
	\begin{tikzpicture}[scale=.6]
	 \draw (9,0) -- (7,0);
	 \draw (0,0) -- (8,2) -- (9,2.25);
	 \draw[->,>=stealth] (0,0) -- (8,0) node[below] {$\vect{v}^k$};
	 \draw[dashed] (8,0)  -- (8,2);
	 \draw[->,>=stealth] (0,0) -- (7.77,1.95);
	 \draw (7.5,2.1) node[left] {$\vect{u}^k$};
	 \draw (8,0) -- (7.77,1.95);
	 \draw[color=blue] (7.52,1.89) -- (8,0) -- (8,2) -- (7.52,1.89) -- (8,0);
	 \draw[color=blue] (8,1) node[right] {$\epsilon_k$};
	 \draw[color=blue] (7.8,1) node[left] {$\zeta_k$};
	 \draw (1,0) arc (0:15:1);
	 \draw (1.5,.22) node[right] {$\theta$};
	 \draw (8.25,0.1) |- (8.1,0.25);
	 \draw (7.48,1.6) -- (7.34,1.565) -- (7.30,1.699);
	\end{tikzpicture}
	\]
	Then, for small $\theta > 0$, we have the following elementary {trigonometric} relations:
	\begin{align} \nonumber
	\delta_k &:= \cos \theta = \langle \vect{u}^k,\vect{v}^k\rangle = \sqrt{\tfrac{1}{\epsilon_k^2+1}}, \\ \label{eqn_proof_trigonio}
	\zeta_k &:= \sin \theta = \sqrt{1 - \delta_k^2} = \delta_k \epsilon_k, \text{ and } \\ \nonumber
	\epsilon_k &:= \tan \theta = \sqrt{\tfrac{1}{\delta_k^2}-1}.
	\end{align}
	It follows from the definition of $D(\epsilon)$ that
	\begin{equation*} 
	\delta_1 = \min \, \{\delta_1,\ldots,\delta_d\} \quad \text{ and } \quad \epsilon_1 = \max \, \{\epsilon_1,\ldots,\epsilon_d\}.
	\end{equation*}
	From the previous figure it is also clear that
	\begin{align}\label{eqn_proof_epsilonbound}
	 \zeta_k = \delta_k \epsilon_k \le \| \vect{u}^k - \vect{v}^k \| \le \epsilon_k, \;\text{ and so }\;
	 \left(\frac{9}{10}\right)^2 \epsilon_1 \le \frac{9}{10} \|\vect{u}^1 - \vect{v}^1\| \le \|\vect{u}^k - \vect{v}^k\| \le \epsilon_k,
	\end{align}
	having used in the right sequence of inequalities that \cref{delta_ineq} implies $\frac{9}{10} \delta_1 \epsilon_1 \ge \left( \frac{9}{10} \right)^2 \epsilon_1$. Then,
	\begin{equation} \label{deltas}
	1-\frac{\epsilon_1^2}{2}\leq1-\frac{\epsilon_k^2}{2}\leq\delta_k=\sqrt{\tfrac{1}{\epsilon_k^2+1}}\leq 1-\frac{\epsilon_k^2}{4}\leq 1-\frac{1}{4}\left(\frac{9}{10}\right)^4\epsilon_1^2.
	\end{equation}
	Finally, we will also use
	\begin{equation}\label{def_z}
	z := \delta_1\cdots \delta_d.
	\end{equation}
	
	\subsubsection*{Transforming $q(U)$}
	For computing $q(U)$, we will first make a convenient orthogonal transformation of $U$'s columns.
	Recall from \cref{eqn_nice_expr_terracini} that $U=\begin{bmatrix} \tensor{U} & L_1 & \tensor{V}& L_2 \end{bmatrix}$. As in \cref{def_M} we write $M=\begin{bmatrix} \tensor{U}&\tensor{V}\end{bmatrix}$. Then, $q(U) = q(\begin{bmatrix} M & L_1 & L_2 \end{bmatrix})$. The block structure of~$L_1,L_2\in\R^{\Pi \times (\Sigma -1)}$ was given in \cref{def_L_i}: $L_1$ is made up of the $d$ blocks
	\begin{align*}
	L_1^k = \sten{u}{}{1} \otimes \cdots \otimes \sten{u}{}{k-1} \otimes \dot{U}^k \otimes \sten{u}{}{k+1} \otimes \cdots \otimes \sten{u}{}{d}, \quad 1\leq k\leq d,
	\end{align*}
	and $L_2$ is analogously made up of the $d$ blocks
	\begin{align*}
	L_2^k = \sten{v}{}{1} \otimes \cdots \otimes \sten{v}{}{k-1} \otimes \dot{V}^k \otimes \sten{v}{}{k+1} \otimes \cdots \otimes \sten{v}{}{d}, \quad 1\leq k\leq d,
	\end{align*}
	where
	$
	\dot{U}^k =
	\begin{bmatrix}
	\sten{\dot{u}}{2}{k} & \sten{\dot{u}}{3}{k} & \cdots & \sten{\dot{u}}{n_k}{k}
	\end{bmatrix} \in \R^{n_k \times n_k-1}$
	and
	$\dot{V}^k =
	\begin{bmatrix}
	\sten{\dot{v}}{2}{k} & \sten{\dot{v}}{3}{k} & \cdots & \sten{\dot{v}}{n_k}{k}
	\end{bmatrix} \in \R^{n_k \times n_k-1}
	$
	are matrices whose columns form an orthonormal basis of $(\vect{u}^k)^\perp:= \Tang{\vect{u}^k}{\Sp(\R^{n_k})}$ and $(\vect{v}^k)^\perp:= \Tang{\vect{v}^k}{\Sp(\R^{n_k})}$, respectively.
	
	The columns of $M$ are rotated into
	\begin{align} \label{eqn_rotated_UV}
	\vect{a}_\uparrow := \frac{1}{\sqrt{2}}(\tensor{U} - \tensor{V}) \text{ and } \vect{a}_\downarrow := \frac{1}{\sqrt{2}}(\tensor{U} + \tensor{V}),\end{align}
	while the columns of $L_1$ and $L_2$ are rotated as follows.
	We define the $\Pi \times (\Sigma -1)$-matrices $R_\downarrow := \begin{bmatrix} R_\downarrow^1 & \ldots & R_\downarrow^d\end{bmatrix}$ and $R_\uparrow := \begin{bmatrix} R_\uparrow^1 & \ldots & R_\uparrow^d\end{bmatrix}$, where
	\begin{align}\label{def_R_i}
	R_\downarrow^k = \frac{1}{\sqrt{2}}(L_1^k - L_2^k) \quad \text{ and } \quad R_\uparrow^k  = \frac{1}{\sqrt{2}}(L_1^k + L_2^k).
	\end{align}
	The reason for using $\downarrow$ and $\uparrow$ will become clear from the computations below: Inner products of two quantities with arrows pointing in opposite directions are zero, and swapping the directions of both arrows flips a sign in the expression of the inner product.
	
	{Now, instead of considering the matrix $U$ we work with the matrix
	\[
	N:= \begin{bmatrix} \vect{a}_\downarrow & R_\downarrow & \vect{a}_\uparrow & R_\uparrow\end{bmatrix}.
	\]
	By construction,
		$\begin{bmatrix} M & L_1 &  L_2\end{bmatrix} = N Q$ for some orthogonal matrix~$Q$, so that}
	\begin{equation*} 
	q(U) = q(\begin{bmatrix} \vect{a}_\downarrow & R_\downarrow & \vect{a}_\uparrow & R_\uparrow\end{bmatrix}) = q(N)
	\quad\text{ and }\quad
	q(U)^2 = q(N^T N).
	\end{equation*}
	Note that the choice of $\dot{U}^k$ and $\dot{V}^k$ does not affect the value of $q$. In particular, we may choose the matrices as follows. Let $H$ be the plane in $\R^{n_k}$ spanned by $\vect{u}^k$ and $\vect{v}^k$. By definition of $D(\epsilon)$, $\vect{u}^k \ne \pm \vect{v}^k$. Let $O$ be the rotation that sends $\vect{u}^k$ to $\vect{v}^k$,
	but leaves the orthogonal complement $H^\perp$ of $H$ fixed. Take $\sten{\dot{u}}{2}{k} \in H$ with $\dot{\vect{u}}_2^k \perp \vect{u}^k$ as the unit norm vector making the smallest angle with $\vect{v}^k$, as follows:
	\[
	\begin{tikzpicture}
			\draw[thick] (0, 0) -- ({cos(90)}, {sin(90)}) node[above] {$\vect{u}^k$};
			\draw[thick] (0, 0) -- ({cos(120)}, {sin(120)}) node[above] {$\vect{v}^k$};
			\draw[thick] (0, 0) -- ({cos(180)}, {sin(180)}) node[left] {$\sten{\dot{u}}{2}{k}$};
			\draw[thick] (0, 0) -- ({cos(210)}, {sin(210)}) node[left] {$\sten{\dot{v}}{2}{k}$};
			\draw [gray,domain=0:360] plot ({cos(\x)}, {sin(\x)});
	\end{tikzpicture}
	\]
	We also take $\sten{\dot{v}}{2}{k} := O \sten{\dot{u}}{2}{k}$, as in the illustration above. If $\vect{h}_3,\ldots, \vect{h}_{n_k}$ is any orthogonal basis of $H^\perp$, then our choice of bases is
	\begin{equation}\label{choice_of_tangent_vectors}
	\sten{\dot{u}}{2}{k} \text{ and } \sten{\dot{v}}{2}{k} = O \sten{\dot{u}}{2}{k} \text{ as above, and for } 3\leq j \leq n_k:
	\sten{\dot{u}}{j}{k} = \sten{\dot{v}}{j}{k} = \vect{h}_j.
	\end{equation}
	In other words, we can assume that all but the first columns of $\dot{U}^k$ and $\dot{V}^k$ are equal. Moreover, as can be seen from the foregoing figure, the following properties hold:
	\begin{align} \nonumber
	\langle \sten{\dot{u}}{2}{k}, \sten{\dot{v}}{2}{k} \rangle &= \langle Q \vect{u}^k, Q \vect{v}^k \rangle = \langle \vect{u}^k, \vect{v}^k\rangle = \delta_k,\\ \label{eqn_proof_angles}
	\langle \sten{\dot{u}}{2}{k}, \vect{v}^k \rangle &= \cos\left( \frac{\pi}{2} - \arccos( \delta_k ) \right) = \sqrt{1 - \delta_k^2} = \delta_k \epsilon_k, \\ \nonumber
	\langle \sten{\dot{v}}{2}{k}, \vect{u}^k \rangle &= \cos\left( \frac{\pi}{2} + \arccos( \delta_k ) \right) = -\sqrt{1 - \delta_k^2} = -\delta_k \epsilon_k,
	\end{align}
	where $Q$ is a rotation by $\frac{\pi}{2}$ radians in same direction as the rotation $O$. In particular, we have
	\begin{align*}
	 (\dot{U}^k)^T \vect{v}^k = \delta_k \epsilon_k \vect{e}^k  \text{ and } (\dot{V}^k)^T \vect{u}^k = -\delta_k \epsilon_k \vect{e}^k, \text{ where } \vect{e}^k := (1, 0, \ldots, 0)^T\in \R^{n_k-1}.
	\end{align*}
	Consider then the Gram matrix $G = N^T N$ for this particular choice of tangent vectors:
	\[
	G = \begin{bmatrix}
	\vect{a}_\downarrow^T\vect{a}_\downarrow & \vect{a}_\downarrow^TR_\downarrow & \vect{a}_\downarrow^T\vect{a}_\uparrow & \vect{a}_\downarrow^TR_\uparrow\\[0.2cm]
	R_\downarrow^T\vect{a}_\downarrow & R_\downarrow^TR_\downarrow & R_\downarrow^T\vect{a}_\uparrow & R_\downarrow^TR_\uparrow\\[0.2cm]
	\vect{a}_\uparrow^T\vect{a}_\downarrow & \vect{a}_\uparrow^TR_\downarrow & \vect{a}_\uparrow^T\vect{a}_\uparrow & \vect{a}_\uparrow^TR_\uparrow\\[0.2cm]
	R_\uparrow^T\vect{a}_\downarrow & R_\uparrow^TR_\downarrow & R_\uparrow^T\vect{a}_\uparrow & R_\uparrow^TR_\uparrow
	\end{bmatrix}.
	\]
	We continue by computing its entries.

	\subsubsection*{Inner products involving only $\vect{a}$}
	Using \cref{inner_prod_rank_one} for computing inner products of rank-1 tensors, we see that
	\begin{align} \label{eqn_innerprods_a_updown}
			\vect{a}_\downarrow^T\vect{a}_\downarrow = 1 + \langle \tensor{U},\tensor{V}\rangle = 1 + z, \quad
			\vect{a}_\downarrow^T\vect{a}_\uparrow = \vect{a}_\uparrow^T\vect{a}_\downarrow= 0, \text{ and}\quad
			\vect{a}_\uparrow^T\vect{a}_\uparrow = 1 - \langle \tensor{U},\tensor{V}\rangle = 1 - z.
		\end{align}
	
	\subsubsection*{Inner products involving both $\vect{a}$ and $R$}
	For each $k$ we have $(L_1^k)^T \tensor{U} = (L_2^k)^T \tensor{V} = 0$,
	\begin{align*}
	(L_1^k)^T\tensor{V} &= \big(\prod_{i\neq k} \delta_i\big)\, (\dot{U}^k)^T\vect{v}^k = z \epsilon_k \vect{e}^k\quad \text{ and }
	\quad
	(L_2^k)^T\tensor{U} = \big(\prod_{i\neq k} \delta_i\big)\, (\dot{V}^k)^T\vect{u}^k = -z\epsilon_k \vect{e}^k.
	\end{align*}
	This implies
	\begin{align*}
	(R_\downarrow^k)^T \vect{a}_\downarrow &= \frac{1}{2}(L_1^k-L_2^k)^T(\tensor{U}+\tensor{V}) = z \epsilon_k \vect{e}^k.\\
	(R_\downarrow^k)^T \vect{a}_\uparrow &= \frac{1}{2}(L_1^k-L_2^k)^T(\tensor{U}-\tensor{V}) = 0,\\
	(R_\uparrow^k)^T \vect{a}_\downarrow &= \frac{1}{2}(L_1^k+L_2^k)^T(\tensor{U}+\tensor{V}) =0,\\
	(R_\uparrow^k)^T \vect{a}_\uparrow &= \frac{1}{2}(L_1^k+L_2^k)^T(\tensor{U}-\tensor{V}) = - z \epsilon_k \vect{e}^k,
	\end{align*}
	having used \cref{eqn_proof_angles} and \cref{inner_prod_rank_one}.
	Combining the above, we obtain
	\begin{equation}\label{eqA55}
	R_\uparrow^T\vect{a}_\downarrow = R_\downarrow^T\vect{a}_\uparrow=0,\quad
	R_\downarrow^T\vect{a}_\downarrow = z\vect{f},
	\quad \text{and} \quad
	R_\uparrow^T\vect{a}_\uparrow = -z \vect{f}, \text{ where } \vect{f}=\begin{bmatrix}\epsilon_1\vect{e}^1\\\vdots\\\epsilon_d\vect{e}^{d}\end{bmatrix}.
	\end{equation}
	
	\subsubsection*{Inner products involving only $R$}
	By construction, we have $(\dot{U}^k)^T \dot{U}^k = (\dot{V}^k)^T \dot{V}^k =\Id_{n_k-1} $, where $\Id_{n_k-1}$ is the $(n_k-1)\times (n_k-1)$ identity matrix. Furthermore, by our choice of tangent vectors from \cref{choice_of_tangent_vectors} and \cref{eqn_proof_angles}, we have
	$(\dot{U}^k)^T \dot{V}^k  = \mathrm{diag}(\langle \sten{\dot{u}}{j}{k},\sten{\dot{v}}{j}{k} \rangle)_{j=2}^{n_k} = \mathrm{diag}(\delta_k, 1,\ldots, 1).$ This implies
	\[
	(L_1^k)^T L_1^k = (L_2^k)^T L_2^k = \Id_{n_k -1},  \quad (L_1^k)^T L_2^k = \big(\prod_{i\neq k}\delta_i\big) \, (\dot{U}^k)^T \dot{V}^k = z \cdot \mathrm{diag}(1, \delta_k^{-1}, \ldots, \delta_k^{-1}).
	\]
	Moreover, for $j \ne k$ we have
	\begin{align*}
	(L_1^j)^T L_1^k &= (L_2^j)^T L_2^k  = 0, \\
	(L_1^j)^T L_2^k &=
	\Bigl(\prod_{i\not\in\{k,j\}} \delta_i\Bigr) \cdot
	\begin{cases}
	\bigl((\dot{U}^j)^T \vect{v}^j \bigr) \otimes \bigl( (\vect{u}^k)^T \dot{V}^k \bigr), & j < k, \\[2pt]
	\bigl( (\vect{u}^k)^T \dot{V}^k \bigr) \otimes \bigl( (\dot{U}^j)^T \vect{v}^j \bigr), & j > k
	\end{cases}
	= -z \epsilon_j \epsilon_k \vect{e}^j (\vect{e}^k)^T.
	\end{align*}
	From this we get
	\begin{align*}
	(R_\downarrow^j)^T R_\downarrow^k &= \frac{1}{2}(L_1^j - L_2^j)^T (L_1^k - L_2^k) = \begin{cases} \Id_{n_k -1} -z \cdot \mathrm{diag}(1, \delta_k^{-1}, \ldots, \delta_k^{-1}),& j=k\\ z \epsilon_j \epsilon_k  \,\vect{e}^j (\vect{e}^k)^T ,& j\neq k\end{cases},\\
	(R_\downarrow^j)^T R_\uparrow^k &= \frac{1}{2}(L_1^j - L_2^j)^T (L_1^k + L_2^k) = 0, \\
	(R_\uparrow^j)^T R_\uparrow^k &= \frac{1}{2}(L_1^j + L_2^j)^T (L_1^k + L_2^k) = \begin{cases} \Id_{n_k -1} + z \cdot \mathrm{diag}(1, \delta_k^{-1}, \ldots, \delta_k^{-1}),& j=k\\ -z \epsilon_j \epsilon_k \, \vect{e}^j (\vect{e}^k)^T,& j\neq k\end{cases}.
	\end{align*}
	Note that
	$$
	\Id_{n_k-1} - z \cdot \mathrm{diag}(1, \delta_k^{-1}, \ldots, \delta_k^{-1})
	= \Id_{n_k-1} - z \cdot \mathrm{diag}(1+\epsilon_k^2, \delta_k^{-1}, \ldots, \delta_k^{-1})  + z \epsilon_k^2 \vect{e}^k(\vect{e}^k)^T,
	$$
	and
	$$
	\Id_{n_k-1} + z \cdot \mathrm{diag}(1, \delta_k^{-1}, \ldots, \delta_k^{-1}) = \Id_{n_k-1} + z \cdot \mathrm{diag}(1+\epsilon_k^2, \delta_k^{-1}, \ldots, \delta_k^{-1}) - z \epsilon_k^2 \vect{e}^k(\vect{e}^k)^T.
	$$%
	Exploiting the definition of the vector $\vect{f}$ in \cref{eqA55}, the foregoing can be expressed concisely as
	\begin{align*}
	R_\downarrow^T R_\downarrow &= D_\downarrow + z\vect{f}\vect{f}^T, \quad
	R_\downarrow^T R_\uparrow =0, \quad\text{and}\quad
	R_\uparrow^T R_\uparrow = D_\uparrow - z\vect{f}\vect{f}^T,
	\end{align*}
	where we introduced
	\begin{align*}
	D_\downarrow&= \Id_{\Sigma-1} - z \cdot \mathrm{diag}(1+\epsilon_1^2, \underbrace{\delta_1^{-1}, \ldots, \delta_1^{-1}}_{(n_1-2) \text{-times}}\,, \ldots, 1 + \epsilon_d^2, \underbrace{\delta_d^{-1}, \ldots, \delta_d^{-1}}_{(n_d-2) \text{-times}}), \text{ and} \\
	D_\uparrow &= \Id_{\Sigma-1} + z \cdot \mathrm{diag}(1+\epsilon_1^2, \underbrace{\delta_1^{-1}, \ldots, \delta_1^{-1}}_{(n_1-2) \text{-times}}\,, \ldots, 1 + \epsilon_d^2, \underbrace{\delta_d^{-1}, \ldots, \delta_d^{-1}}_{(n_d-2) \text{-times}}).
	\end{align*}
	
	\subsubsection*{Putting everything together}
	From the definition of the vector $\vect{f}$ in \cref{eqA55}, it is clear that we can construct a permutation matrix $P$ that moves the nonzero elements of $\vect{f}$ to the first $d$ positions:
	\begin{align*}
	 P \vect{f} = \begin{bmatrix} \vect{g} \\ 0 \end{bmatrix}, \text{ where } \vect{g} = \begin{bmatrix} \epsilon_1 & \cdots & \epsilon_d \end{bmatrix}^T
	\end{align*}
	and $0$ is a vector of zeros of length $\Sigma-1-d$. Applying $P^T$ on the right of the $R$'s yields
	\begin{align} \label{eqn_def_S_and_T}
	 R_\downarrow P^T := \begin{bmatrix} T_\downarrow & S_\downarrow \end{bmatrix} \text{ and }
	 R_\uparrow P^T := \begin{bmatrix} T_\uparrow & S_\uparrow \end{bmatrix},
	\end{align}
	where the $T$ matrices are respectively given by
	 {\small\begin{multline*}
	\sqrt{2} T_\downarrow =
	[
	  \vect{u}^1 \otimes \cdots \otimes \vect{u}^{k-1} \otimes \sten{\dot{u}}{2}{k} \otimes \vect{u}^{k+1} \otimes \cdots \otimes \vect{u}^d - \vect{v}^1 \otimes \\ \cdots \otimes \vect{v}^{k-1} \otimes \sten{\dot{v}}{2}{k} \otimes \vect{v}^{k+1} \otimes \cdots \otimes \vect{v}^d
	 ]_{k=1}^d,
	\end{multline*}}%
	and analogously for $T_\uparrow$ replacing the subtraction by an addition; the matrices $S_\downarrow$ and $S_\uparrow$ contain the remainder of the columns of $R_\downarrow$ and $R_\uparrow$ respectively.
	Then, we have
	\begin{align*}
	 P (D_\downarrow + z \vect{f} \vect{f}^T) P^T
	 &= \begin{bmatrix} E_\downarrow & 0 \\[3pt] 0 & F_\downarrow \end{bmatrix} + z \begin{bmatrix} \vect{g} \\[3pt] 0 \end{bmatrix} \begin{bmatrix} \vect{g}^T & 0^T \end{bmatrix}
	 = \begin{bmatrix} E_\downarrow + z \vect{g} \vect{g}^T & 0 \\[3pt] 0 & F_\downarrow \end{bmatrix}
	 = \begin{bmatrix} T_\downarrow^T T_\downarrow & T_\downarrow^T S_\downarrow \\[3pt] S_\downarrow^T T_\downarrow & S_\downarrow^T S_\downarrow \end{bmatrix}, \text{ and} \\ \nonumber
	 P (D_\uparrow - z \vect{f} \vect{f}^T) P^T
	 &= \begin{bmatrix} E_\uparrow & 0 \\[3pt] 0 & F_\uparrow \end{bmatrix} - z \begin{bmatrix} \vect{g} \\[3pt] 0 \end{bmatrix} \begin{bmatrix} \vect{g}^T & 0^T \end{bmatrix}
	 = \begin{bmatrix} E_\uparrow - z \vect{g}\vect{g}^T & 0 \\[3pt] 0 & F_\uparrow \end{bmatrix}
	 = \begin{bmatrix} T_\uparrow^T T_\uparrow & T_\uparrow^T S_\uparrow \\[3pt] S_\uparrow^T T_\uparrow & S_\uparrow^T S_\uparrow \end{bmatrix},
	\end{align*}%
	where
	\begin{align} \label{eqn_def_EandF}
	E_\downarrow &:= \Id_d - z \cdot \mathrm{diag}( 1 + \epsilon_1^2, 1 + \epsilon_2^2, \ldots, 1 + \epsilon_d^2 ), \\ \nonumber
	E_\uparrow &:= \Id_d + z \cdot \mathrm{diag}( 1 + \epsilon_1^2, 1 + \epsilon_2^2, \ldots, 1 + \epsilon_d^2 ), \\ \nonumber
	F_\downarrow &:= \Id_{\Sigma-1-d} - z \cdot \mathrm{diag}( \underbrace{\delta_1^{-1}, \ldots, \delta_1^{-1}}_{(n_1-2) \text{-times}}\,, \ldots, \underbrace{\delta_d^{-1}, \ldots, \delta_d^{-1}}_{(n_d-2) \text{-times}} ), \text{ and} \\ \nonumber
	F_\uparrow &:= \Id_{\Sigma-1-d} + z \cdot \mathrm{diag}( \underbrace{\delta_1^{-1}, \ldots, \delta_1^{-1}}_{(n_1-2) \text{-times}}\,, \ldots, \underbrace{\delta_d^{-1}, \ldots, \delta_d^{-1}}_{(n_d-2) \text{-times}} ).
	\end{align}
	Hence, by swapping rows and columns of $G$, which leaves the value of $q$ unchanged because they are orthogonal operations, we find that $q(G) = q(G')$ with
	\begin{align} \label{eqn_everything}
	G'
	&:= \begin{bmatrix}
	\vect{a}_\downarrow^T \vect{a}_\downarrow & \vect{a}_\downarrow^T T_\downarrow & \vect{a}_\downarrow^T S_\downarrow & \vect{a}_\downarrow^T\vect{a}_\uparrow & \vect{a}_\downarrow^T T_\uparrow & \vect{a}_\downarrow^T S_\uparrow \\[0.2cm]
	T_\downarrow^T \vect{a}_\downarrow & T_\downarrow^T T_\downarrow & T_\downarrow^T S_\downarrow & T_\downarrow^T\vect{a}_\uparrow & T_\downarrow^T T_\uparrow & T_\downarrow^T S_\uparrow \\[0.2cm]
	S_\downarrow^T \vect{a}_\downarrow & S_\downarrow^T T_\downarrow & S_\downarrow^T S_\downarrow & S_\downarrow^T\vect{a}_\uparrow & S_\downarrow^T T_\uparrow & S_\downarrow^T S_\uparrow \\[0.2cm]
	\vect{a}_\uparrow^T \vect{a}_\downarrow & \vect{a}_\uparrow^T T_\downarrow & \vect{a}_\uparrow^T S_\downarrow & \vect{a}_\uparrow^T\vect{a}_\uparrow & \vect{a}_\uparrow^T T_\uparrow & \vect{a}_\uparrow^T S_\uparrow \\[0.2cm]
	T_\uparrow^T \vect{a}_\downarrow & T_\uparrow^T T_\downarrow & T_\uparrow^T S_\downarrow & T_\uparrow^T\vect{a}_\uparrow & T_\uparrow^T T_\uparrow & T_\uparrow^T S_\uparrow \\[0.2cm]
	S_\uparrow^T \vect{a}_\downarrow & S_\uparrow^T T_\downarrow & S_\uparrow^T S_\downarrow & S_\uparrow^T\vect{a}_\uparrow & S_\uparrow^T T_\uparrow & S_\uparrow^T S_\uparrow
	\end{bmatrix} \\ \nonumber
	&=
	\begin{bmatrix}
	 1 + z & z \vect{g}^T & 0 & 0 & 0 & 0 \\[0.2cm]
	 z \vect{g} & E_\downarrow + z \vect{g} \vect{g}^T & 0 & 0 & 0 & 0 \\[0.2cm]
	 0 & 0 & F_\downarrow & 0 & 0 & 0 \\[0.2cm]
	 0 & 0 & 0 & 1 -z & -z \vect{g}^T & 0 \\[0.2cm]
	 0 & 0 & 0 & -z \vect{g} & E_\uparrow - z \vect{g} \vect{g}^T & 0 \\[0.2cm]
	 0 & 0 & 0 & 0 & 0 & F_\uparrow
	\end{bmatrix}.
	\end{align}
	
	\subsubsection*{Bounding $q(G)$}
	To simplify more, we write
	\[
	G_\downarrow := \begin{bmatrix} 1 & 0 \\ 0 & E_\downarrow  \end{bmatrix},\quad
	G_\uparrow := \begin{bmatrix} 1 & 0 \\ 0 & E_\uparrow \end{bmatrix},\quad \text{ and }\quad
	\vect{h} := \begin{bmatrix} 1 \\ \vect{g}\end{bmatrix},
	\]
	so that
	\[
	q(G) = q(G')  = q\left(\begin{bmatrix}
	 G_\downarrow + z\vect{h}\vect{h}^T & 0 & 0 & 0 \\
	0 & G_\uparrow - z\vect{h}\vect{h}^T & 0 & 0\\
	0 & 0 & F_\downarrow & 0 \\
	0 & 0 & 0 &  F_\uparrow
	\end{bmatrix}\right),
	\]
	where we swapped some rows and columns again.
	Because the smallest singular value of the matrix $G_\uparrow - z\vect{h}\vect{h}^T$ is larger than or equal to the smallest singular value of the positive semidefinite matrix $G'$, we obtain the bound
	\begin{equation}\label{eqA62}
	q(G) \geq q( G_\uparrow - z\vect{h}\vect{h}^T ) \det( G_\downarrow + z\vect{h}\vect{h}^T ) \det(F_\downarrow) \det(F_\uparrow).
	\end{equation}
	
	We now obtain bounds on the individual factors on the right-hand side of \cref{eqA62}.
	First, we compute the determinants of the diagonal matrices:
	\[
	\det(F_\downarrow) \det(F_\uparrow)
	= \prod_{i=1}^d \bigl( (1-z \delta_i^{-1}) (1 + z \delta_i^{-1}) \bigr)^{n_i-2}
	\ge \prod_{i=1}^d ( 1-z \delta_i^{-1} )^{n_i-2},
	\]
	having used $0 < z = \delta_1 \cdots \delta_d \le 1$. Next, we see that
	\begin{align*} 
	 1 - z \delta_i^{-1} \ge 1 - \delta_{\max}^{d-1} \ge 1 - \delta_{\max} > \left(\frac{9}{10}\right)^2(1 - \delta_1), \text{ where } \delta_{\max} := \max\{\delta_1, \ldots, \delta_d \}.
	\end{align*}
	Note that we used $2 \left(\frac{9}{10}\right)^2 (1-\delta_1) = \left(\frac{9}{10}\right)^2 \| \vect{u}^1 - \vect{v}^1 \|^2 < \| \vect{u}^k - \vect{v}^k \|^2 = 2 (1-\delta_k)$ in the last inequality. As a result, we obtain
	\begin{align} \label{eq70}
	 \det(F_\downarrow) \det(F_\uparrow)
	 > \left(\frac{81}{100}\right)^{\Sigma-d-1} (1-\delta_1)^{\Sigma-d-1}
	 \ge \left( \frac{1-\delta_1}{2} \right)^{\Sigma-d-1}.
	\end{align}
	
	The final determinant in \cref{eqA62} can be computed as follows. Note that $z \vect{h} \vect{h}^T$ is a symmetric matrix with one positive eigenvalue and all others zero. Hence, it follows from Weyl's inequalities \cite[Theorem 4.3.7]{HJ1990} that the eigenvalues of $G_\downarrow$ cannot decrease by adding $z \vect{h} \vect{h}^T$. Hence,
	\begin{align*}
		 \det( G_\downarrow + z\vect{h}\vect{h}^T )
		 \ge \det(G_\downarrow)
		 = \det(E_\downarrow)
		 = \prod_{i=1}^d (1 - z(1+\epsilon_i^2)).
	 \end{align*}
	
		 Next, we bound $1 - z(1+\epsilon_i^2)$ from below and above. From \eqref{eqn_proof_epsilonbound} and \eqref{deltas} we have for some universal constant $C>0$:
		 \begin{align}\label{newboundU}
		 1-z(1+\epsilon_i^2)&\leq 1-(1-C\epsilon_1^2)^{d+1}\\
		&=(1-(1-C\epsilon_1^2))(1+(1-C\epsilon_1^2)+\cdots+(1-C\epsilon_1^2)^d)\nonumber\\
		&\leq C(d+1)\epsilon_1^2.\nonumber
	\end{align}
		 For obtaining the lower bound, note that
		 \begin{multline*}
		 z = \delta_1\cdots\delta_d
		 \leq \sqrt{\frac{1}{(1+\epsilon_1^2)(1+\epsilon_2^2)(1+\epsilon_3^2)}}
		 \leq \sqrt{\frac{1}{1+\epsilon_1^2} \frac{1}{(1+(9/10)^4 \epsilon_1^2)^2}} \\= \frac{1}{1+(9/10)^4 \epsilon_1^2} \sqrt{\frac{1}{1+\epsilon_1^2}},
		 \end{multline*}
		so that
			  \begin{equation}\label{newboundL}
		 1-z(1+\epsilon_i^2)\geq 1-\frac{\sqrt{1+\epsilon_1^2}}{1+(9/10)^4\epsilon_1^2}\geq \frac{\epsilon_1^2}{8}.
		 \end{equation}
	By \cref{eqn_proof_epsilonbound}, $\epsilon_1^2 \ge 2(1 - \delta_1)$, so that
	\begin{align} \label{eq71}
	\det( G_\downarrow + z\vect{h}\vect{h}^T )
	\ge \left(\frac{2}{8}\right)^d (1 - \delta_1)^d
	= 2^{-d} \left(\frac{1 - \delta_1}{2}\right)^d.
	\end{align}
	
	The proof can be completed by a fortunate application of Ger\v{s}gorin's theorem \cite[Satz~III]{Gerschgorin1931}. According to this theorem, the eigenvalues of
	\[
	G^\uparrow - z\vect{h}\vect{h}^T =
	\begin{bmatrix}
	 1 - z      & -z\epsilon_1            & -z\epsilon_2            & \cdots & -z\epsilon_d \\
	 -z\epsilon_1 & 1 + z                 & -z\epsilon_1 \epsilon_2 & \cdots & -z\epsilon_1 \epsilon_d \\
	 -z\epsilon_2 & -z\epsilon_1 \epsilon_2 & \ddots                &        & \vdots \\
	 \vdots     & \vdots                    &                       & 1+z    & -z\epsilon_{d-1}\epsilon_d \\
	 -z\epsilon_d & -z\epsilon_1 \epsilon_d & \cdots                & -z\epsilon_{d-1}\epsilon_d & 1+z
	\end{bmatrix}
	\]
	are contained in the following {Ger\v{s}gorin discs
	\begin{align*}
		 \mathrm{disc}_0 &:= \Bigl\{ x \in \C \;\;|\;\; | (1-z) - x | \le z \sum_{k=1}^d \epsilon_k \Bigr\}, \text{ and}\\
		 \mathrm{disc}_i &:= \Bigl\{ x \in \C \;\;|\;\; | (1+z) - x | \le z\epsilon_k + z\epsilon_k \sum_{j\ne k} \epsilon_j\Bigr\}, \quad i=1, \ldots, d.
		\end{align*}
	{Recall from \cref{def_z} that $z = \delta_1\cdots\delta_d$ and from \cref{eqn_proof_trigonio} that $\delta_k = \sqrt{\tfrac{1}{1+\epsilon_k^2}}$. By choosing $\epsilon = \epsilon_1$ small enough, we can get $z$ as close to $1$ and $\epsilon_k \le \epsilon_1$ as close to $0$ as we want. Therefore, if we choose $\epsilon$ small enough,} $\mathrm{disc}_0$ is a small disc near zero, and $\mathrm{disc}_k$ are small, pairwise overlapping discs near two. When $\epsilon$ is sufficiently small, $\mathrm{disc}_0$ is disjoint from the other discs, so that $\mathrm{disc}_0$ contains exactly $1$ eigenvalue close to $0$, and $\bigcup_{i=1}^d \mathrm{disc}_i$ contains $d$ eigenvalues close to $2$. Furthermore, since we are dealing with symmetric matrices, all the eigenvalues are real.} Therefore, for sufficiently small
	$\epsilon$,
	\begin{align}
		\label{eq72}
	 q(G_\uparrow - z \vect{h}\vect{h}^T) \ge (1 + z - d \epsilon_1)^d \ge 1.
	\end{align}
	Finally, plugging \cref{eq70}, \cref{eq71}, and \cref{eq72} into \cref{eqA62}, we find
	\[
	q(G) \geq 2^{-d} \left( \frac{1-\delta_1}{2} \right)^{\Sigma -1}.
	\]
	This finishes the proof.
	
	\section{Proof of Lemma \ref{important_lemma}}\label{sec:proof_important_lemma}
	For the proof of \cref{important_lemma}, we need the following two simple results. The first one is a well-known result.
	
	\begin{lemma}\label{lemmaB1}
	For $1\leq k\leq d$, let $\Var{A}_k \subset \R^{n_k}$ be a linear subspace of dimension $m_k$ and let $\Var{A}_k^\perp$ be its orthogonal complement. Put $\Var{V} := \Var{A}_1 \otimes \cdots\otimes \Var{A}_d$ and $\Var{W} := \Var{A}_1^\perp\otimes \cdots\otimes \Var{A}_d^\perp$ and let $\tensor{A}\in \Var{S}_{n_1,\ldots,n_d} \cap \Var{V}$ and $\tensor{B} \in \Var S_{n_1,\ldots,n_d} \cap \Var W$. Then, the tangent spaces $\Tang{\tensor{A}}{\Var{S}_{n_1,\ldots,n_d}}$ and $\Tang{\tensor{B}}{\Var{S}_{n_1,\ldots,n_d}}$ are orthogonal to each other.
	\end{lemma}
	\begin{proof}
	Let us write
			$\tensor{A} = \vect{a}^1\otimes \cdots \otimes \vect{a}^d$ and $\tensor{B} = \vect{b}^1\otimes \cdots \otimes \vect{b}^d$. By \cref{tangent_space_segre}, all vectors of the form $\vect{a}^1\otimes \cdots \otimes \vect{a}^{k-1}\otimes \vect{v}\otimes \vect{a}^{k+1}\otimes\cdots \otimes \vect{a}^d$ for $1\leq k\leq d$ and where $\vect{v}\in \R^{n_k}$ constitute a generating set of $\Tang{\tensor{A}}{\Var{S}_{n_1,\ldots,n_d}}$. A similar statement holds for $\Tang{\tensor{B}}{\Var{S}_{n_1,\ldots,n_d}}$.
			Then, by~\cref{inner_prod_rank_one}, the inner product between two such generators $\vect{t}:=\vect{a}^1\otimes \cdots \otimes \vect{a}^{k-1}\otimes \vect{v}\otimes \vect{a}^{k+1}\otimes\cdots \otimes \vect{a}^d$ and $\vect{s}:=\vect{b}^1\otimes \cdots \otimes \vect{b}^{\ell-1}\otimes \vect{w}\otimes \vect{b}^{\ell+1}\otimes\cdots \otimes \vect{b}^d$ is
		\[
		\langle\vect{t}, \vect{s}\rangle =
		\begin{cases}
			\langle \vect{v}, \vect{w}\rangle \, \prod_{j\neq k} \langle \vect{a}^j, \vect{b}^j\rangle, & \text{ if } k = \ell\\
			\langle \vect{a}^\ell, \vect{w}\rangle \,\langle \vect{v}, \vect{b}^k\rangle\, \prod_{j\neq k,\ell} \langle \vect{a}^j, \vect{b}^j\rangle, & \text{ if } k \neq \ell
		\end{cases} \;= 0,
		\]
		because $d\geq 3$ and $\langle \vect{a}^j, \vect{b}^j \rangle=0$ for all $1\leq j\leq d$. This completes the proof.
	\end{proof}

	\begin{lemma}\label{lemmaB2}
	For all $1\leq k\leq d$, let $\Var{A}_k\subset \R^{n_k}$ be a linear subspace of dimension $m_i$, and define the tensor space $\Var{V} := \Var{A}_1\otimes \cdots\otimes \Var{A}_d$. Then, the following holds.
	\begin{enumerate}
	\item $\Var S_{n_1,\ldots,n_d} \cap \Var V$ is a manifold.
	\item For $1\leq k\leq d$, let $U_k\in \R^{n_k\times m_k}$ be a matrix whose columns form an orthonormal basis for $\Var{A}_k$. The map
	\[
	\Var{S}_{m_1,\ldots,m_d}  \to \Var{S}_{n_1,\ldots,n_d} \cap \Var V,\; \vect{a}^1\otimes \cdots\otimes \vect{a}^d \mapsto (U_1 \vect{a}^1)\otimes \cdots\otimes (U_d \vect{a}^d)
	\]
	is an isometry.
	\end{enumerate}
	\end{lemma}
	\begin{proof}
	Let us write $U_1\otimes \cdots \otimes U_d$ for the map from (2). It extends to a linear map $\R^{m_1}\otimes \cdots\otimes\R^{m_d} \to \R^{n_1}\otimes \cdots\otimes\R^{n_d}$ because of the universal property \cite[Chapter 1]{Greub1978}.
	By definition, we have $(U_1\otimes \cdots \otimes U_d)(\Var{S}_{m_1,\ldots,m_d}) = \Var{S}_{n_1,\ldots,n_d} \cap \Var{V}$ and $(U_1^T\otimes \cdots \otimes U_d^T)(\Var{S}_{n_1,\ldots,n_d} \cap \Var{V}) = \Var{S}_{m_1,\ldots,m_d}$. Hence, $\Var{S}_{n_1,\ldots,n_d} \cap \Var{V}$ is the image of a manifold under an invertible linear map. This implies that $\Var{S}_{n_1,\ldots,n_d} \cap \Var{V}$ itself is a manifold. Furthermore,~$U$ preserves the Euclidean inner product, and so the manifolds $\Var{S}_{n_1,\ldots,n_d} \cap \Var{V}$ and $\Var{S}_{m_1,\ldots,m_d}$ are isometric.
	\end{proof}
	
	We are now ready to prove \cref{important_lemma}.
	
	\begin{proof}[Proof of Lemma \ref{important_lemma}]
	We assumed that $\sigma_{r-2; n_1-2,\ldots,n_d-2}\subset \R^{n_1-2}\otimes \cdots\otimes \R^{n_d-2}$ is generically complex identifiable. Then, \cref{prop:delotropaper} (2) tells us that $\Var M_{r-2;n_1-2,\ldots,n_d-2}$ is Zariski-open in $(\Var S_{n_1-2,\ldots,n_d-2})^{\times (r-2)}$. {Fix any} tuple $(\tensor{X}_1,\ldots,\tensor{X}_{r-2})\in \Var M_{r-2;n_1-2,\ldots,n_d-2}$.
	By definition of $\Var{M}_{r-2;n_1-2,\ldots,n_d-2}$, the least singular value of the derivative  $\deriv{\Phi}{(\tensor{X}_1,\ldots,\tensor{X}_{r-2})}$ of the addition map $\Phi: (\Var S_{n_1-2,\ldots,n_d-2})^{\times (r-2)}\to \Var{N}_{r-2;n_1-2,\ldots,n_d-2}$ is positive. This implies that
	$\vert \det \deriv{\Phi}{(\tensor{X}_1,\ldots,\tensor{X}_{r-2})} \vert > 0.$
	Hence, if $X_i$ is a matrix whose columns form an orthonormal basis for $\Tang{\Var S_{n_1-2, \ldots,n_d-2}}{\tensor{X}_i}$, we have
	\[
	\vert \det \deriv{\Phi}{(\tensor{X}_1,\ldots,\tensor{X}_{r-2})} \vert = \mathrm{vol}(\begin{bmatrix} X_1 & \ldots & X_{r-2}\end{bmatrix}) > 0.
	\]
	Moreover, if we write {$\tensor{X}_i = \nu_i \vect{x}^1_i \otimes \cdots \otimes \vect{x}^d_i$ with $\|\vect{x}_i^k\| =1$ and $\nu_i = \|\tensor{X}_i\|$,} then we have
	\begin{equation*}
	\mathrm{vol}\begin{bmatrix}
	\vect{x}_1^{h_1} \otimes \cdots \otimes \vect{x}_1^{h_{d-1}} & \cdots & \vect{x}_{r-2}^{h_1} \otimes \cdots \otimes \vect{x}_{r-2}^{h_{d-1}}
	\end{bmatrix}>0
	\end{equation*}
	for all subsets $\{h_1,\ldots,h_{d-1}\}\subset \{1,\ldots,d\}$ of cardinality $d-1$. Indeed, suppose to the contrary that the foregoing matrix would have linearly dependent columns for some subset; w.l.o.g., we can assume $h_i = i$ for $i=1,\ldots,d-1$. Then there are nonzero $\lambda_i$ such that $\sum_{i=1}^{r-2} \lambda_i \sten{x}{i}{1}\otimes\cdots\otimes\sten{x}{i}{d-1} = 0$. Tensoring with an arbitrary vector $\vect{v}\in\R^{n_d-2}$ yields $\sum_{i=1}^{r-2} \lambda_i \sten{x}{i}{1}\otimes\cdots\otimes\sten{x}{i}{d-1} \otimes \vect{v} = 0 \otimes \vect{v} = 0$, having exploited multilinearity. Note that the $i$th term in the last sum lives in $\Tang{\tensor{X}_i}{\Var{S}_{{n_1-2},\ldots,{n_d-2}}}${, and can be obtained as a linear combination of the columns of the $i$--th block of Terracini's matrix}, so that \cref{characterization_CN} implies that $\kappa(\tensor{X}_1, \ldots, \tensor{X}_{r-2}) = \infty$, which contradicts $(\tensor{X}_1, \ldots, \tensor{X}_{r-2}) \in \Var{M}_{r-2;n_1-2,\ldots,n_d-2}$.
	We define the following {constant}, which will play an important role in this proof:
	\[
	\mu := \frac{\mathrm{vol}(\begin{bmatrix} X_1 & \cdots & X_{r-2}\end{bmatrix})}{2}\hspace{-.2cm}\prod_{1\leq h_1<\cdots<h_{d-1}\leq d} \mathrm{vol}\begin{bmatrix}
	\vect{x}_1^{h_1} \otimes \cdots \otimes \vect{x}_1^{h_{d-1}} & \cdots & \vect{x}_{r-2}^{h_1} \otimes \cdots \otimes \vect{x}_{r-2}^{h_{d-1}}
	\end{bmatrix}.
	\]
	It is important to note that $\mu > 0$ is only defined through the choice of $(\tensor{X}_1,\ldots,\tensor{X}_{r-2})$. The key observation is that we can choose $\mu$ \emph{independently of what follows}.
	
	Recall that the restriction of $\phi$ to $\Var N_{2;n_1,\ldots,n_d} \times (\Var S_{n_1,\ldots,n_d})^{\times (r-2)}$ is
	\[
	\phi : \Var N_{2;n_1,\ldots,n_d} \times (\Var S_{n_1,\ldots,n_d})^{\times r-2} \to \Var N_r^*,\; (\tensor{B},\tensor{A}_1,\ldots,\tensor{A}_{r-2}) \mapsto\tensor{B}+\sum_{i=1}^{r-2}\tensor{A}_i.
	\]
	In the terminology of \cite{BV2017}, $\Var N_r^*$ is the \emph{join} of $\Var N_{2;n_1,\ldots,n_d}$ and $r-2$ copies of $\Var S_{n_1,\ldots,n_d}$.
	
	Let $\tensor{B}\in \Var{N}_{2; n_1,\ldots,n_d}$.
	By construction, the tensor $\tensor{B}$ has multilinear rank bounded by $(2,\ldots,2)$, meaning that there exists a tensor subspace $\Var{V} = \Var{A}_1 \otimes \cdots \otimes \Var{A}_d \subset \R^{n_1} \otimes \cdots \otimes \R^{n_d}$ where the linear subspace $\Var{A}_k \subset \R^{n_k}$ has $\dim \Var{A}_k = 2$ and such that $\tensor{A} \in \Var V$.
	For each $1\leq k\leq d$ we denote the orthogonal complement of $\Var{A}_k$ in $\mathbb{R}^{n_k}$ by $\Var{A}_k^\perp$. Let us define $\Var{W} := \Var{A}_1^\perp \otimes \cdots \otimes \Var{A}_d^\perp$.
	
	Now, we make the following choice of rank-$1$ tensors
		$\tensor{A}_1,\ldots,\tensor{A}_{r-2} \in \Var{S}_{n_1,\ldots,n_d} \cap \Var{W}$: for $1\leq k \leq d$ let $U_k \in \R^{\Pi \times n_k-2}$ be a matrix whose columns form an orthonormal basis of $\Var{A}_k^\perp$. Then, by \cref{lemmaB2}, the map
		\[
		U: \Var{S}_{n_1-2,\ldots,n_d-2}  \to \Var{S}_{n_1,\ldots,n_d} \cap \Var{W},\; \vect{a}^1\otimes \cdots\otimes \vect{a}^d \mapsto (U_1 \vect{a}^1)\otimes \cdots\otimes (U_d \vect{a}^d)
		\]
	is an isometry of manifolds. For $1\leq i\leq r-2$ we define $\tensor{A}_i:=U(\tensor{X}_i)$, where $\tensor{X}_i$ are the rank-$1$ tensors from above. {We write $\tensor{A}_i = \nu_i \vect a^1_i \otimes \cdots\otimes \vect a^d_i$. Because $U$ is an isometry, $\Vert \tensor{X}_i\Vert = \Vert \tensor{A}_i\Vert = \nu_i$, so we can choose $\Vert \vect a^1_i \Vert = \cdots =\Vert \vect a^d_i\Vert = 1$.} The plan for the rest of the proof is to show that $\mathrm{Jac}(\phi)(\tensor{B},\tensor{A}_1,\ldots,\tensor{A}_{r-2})$ is a finite value which does not depend on $\tensor{B}$, and that there is a neighborhood around $(\tensor{A}_1,\ldots,\tensor{A}_{r-2})$, whose size is also independent of $\tensor{B}$, on which the Jacobian does not deviate too much.
	
	For showing this, we first compute the derivative of $\phi$ at the point $(\tensor{B},\tensor{A}_1,\ldots,\tensor{A}_{r-2})$:
			\begin{align*}
			\deriv{\phi}{(\tensor{B},\tensor{A}_1,\ldots,\tensor{A}_{r-2})} : \; \Tang{\tensor{B}}{\Var N_{2;n_1,\ldots,n_d}} \times \Tang{\tensor{A}_1}{\Var S_{n_1,\ldots,n_d}}\times \cdots \times \Tang{\tensor{A}_{r-2}}{\Var S_{n_1,\ldots,n_d}} &\to \Tang{\tensor{A}}{\Var N^*_r},\\
				(\dot{\tensor{B}},\dot{\tensor{A}}_1,\ldots,\dot{\tensor{A}}_{r-2}) &\mapsto \dot{\tensor{B}}+\sum_{i=1}^{r-2}\dot{\tensor{A}}_i,
			\end{align*}
	where $\tensor{A} = \tensor{B}+\sum_{i=1}^{r-2}\tensor{A}_i$.
	Let $V\in \R^{\Pi\times 2\Sigma}$ be a matrix whose columns form an orthonormal basis of $\Tang{\tensor{B}}{\Var N_{2;n_1,\ldots,n_d}}$ and, for $1\leq i\leq r-2$ let $W_i\in \R^{\Pi\times \Sigma}$ be a matrix whose columns form an orthonormal basis of $\Tang{\tensor{A}_i}{\Var S_{n_1,\ldots,n_d}}$. Then, we have
			\begin{equation}\label{eqB1}
				\mathrm{Jac}(\phi)(\tensor{B},\tensor{A}_1,\ldots,\tensor{A}_{r-2}) = \mathrm{vol}(\begin{bmatrix} V & W_1&\cdots&W_{r-2} \end{bmatrix}).
			\end{equation}
	Let us write $\tensor{B}=\tensor{B}_1+\tensor{B}_2$ with $\tensor{B}_1,\tensor{B}_2\in \Var S_{n_1,\ldots,n_d}$. From \cite[Corollary 2.2]{BL2013} we know that $\tensor{B}_1, \tensor{B}_2\in \Var V$. Moreover, by \cref{tangent_space_N},
	$\Tang{\tensor{B}}{\Var N_2} = \Tang{\tensor{B}_1}{\Var S_{n_1,\ldots,n_d}} + \Tang{\tensor{B}_{2}}{\Var S_{n_1,\ldots,n_d}}$. Since, $\tensor{A}_1,\ldots,\tensor{A}_{r-2}$ are, by assumption, elements of $\Var W$,
	\cref{lemmaB1} implies that $\Tang{\tensor{B}}{\Var N_2}$ is orthogonal to $\Tang{\tensor{A}}{\Var S_{n_1,\ldots,n_d}}$ for all $1\leq i\leq r-2$. Therefore, we get the following equation for \cref{eqB1}:
	\[
	\mathrm{vol}(\begin{bmatrix} V & W_1&\cdots&W_{r-2}\end{bmatrix})
	= \mathrm{vol}(V)\,\mathrm{vol}(\begin{bmatrix} W_1&\cdots&W_{r-2}\end{bmatrix})
	= \mathrm{vol}(\begin{bmatrix} W_1&\cdots&W_{r-2}\end{bmatrix}),
	\]
	the last equality because $V$ has orthonormal columns. Let us further investigate the $W_i$.
	
	By \cref{tangent_space_segre}, the tangent space of $\Var S_{n_1,\ldots,n_d}$ at $\tensor{A}_i=\vect a^1_i \otimes \cdots\otimes \vect a^d_i$ is
	\[
	\Tang{\tensor{A}_i}{\Var{S}_{n_1,\ldots,n_d}} =
	\R^{n_1} \otimes \vect{a}_i^2 \otimes \cdots \otimes \vect{a}_i^d + \cdots +\vect{a}_i^1 \otimes \cdots \otimes \vect{a}_i^{d-1} \otimes \R^{n_d}.
	\]
	Moreover, by \cref{lemmaB2}, $\Var S_{n_1,\ldots,n_d} \cap\Var W$ is a manifold, and its tangent space at $\tensor{A}_i$ is
	\[
	\Tang{\tensor{A}_i}{(\Var S_{n_1,\ldots,n_d} \cap\Var W)} = \Var{A}_1^\perp \otimes \vect{a}_i^2 \otimes \cdots \otimes \vect{a}_i^d + \cdots +\vect{a}_i^1 \otimes \cdots \otimes \vect{a}_i^{d-1} \otimes \Var{A}_d^\perp.
	\]
	For all $1\leq k\leq d$ let $\{\vect{t}^k,\vect{s}^k\}$ be an orthonormal basis of $\Var{A}_k$, and let us write
	\begin{align*}
	\Var K_i&:= \mathrm{span}\{\vect{t}^1\otimes \vect{a}_i^2 \otimes \cdots \otimes \vect{a}_i^d, \ldots, \vect{a}_i^1 \otimes \cdots \otimes \vect{a}_i^{d-1} \otimes \vect{t}^d\}, \text{ and }\\
	\Var L_i&:= \mathrm{span}\{\vect{s}^1\otimes \vect{a}_i^2 \otimes \cdots \otimes \vect{a}_i^d, \ldots, \vect{a}_i^1 \otimes \cdots \otimes \vect{a}_i^{d-1} \otimes \vect{s}^d\}.
	\end{align*}
	Because $\vect{a}_i^k\in\Var A_k^\perp$ for $1\leq i\leq r-2$, and because $\Vert\vect a_i^j\Vert =1$, the tensors listed even form orthogonal bases for $\Var K_i$ and $\Var L_i$, respectively. Furthermore, we have for all pairs of indices $i,j$ that
	\[
	\Tang{\tensor{A}_i}{(\Var S_{n_1,\ldots,n_d} \cap \Var W)}   \perp \Var K_j,\quad \Tang{\tensor{A}_i}{(\Var S_{n_1,\ldots,n_d} \cap \Var W)}   \perp \Var L_j,\quad \text{and}\quad \Var K_i  \perp \Var L_j.
	\]
	Therefore, we have the following orthogonal decomposition:
	\[
	\Tang{\tensor{A}_i}{\Var S_{n_1,\ldots,n_d}} = \Tang{\tensor{A}_i}{(\Var S_{n_1,\ldots,n_d} \cap \Var W)} \oplus  \Var K_i \oplus \Var L_{i}, \text{ for } 1\leq i\leq r-2.
	\]
	The columns of $UX_i$ form an orthonormal basis of $\Tang{\tensor{A}_i}{(\Var{S}_{n_1,\ldots,n_d} \cap \Var{W})}$. Altogether, we have
	{\small\begin{align*}
	&\mathrm{vol}\left(\begin{bmatrix} W_1&\cdots&W_{r-2}\end{bmatrix}\right)\\
	 &= \mathrm{vol}\left(\begin{bmatrix} UX_1&\cdots&UX_{r-2}\end{bmatrix}\right) \prod \mathrm{vol}\left(\begin{bmatrix}\vect{a}_1^{h_1} \otimes \cdots \otimes \vect a_1^{h_{d-1}}& \cdots & \vect{a}_{r-2}^{h_1} \otimes \cdots \otimes \vect a_{r-2}^{h_{d-1}}\end{bmatrix}\right)^2\\
	 &= \mathrm{vol}\left(\begin{bmatrix} X_1&\cdots&X_{r-2}\end{bmatrix}\right) \prod \mathrm{vol}\left(\begin{bmatrix}\vect{x}_1^{h_1} \otimes \cdots \otimes \vect x_1^{h_{d-1}}& \cdots & \vect{x}_{r-2}^{h_1} \otimes \cdots \otimes \vect x_{r-2}^{h_{d-1}}\end{bmatrix}\right)^2,
	\end{align*}}%
	where both products range for $1\leq h_1<\cdots<h_{d-1}\leq d$. This implies that\enlargethispage{\baselineskip}
	\[
	\mathrm{vol}(\begin{bmatrix} W_1&\cdots&W_{r-2}\end{bmatrix}) =2 \mu
	\]
	is \emph{independent of $\tensor{B}$}.
	
	The rest of the proof is a variational argument:
	Let us denote by $\mathrm{Gr}(\Sigma,\Pi)$ the Grassmann manifold of $\Sigma$-dimensional linear spaces in~$\R^\Pi$. We endow $\mathrm{Gr}(\Sigma,\Pi)$ with the standard Riemannian metric, such that the distance between two spaces is the Euclidean length of the vector of \emph{principal angles} \cite{BG1973}. Let us denote this distance by $d(\cdot,\cdot)$.
	Furthermore, let $G: \Var S_{n_1,\ldots,n_d}\to \mathrm{Gr}(\Sigma,\Pi), \tensor{A}\mapsto \mathrm{T}_\tensor{A}\,\Var S_{n_1,\ldots,n_d}$ be the  \emph{Gauss map}.
	
	From \cite[Proposition 4.3]{BV2018c} we get for each $1\leq i\leq r-2$ that $\Vert\deriv{G}{\tensor{A}_{i}}  \Vert \leq \sqrt{\Sigma}$. This means, that for $\epsilon>0$ and any tuple $(\tensor{A}_1',\ldots,\tensor{A}_{r-2}') \in \Var S_{n_1,\ldots,n_d}$ satisfying $\Vert \tensor{A}_i - \tensor{A}_i'\Vert <\epsilon$ for $1\leq i\leq r-2$,
	we have $d(\Tang{\tensor{A}_i}{\Var S_{n_1,\ldots,n_d}}, \Tang{\tensor{A}_i'}{\Var S_{n_1,\ldots,n_d}})<2\epsilon \sqrt{\Sigma}$ for sufficiently small $\epsilon$. Let $W_i'\in \R^{\Pi\times \Sigma}$ be a matrix with orthonormal columns that span $\Tang{\tensor{A}_i'}{\Var S_{n_1,\ldots,n_d}}$. Consequently, there is a constant $K>0$, which depends only on $n_1,\ldots,n_d$, such that
	$\Vert W_i - W_i'\Vert <\epsilon K\sqrt{\Sigma}$.
	
	Recall that the columns of $\left[\begin{smallmatrix} W_1 & \cdots & W_{r-2}\end{smallmatrix}\right]$ are orthogonal to the columns of $V$. Moreover, $r\Sigma \leq \Pi$, since we have assumed that $\sigma_{r;n_1,\ldots,n_d}$ is generically complex identifiable.  Hence, there is a matrix $M\in \R^{\Pi\times \Pi}$ with $MV=V$ and $MW_i=W_i'$ for $1\leq i\leq r-2$ that leaves the orthogonal complement of $V+W_1+\cdots+W_{r-2}$ fixed.
	This implies
		$$ \Vert I_\Pi-M\Vert_2=\Vert (I_\Pi-M)\left[\begin{smallmatrix} W_1 & \cdots & W_{r-2}\end{smallmatrix}\right]\Vert_2 \leq\sum_{i=1}^{r-2}\Vert W_i-W_i'\Vert < \epsilon rK\sqrt{\Sigma},$$
	where $I_\Pi$ is the $\Pi\times \Pi$ identity matrix.
	Therefore, if we choose $\epsilon$ small enough, we may assume $\det M > \tfrac{1}{2}$. Note that such a choice of $\epsilon$ is independent of $\tensor{B}$.

	Altogether, we have shown that for all tuples $(\tensor{A}_1',\ldots,\tensor{A}_{r-2}') \in \Var S_{n_1,\ldots,n_d}$ satisfying $\Vert \tensor{A}_i - \tensor{A}_i'\Vert <\epsilon$ we have that
	\begin{align*}
		\mathrm{Jac}(\phi)(\tensor{B},\tensor{A}_1',\ldots,\tensor{A}_{r-2}')& = \mathrm{vol}(\begin{bmatrix} V & W_1'&\cdots&W_{r-2}'\end{bmatrix})\\
		& = \mathrm{vol}(\begin{bmatrix} MV & MW_1&\cdots& MW_{r-2} \end{bmatrix})\\
		& = \vert \det(M)\vert \, \mathrm{vol}(\begin{bmatrix} V & W_1&\cdots& W_{r-2} \end{bmatrix})\\
		& > \frac{1}{2}\,\mathrm{vol}(\begin{bmatrix} V & W_1&\cdots& W_{r-2} \end{bmatrix})\\
			& = \mu,
	\end{align*}
	and both $\epsilon$ and $\mu$ have been chosen independent of $\tensor{B}$. This concludes the proof.
	\end{proof}

	\section{Proofs of the Lemmata in Section \ref{sec:angular_CN}} \label{app_angular}
	
	\subsection{Proof of Lemma \ref{lem:inner}}\label{proof_lem:inner}
	The proof is similar to the proof of \cref{lem:innerBV}. Integrating in polar coordinates, we have
		\[
		J_\mathrm{inner} = \int_{0}^{\frac{\pi}{2}}\int_0^\infty \rho\, q\left((\Id-MM^\dagger)\begin{bmatrix}\rho \cos(\theta) L_1 & \rho \sin(\theta)L_2\end{bmatrix}\right)\;e^{-\frac{\rho^2\|\cos(\theta)\tensor{U}+\sin(\theta)\tensor{V}\|^2}{2}}\,\d{}\rho \, \d{}\theta.
		\]
	Note that the argument of $q$ is a $\Pi \times (2\Sigma-2)$ matrix. Since $q(A) = \varsigma_1(A) \cdots \varsigma_{n-1}(A)$ for $A \in \R^{m \times n}$ with $n\le m$, we have
	\[
	q\left((\Id-M M^\dagger)\begin{bmatrix} \rho \cos(\theta)L_1 & \rho \sin(\theta)L_2\end{bmatrix}\right)
	= \rho^{2\Sigma-3}\, q\left((\Id-MM^\dagger)\begin{bmatrix} \cos(\theta)L_1 &  \sin(\theta)L_2\end{bmatrix}\right).
	\]
	This yields
	\[
			J_\mathrm{inner}=\int_{0}^{\frac{\pi}{2}} q\left((\Id-MM^\dagger)\begin{bmatrix} \cos(\theta)L_1 &\sin(\theta)L_2\end{bmatrix}\right) \int_0^\infty \rho^{2\Sigma-2} \;e^{-\frac{\rho^2\|\cos(\theta)\tensor{U}+\sin(\theta)\tensor{V}\|^2}{2}}\d{}\rho\d{}\theta.
			\]
		The change of variables $t=\rho\|\cos(\theta)\tensor{U}+\sin(\theta)\tensor{V}\|$ transforms the integral for $\rho$ into
		\[
		\frac{1}{\|\cos(\theta)\tensor{U}+\sin(\theta)\tensor{V}\|^{2\Sigma-1}} \int_0^\infty t^{2\Sigma-2} \;e^{-\frac{t^2}{2}}\d{}t
		= \frac{2^{\frac{2\Sigma-3}{2}}\Gamma\left(\frac{2\Sigma-1}{2}\right)}{\|\cos(\theta)\tensor{U}+\sin(\theta)\tensor{V}\|^{2\Sigma-1}}.
		\]
	Plugging the foregoing into the expression for $J_\text{inner}$ concludes the proof.\qed
	
	\subsection{Proof of Lemma \ref{lem:cotaq}}\label{proof_lem:cotaq}
	
	First, note that the left-hand term in \eqref{eqcotaq} is bounded above by a constant depending on $n_1,\ldots,n_d,d$. Thus, by choosing an appropriate constant $K$ in \eqref{eqcotaq} we can assume that $\|\tensor{U}-\tensor{V}\|$ is smaller than any predefined quantity. We thus assume from now on that $\|\tensor{U}-\tensor{V}\|\leq\epsilon$ for some $\epsilon>0$ that can be chosen as small as desired. {Furthermore, as in \cref{eqn_proof_trigonio}, we write
	\begin{align*}
	\delta_k := \langle \vect{u}^k,\vect{v}^k\rangle = \sqrt{\tfrac{1}{\epsilon_k^2+1}}, \quad
	\epsilon_k := \Vert \vect{u}^k - \vect{v}^k\Vert = \sqrt{\tfrac{1}{\delta_k^2}-1}, \quad\text{and}\quad
	z := \delta_1\cdots \delta_d.
	\end{align*}
	We also assume that $\delta_1 = \min\{\delta_1,\ldots,\delta_d\}$, or, equivalently, $\epsilon_1 = \max\{\epsilon_1,\ldots,\epsilon_d\}$. Note that, if $\epsilon \approx 0$, then {one can assume} $\epsilon_k \approx 0$ and $\delta_k\approx 1$ for all $1\leq k\leq d$. In particular, all inequalities from the proof of \cref{sec3:auxiliary_lemma} in \cref{proof_lem3} are still valid here.}
	
	\subsubsection*{Dropping the scaling}
	For simplifying notation, we abbreviate
	\[
	A = (\Id-MM^\dagger)\begin{bmatrix} L_1 &  L_2\end{bmatrix} \text{ and } B = \mathrm{diag}(\underbrace{\cos(\theta), \ldots, \cos(\theta)}_{(\Sigma -1) \text{ times}}, \underbrace{\sin(\theta), \ldots, \sin(\theta)}_{(\Sigma -1) \text{ times}}),
	\]
	so that $AB = (\Id-MM^\dagger)\begin{bmatrix} \cos(\theta) L_1 & \sin(\theta) L_2\end{bmatrix}$. Observe that, by definition, $A$ ultimately depends on $\tuple{u}$ and $\tuple{v}$; that is, $A=A(\tuple{u},\tuple{v})$.
	
	Recall from \cref{def_q} that $q(\cdot)$ is the product of all but the smallest singular value of its argument. We first show that $q(AB)\leq q(A)$. To see this, let $\varsigma_1 \ge \cdots \ge \varsigma_{2(\Sigma -1)} \ge 0$ be the singular values of $A$. Then,
	\[
	\det((AB)^T (AB)) = \det(BB^T) \det(A^TA) = (\varsigma_1\cdots \varsigma_{2(\Sigma -1)})^2 (\cos(\theta)\sin(\theta))^{2(\Sigma -1)}.
	\]
	This shows
	\[
	q(AB) = \frac{\varsigma_1\cdots \varsigma_{2(\Sigma -1)} \cdot  (\cos(\theta)\sin(\theta))^{\Sigma -1}}{\varsigma_{\min}(AB)},
	\]
	where $\varsigma_{\min}(\cdot)$ denotes the smallest singular value; see \cref{def_spectral_norm}. Next, we have
	\begin{align*}
		 \varsigma_{\min}(AB)
		 = \min_{\vect{x}\in \Sp(\R^{2(\Sigma -1)})}\, \Vert AB \vect{x}\Vert
		 &\geq \min_{\vect{x}\in \Sp(\R^{2(\Sigma -1)})}\, \Vert A \vect{x}\Vert \quad \cdot \min_{\vect{x}\in \Sp(\R^{2(\Sigma -1)})}\, \Vert B \vect{x}\Vert\\
		 &= \varsigma_{2(\Sigma -1)} \cdot \min\{|\cos(\theta)|, |\sin(\theta)|\}.
		 \end{align*}
	Moreover, for $0\leq \theta  \leq \tfrac{\pi}{2}$ we have $0\leq \cos(\theta)\leq 1$ and $0\leq \sin(\theta)\leq 1$. Altogether, this implies $q(AB)\leq \varsigma_1\cdots \varsigma_{2(\Sigma -1)-1} = q(A)$. In the rest of the proof, we bound the $\varsigma_i$'s.
	
	\subsubsection*{Simplifying the matrix by orthogonal transformations}
	Recall from the proof of \cref{sec3:auxiliary_lemma} that applying the orthogonal transformation in \cref{def_R_i} on the right, we have
	\[
	\varsigma_i = \varsigma_i \bigl( (I-MM^\dagger) \begin{bmatrix} L_1 & L_2 \end{bmatrix} \bigr)
	 = \varsigma_i \bigl( (I-MM^\dagger) \begin{bmatrix} R_\uparrow & R_\downarrow \end{bmatrix} \bigr), \quad i=1,\ldots,2(\Sigma-1),
	\]
	where $\varsigma_i(A)$ denotes the $i$th largest singular value of the matrix $A$.
	Recalling \cref{eqn_def_S_and_T}, we have
	\[
	 \varsigma_i = \varsigma_i\bigl( (I-P) \begin{bmatrix} S_\uparrow & S_\downarrow & T_\uparrow & T_\downarrow \end{bmatrix} \bigr),
	\]
	{where $P = M M^\dagger$. Let $\vect{a}_\uparrow$ and $\vect{a}_\downarrow$ be as in \cref{eqn_rotated_UV}. The matrix $M M^\dagger$ projects orthogonally onto the span of $\tensor{U}$ and $\tensor{V}$, which coincides with the span of the orthonormal vectors $\frac{\vect{a}_\uparrow}{\|\vect{a}_\uparrow\|}$ and $\frac{\vect{a}_\downarrow}{\|\vect{a}_\downarrow\|}$. Moreover, by \cref{eqn_innerprods_a_updown} we have $\vect{a}_\downarrow^T \vect{a}_\downarrow = 1+z$ and  $\vect{a}_\uparrow^T \vect{a}_\uparrow = 1-z$. This shows that
	\[
	 P = \frac{1}{1+z} \vect{a}_\downarrow \vect{a}_\downarrow^T + \frac{1}{1-z} \vect{a}_\uparrow \vect{a}_\uparrow^T.
	\]
	}Next, it follows from \cref{eqn_everything} that $ P S_\uparrow =0$ and $ P S_\downarrow =0$, so that $\varsigma_i=\varsigma_i( \widetilde{N} )$, where
	\[
	 \widetilde{N} := \begin{bmatrix} S_\uparrow & S_\downarrow & (I-P)T_\uparrow & (I-P)T_\downarrow \end{bmatrix}.
	\]
	
	\subsubsection*{Computing the Gram matrix}
	Next, we compute the Gram matrix of $\widetilde{N}$. Consider again \cref{eqn_everything}, from which all of the following computations follow.
	We have
	\[
	 \begin{bmatrix}
	S_\uparrow^T S_\uparrow & S_\uparrow^T S_\downarrow \\
	S_\downarrow^T S_\uparrow & S_\downarrow^T S_\downarrow
	 \end{bmatrix} =
	 \begin{bmatrix}
	  F_\uparrow & 0 \\ 0 & F_\downarrow
	 \end{bmatrix},
	\]
	where the $F$'s are {the diagonal matrices} from \cref{eqn_def_EandF}.
	{From the symmetry} of $P$ we obtain
	\(S_\downarrow^T P = 0\) and \(S_\uparrow^T P=0,\)
	so that
	\begin{align*}
	S_\uparrow^T (I - P) T_\uparrow
	&= S_\uparrow^T T_\uparrow - 0 = 0, &
	S_\downarrow^T (I - P) T_\uparrow
	&= S_\downarrow^T T_\uparrow - 0 = 0, \\
	S_\uparrow^T (I - P) T_\downarrow
	&= S_\uparrow^T T_\downarrow - 0 = 0, &
	S_\downarrow^T (I - P) T_\downarrow
	&= S_\downarrow^T T_\downarrow - 0 = 0.
	\end{align*}
	We also find
	\begin{align*}
	 T_\uparrow^T (I - P) T_\downarrow
	 = &T_\uparrow^T T_\downarrow - T_\uparrow^T \left( \frac{1}{1+z} \vect{a}_\downarrow \vect{a}_\downarrow^T + \frac{1}{1-z} \vect{a}_\uparrow \vect{a}_\uparrow^T \right) T_\downarrow
	 \\=& 0 - \left( 0 + \frac{1}{1-z} T_\uparrow^T \vect{a}_\uparrow \vect{a}_\uparrow^T \right) T_\downarrow
	 \\=& 0.
	 \end{align*}
	We observe
	\begin{align*}
	 T_\uparrow^T (I - P) T_\uparrow
	 &= T_\uparrow^T T_\uparrow - T_\uparrow^T \left( \frac{1}{1+z} \vect{a}_\downarrow \vect{a}_\downarrow^T + \frac{1}{1-z} \vect{a}_\uparrow \vect{a}_\uparrow^T \right) T_\uparrow \\
	 &= (E_\uparrow - z \vect{g}\vect{g}^T) - 0 - \frac{1}{1-z} (-z \vect{g})(-z \vect{g})^T \\
	 &= E_\uparrow - \frac{z}{1-z} \vect{g}\vect{g}^T.
	\end{align*}
	Analogously we find
	\[
	 T_\downarrow^T (I - P) T_\downarrow
	 = (E_\downarrow + z\vect{g}\vect{g}^T) - \frac{1}{1+z} (z \vect{g})(z\vect{g})^T
	 = E_\downarrow + \frac{z}{1 + z} \vect{g} \vect{g}^T.
	\]
	Combining all of the foregoing observations, results in
	\[
	 \widetilde{N}^T \widetilde{N} =
	 \begin{bmatrix}
	  F_\uparrow & 0 & 0 & 0 \\
	  0 & F_\downarrow & 0 & 0 \\
	  0 & 0 & E_\uparrow - \frac{z}{1-z} \vect{g}\vect{g}^T & 0 \\
	  0 & 0 & 0 & E_\downarrow + \frac{z}{1+z} \vect{g}\vect{g}^T
	 \end{bmatrix},
	\]
	so that the singular values of $\widetilde{N}$ are given by the square roots of the eigenvalues of the matrices on the block diagonal.
	
	\subsubsection*{Bounding the singular values}
	In the remainder, let $\lambda_i(A)$ denote the $i$th largest eigenvalue of the positive semidefinite matrix $A$. Since $0 < \delta_k \le 1$ for all $1\leq k\leq d$, the eigenvalues $1+z\delta_k^{-1} = 1 + \prod_{j\ne k}\delta_j$ for $1\leq k\leq d,$ of the diagonal matrix $F_\uparrow$ satisfy
	\begin{align} \label{eqn_bound_piece_1}
	\lambda_i( F_\uparrow ) \le 2, \quad i=1,\ldots, \Sigma-d-1.
	\end{align}
	An upper bound for the eigenvalues of $F_\downarrow$ is given by
	\[
	 1 - z\delta_k^{-1} = 1 - \prod_{1 \le j \ne k \le d} \delta_j
	 \le 1 - \delta_1^{d-1}= (1-\delta_1)(1+\delta_1+\cdots+\delta_1^{d-2})\leq d(1-\delta_1),
	\]
	since $0 < \delta_1 \le 1$ for sufficiently small $\epsilon$. Hence, we obtain the bound
	\begin{align} \label{eqn_bound_piece_2}
	 \lambda_i( F_\downarrow ) \le d (1 - \delta_1), \quad i=1,\ldots,\Sigma-d-1.
	\end{align}
	
	The eigenvalues of $E_\downarrow + \frac{z}{1+z} \vect{g}\vect{g}^T$ can be bounded by using that the spectral norm of the rank-1 term is bounded by $\|\vect{g}\|^2 = \sum_{k=1}^d \epsilon_k^2 \le d \epsilon_1^2$. It follows from Weyl's perturbation inequality; see, e.g., \cite[Corollary~7.3.8]{HJ1990}, and the positive semidefiniteness of $E_\downarrow + \frac{z}{1+z} \vect{g}\vect{g}^T$ that
	\[
	 \lambda_i \left( E_\downarrow + \frac{z}{1+z} \vect{g}\vect{g}^T \right) \le \lambda_i ( E_\downarrow ) + d \epsilon_1^2.
	\]
	{From the definition of $\delta_1$ and $\epsilon_1$}, we have
	\begin{align} \label{eqn_epsilon_upper_bound}
	 {\epsilon_1^2 \leq \frac32\|\vect u^1-\vect v^1\|^2} \le 3(1-\delta_1),
	\end{align}
	provided that $\epsilon$ is sufficiently small.
	The eigenvalues of the diagonal matrix $E_\downarrow$ are bounded from above by $C (d+1) \epsilon_1^2$ due to \cref{newboundU}. Putting all of these together and using $d \ge 1$, we find
	\begin{align} \label{eqn_bound_piece_3}
	\lambda_i\left( E_\downarrow + \frac{z}{1+z} \vect{g}\vect{g}^T \right) \le C' d (1 - \delta_1), \quad i = 1,\ldots, d,
	\end{align}
	for some universal constant $C'>0$.
	
	{For computing an upper bound on the eigenvalues of $E_\uparrow - \frac{z}{1-z} \vect{g}\vect{g}^T$, we start by noting that
	\[
	E_\uparrow - \frac{z}{1-z} \vect{g}\vect{g}^T
	= - \frac{z}{1-z} \vect{g}\vect{g}^T + 2 \Id_{d} + \mathrm{diag}\bigl( z(1+\epsilon_1^2)-1, \ldots, z(1+\epsilon_d^2)-1 \bigr);
	\]
	the spectral norm of the last diagonal matrix is bounded by $C (d+1) \epsilon_1^2$ because of \cref{newboundU}. Adding the matrix $2\Id_d$ causes all eigenvalues to be shifted by $2$; hence, it suffices to compute the nonzero eigenvalue of the rank-$1$ matrix. Its eigenvalues are
	\[
	 \underbrace{0, \ldots, 0}_{d-1 \text{ times}}, \frac{-z}{1-z} \|\vect{g}\|^2;
	\]
	the eigenvector corresponding to the last eigenvalue is $\frac{\vect{g}}{\|\vect{g}\|}$. In order to bound that eigenvalue, note that from \eqref{newboundL} and \eqref{newboundU} we have for some constant $c=c(d)>0$:
	\begin{align*}
	-z \frac{\|\vect{g}\|^2}{1 - z}
	&= \frac{\sum_{k=1}^d \epsilon_k^2}{1-\prod_{k=1}^d\sqrt{\epsilon_k^2+1}} \leq -2+c\epsilon_1^2,
	\end{align*}
	where the last inequality is proved as follows. Note that $\sqrt{1+t^2}\leq 1+\frac{t^2}{2}$, which can be verified by squaring the terms and comparing. Then,
	\[
	\frac{\sum_{k=1}^d \epsilon_k^2}{-1+\prod_{k=1}^d\sqrt{\epsilon_k^2+1}} \geq\frac{\sum_{k=1}^d \epsilon_k^2}{-1+\prod_{k=1}^d\left(1+\frac{\epsilon_k^2}{2}\right)} =\frac{\sum_{k=1}^d \epsilon_k^2}{\frac12\sum_{k=1}^d \epsilon_k^2+p},
	\]
	where $p=p(\epsilon_1,\ldots,\epsilon_d)$ is a polynomial expression in the $\epsilon_i$ of degree and coefficients bounded by a constant depending only on $d$, with all its monomials of degree at least $4$ in $\epsilon_1,\ldots,\epsilon_d$. Hence, for some constant $c=c(d)$ we have $|p|\leq c(d)\epsilon_1^4$ and we conclude that
	\[
	\frac{\sum_{k=1}^d \epsilon_k^2}{-1+\prod_{k=1}^d\sqrt{\epsilon_k^2+1}} \geq\frac{2}{1+\frac{2c(d)\epsilon_1^4}{\sum_{k=1}^d \epsilon_k^2}}\geq 2-\frac{4c(d)\epsilon_1^4}{\sum_{k=1}^d \epsilon_k^2}\geq2-\hat c(d)\epsilon_1^2,
	\]
	for some new constant $\hat{c}(d)$; the last step follows from \cref{eqn_proof_epsilonbound}.
	
	Putting all of the foregoing together with \cref{eqn_epsilon_upper_bound}, we have thus shown that $E_\uparrow - \frac{z}{1-z} \vect{g}\vect{g}^T$ has eigenvalues satisfying
	\begin{align} \label{eqn_bound_piece_4}
	\lambda_d\left( E_\uparrow - \frac{z}{1-z} \vect{g}\vect{g}^T \right) &\le c'(d)(1 - \delta_1), \text{ and }\\ \nonumber
	\lambda_i\left( E_\uparrow - \frac{z}{1-z} \vect{g}\vect{g}^T \right) &\le 2 + c'(d) (1-\delta_1), \quad i=1,\ldots,d-1,
	\end{align}}
	where $c'(d)$ is again a constant depending only on $d$.
	
	In \cref{eqn_bound_piece_1,eqn_bound_piece_2,eqn_bound_piece_3,eqn_bound_piece_4}, we have shown that precisely $\Sigma-d-1 + d + 1 = \Sigma$ eigenvalues are smaller than some constant times $1 - \delta_1$, while the remaining eigenvalues are clustered near~$2$. It follows that there exists a constant $K$, depending only on $n_1, \ldots, n_d$ and~$d$, such that
	\begin{align*}
	 q\left( (I-MM^\dagger) \begin{bmatrix} \cos(\theta) L_1 & \sin(\theta) L_2 \end{bmatrix} \right)
	 &\le q\left( (I-MM^\dagger) \begin{bmatrix} L_1 & L_2 \end{bmatrix} \right)
	 \le K (1 - \delta_1)^{\frac{\Sigma - 1}{2}} \\&
	 = K \left( \frac{\|\vect{u}^1 - \vect{v}^1\|}{\sqrt2} \right)^{\Sigma-1}
	 \le K \left( \frac{\|\tensor{U} - \tensor{V}\|}{\sqrt2} \right)^{\Sigma-1},
	\end{align*}
	where the last step is by \cref{bound_tensor_norm}. This finishes the proof.
	\qed

	\subsection{Proof of Lemma \ref{lem:integralauxiliar}}\label{proof_lem:integralauxiliar}
	Let $J$ be the integral in question; i.e.,
	\[
	J = \int_{0}^{\frac{\pi}{2}} \frac{1}{\Vert\cos(\theta)\vect{x} - \sin(\theta)\vect{y}\Vert^{a}}\, \d{}\theta.
	\]
	Writing  $\Vert\cos(\theta)\vect{x} - \sin(\theta)\vect{y}\Vert = \sqrt{1-\sin(2\theta) \langle \vect{x},\vect{y}\rangle}$ and exploiting the symmetry of $\sin(\theta)$ around~$\tfrac{\pi}{4}$ we have
			\[
		J=2\int_{0}^{\frac{\pi}{4}}\frac{1}{\sqrt{1-\sin(2\theta) \langle \vect{x},\vect{y}\rangle}^{\,a}} \d{}\theta \leq  \int_0^1\frac{1}{\sqrt{1-(1-t)\langle \vect{x},\vect{y}\rangle}^{\,a}\sqrt{t}} \,\d{}t;
			\]
	the inequality is due to the change of variables $\sin(2\theta)=1-t$ and $\sqrt{2t-t^2}\geq\sqrt{t}$ for $|t|\leq1$. Let us write $h:= \langle \vect{x},\vect{y}\rangle$. We distinguish between two cases. In the case $h\leq \tfrac{1}{2}$, we can bound
	\[
	J\leq \sqrt{2}^{\,a} \int_0^1 t^{-\frac{1}{2}} \,\d{} t = \sqrt{2}^{\,a+2}  \leq \frac{\sqrt{2}^{\,3a}}{\sqrt{2}^{\,a-1} \sqrt{1-h}^{\, a-1}}=\frac{\sqrt{2}^{\,3a+1}}{\Vert \vect{x} - \vect{y}\Vert^{\, a-1}}.
	\]
	The second case is $h>\tfrac{1}{2}$:
	A new change of variables $t=\tfrac{1-h}{h} \,u$ yields
	{		\begin{align*}
			J \leq& \sqrt{\frac{1-h}{h}} \int_{0}^{\frac{h}{1-h}}\frac{1}{\sqrt{1-h+(1-h)u}^{\,a}\sqrt{u}}\,\d{}u \\=&  \frac{1}{\sqrt{h}} \, \frac{1}{\sqrt{1-h}^{\,a-1}} \int_{0}^{\frac{h}{1-h}}\frac{1}{\sqrt{1+u}^{\,a}\sqrt{u}}\,\d{}u \\
			\leq & \frac{1}{\sqrt{h}} \, \frac{1}{\sqrt{1-h}^{\,a-1}} \left(\int_{0}^{1}\frac{1}{\sqrt{1+u} \sqrt{u}}\,\d{}u +\int_{1}^{\infty}\frac{1}{u^{(a+1)/2}}\,\d{}u \right).
			\end{align*}
	The last integrals add to at most $2+2(a-1)^{-1}$, and we thus have proved for $h>\tfrac{1}{2}$:
			\[
			J\leq \frac{2(1+(a-1)^{-1})}{\sqrt{h}\sqrt{1-h}^{\,a-1}}
			\leq \frac{\sqrt{2}^{\,3}(1+(a-1)^{-1})}{\sqrt{1-\langle \vect{x},\vect{y}\rangle}^{\, a-1}}
			=\frac{\sqrt{2}^{\,a+2}(1+(a-1)^{-1})}{\|\vect{x}-\vect{y}\|^{a-1}}.
			\]
			The lemma is proved.\qed}
	
	\subsection{Proof of Lemma \ref{lem:cosenos}}\label{proof_lem:cosenos}
	We prove the lemma by induction. The first case $d=1$ reads $\cos(\theta_1) \leq 1-\tfrac{\theta_1^2}{7}$. In fact, in this case we have the stronger inequality $\cos(\theta) \leq 1-\tfrac{\theta^2}{4}$ as the following argument shows: Consider the map $f:[0,\tfrac{\pi}{2}] \to \R, \theta \mapsto 1 - \tfrac{\theta^2}{4} - \cos(\theta)$. We have $f(0)=0$ and $f(\tfrac{\pi}{2}) = 1 - \tfrac{\pi^2}{16}>0$. Moreover, $f'(\theta)=\sin(\theta)- \tfrac{\theta}{2}>0$ for $0\leq \theta \leq \tfrac{\pi}{2}$. This implies that we have $f\geq 0$ proving the case $d=1$. For general~$d$, note that by the induction hypothesis
				\begin{align*}
				\cos (\theta_1)\cdots\cos(\theta_d)&=\cos (\theta_1)\cdots\cos(\theta_{d-1})\cos(\theta_d)\\ &\leq\left(1-\frac{\theta_1^2+\cdots+\theta_{d-1}^2}{7(d-1)}\right)\cos(\theta_d)\\ &\leq\left(1-\frac{\theta_1^2+\cdots+\theta_{d-1}^2}{7d}\right)\left(1-\frac{\theta_d^2}{4}\right)\\
				&=1-\frac{\theta_1^2+\cdots+\theta_d^2}{7d}+\frac{\theta_1^2+\cdots+\theta_{d-1}^2-7d+4}{28d}\,\theta_d^2.
			\end{align*}
				Since $\theta_1^2+\cdots+\theta_d^2-7d+4\leq d\,\tfrac{\pi^2}{4}-7d+4\leq0$ for $d\geq2$, we conclude the proof of the lemma.\qed

\bibliographystyle{amsplain}
\bibliography{ECN}

\end{document}